\documentclass[a4paper]{article}
\usepackage[T1]{fontenc}
\usepackage[utf8]{inputenc}
\usepackage{lmodern,microtype}
\usepackage[margin=26.5mm]{geometry}
\usepackage{amsmath,amssymb,amsthm,mathtools,mathrsfs}
\usepackage{enumitem}
\usepackage{xcolor}
\usepackage[colorlinks=true,linkcolor=blue,citecolor=blue,urlcolor=blue]{hyperref}

\newcommand{\dd}[1]{\mathrm{#1}}
\newcommand{\norm}[1]{\Vert#1\Vert}

\def\R{\mathrm{I\kern-0.21emR}}
\def\N{\mathrm{I\kern-0.21emN}}
\def\Z{\mathbb{Z}}
\newcommand{\C} {\mathbb{C}}
\newcommand{\T}{\mathbb{T}}

\newcommand{\Tobs}{T_{\mathrm{obs}}}
\newcommand{\dsr}{d_{\mathrm{sR}}}
\newcommand{\re}{\operatorname{Re}}
\newcommand{\im}{\operatorname{Im}}
\newcommand{\supp}{\operatorname{supp}}
\newcommand{\A}{\mathcal{A}}

\newcommand{\E}{\mathcal{E}}
\newcommand{\obs}{\mathcal{O}_T(u)}

\renewcommand{\geq}{\geqslant}
\renewcommand{\leq}{\leqslant}

\newtheorem{thm}{Theorem}[section]
\newtheorem{prop}[thm]{Proposition}
\newtheorem{cor}[thm]{Corollary}
\newtheorem{lem}[thm]{Lemma}
\theoremstyle{definition}
\newtheorem{hyp}[thm]{Assumption}
\theoremstyle{remark}
\newtheorem{remark}[thm]{Remark}
\theoremstyle{plain}
\newtheorem{maintheorem}{Theorem}

\setlist[enumerate]{itemsep=2pt,topsep=3pt}
\allowdisplaybreaks[2]

\title{Minimal time for heat observability\\ on almost-Riemannian surfaces}

\author{Vincent Boulard\footnote{CERMICS, CNRS, \'Ecole nationale des ponts et chauss\'ees, Institut Polytechnique de Paris, Marne-la-Vall\'ee, France (\texttt{vincent.boulard@enpc.fr}).}
\and
Amaury Hayat\footnote{CERMICS, CNRS, \'Ecole nationale des ponts et chauss\'ees, Institut Polytechnique de Paris, Marne-la-Vall\'ee, France \& Korea Institute for Advanced Study, Seoul, South Korea (\texttt{amaury.hayat@enpc.fr}).}
\and
Emmanuel Tr\'elat\footnote{Sorbonne Universit\'e, Universit\'e Paris Cit\'e, CNRS, Inria, Laboratoire Jacques-Louis Lions, LJLL, F-75005 Paris, France (\texttt{emmanuel.trelat@sorbonne-universite.fr}).}}

\date{}

\begin{document}
\maketitle

\begin{abstract}
We determine the minimal observation time for the heat equation on a closed almost-Riemannian surface, for a class of observation sets leaving sufficiently thin collars around simple Baouendi-Grushin circles. The structure and density are analytic near these circles, and the coefficients may depend on both variables. The formula is expressed through a complex phase describing concentration near the singular set and extends the thresholds known for separated models. Regular unobserved regions away from the circles do not change this time, provided all other singularities are observed. For singular circles of constant bracket step greater than two, local analyticity implies failure of observability at every time for observation sets at positive distance. We also construct quasimodes concentrated near simple singular circles, with exponentially small errors. Their superpositions give a lower bound for the Schr\"odinger observation time.
\end{abstract}


\section{Introduction and main results}

\paragraph{Setting and previous results.}
Let $M$ be a compact connected smooth surface without boundary, endowed with a bracket-generating almost-Riemannian structure and a smooth positive density $\nu$. The structure consists of a rank-two Euclidean bundle $E\to M$ and a smooth morphism $f:E\to TM$, of rank two on a dense open set \cite{ABS08,BL13}. Let $\triangle \geq 0$ be the self-adjoint operator associated with the closure of the horizontal Dirichlet form
$\mathcal{E}(u) = \int_M\vert \nabla_Hu\vert ^2\,\dd{d}\nu$, for $u\in C^\infty(M)$.
In a local orthonormal horizontal frame, $\triangle = X_1^*X_1 + X_2^*X_2$, with adjoints in $L^2(M,\nu)$. The density is smooth across the singular set. In particular, considering instead the singular Riemannian volume would define a different Laplace--Beltrami operator \cite{BL13}. The bracket step is the smallest bracket length for which the horizontal fields and their iterated brackets span the whole tangent plane. We denote the associated sub-Riemannian distance by $\dsr$. For background on sub-Riemannian geometry, we refer to \cite{ABB20,BR96,Mon02}.

For a nonempty open set $\omega\subset M$, heat observability in time $T>0$ means that there is $C_T>0$ such that
\begin{equation}\label{intro:observability}
\Vert e^{-T\triangle}u_0\Vert_{L^2(M,\nu)}^2 \leq C_T\int_0^T\Vert e^{-t\triangle}u_0\Vert_{L^2(\omega,\nu)}^2\,\dd{d}t \qquad\forall u_0\in L^2(M,\nu).
\end{equation}
By duality, this is equivalent to null controllability of $\partial_tu + \triangle u = \mathbf1_\omega v$ with an $L^2$ control \cite{DR77}. Set $\Tobs(\omega) = \inf\{T>0 \,\mid\, \eqref{intro:observability}\text{ holds}\}$, with $\inf\varnothing = +\infty$. Observability at one time implies observability at every larger time. Note that the infimum does not determine the validity of the estimate at the critical time.

For elliptic operators, every nonempty open set observes the heat equation at every positive time, hence $\Tobs(\omega)=0$ \cite{FI96,Im95,LR95}. This fails for subelliptic operators despite hypoellipticity \cite{H67}. For the flat Baouendi-Grushin operator $\triangle_0 = -\partial_x^2 + x^2D_y^2$, with $D_y = -i\,\partial_y$, the minimal time on a Dirichlet rectangle with symmetric observed strips at distance $a$ from the singular line is $\Tobs=a^2/2$ \cite[Theorem~1.3]{BMM15}. This positive threshold reflects concentration near the singular line.

For such Baouendi-Grushin operators of the form $-\partial_x^2 - |x|^{2\gamma}\partial_y^2$, observed from a strip at positive distance from the singular line, different degeneracies were distinguished in \cite{BCG14}: weak ($\gamma\in(0,1)$), critical ($\gamma=1$) and strong ($\gamma>1$). Exact thresholds and variable normal coefficients are treated in \cite{BDE20,BMM15}. Nonrectangular observation sets reveal additional geometric obstructions \cite{DKR23,DK20,Ko17}. In particular, holomorphic combinations of ground states and Runge approximation detect obstructions that a single eigenstate cannot detect \cite{DKR23,Ko17}. Positive minimal times also occur for the Heisenberg heat equation \cite{BC17} and the Baouendi-Grushin sphere \cite{Tam22,Tam26}.

On almost-Riemannian surfaces, observation obstructions for locally separated principal coefficients and compatible densities are proved in \cite{Van26}. Removing separation of variables is an open question in \cite[Section~6]{Van26}. Partially analytic, nonseparated Baouendi-Grushin operators already appear in the tunneling and approximate heat controllability results of \cite[Section~1.2.4]{LL22}, with tangential independence near the observation set and the nondegenerate boundary. The estimates allow a terminal error and rely on quantitative unique continuation \cite{LL19}. A different concentration argument for the half-heat and flat Baouendi-Grushin equations is given in \cite{Lis25}.

Our two main results allow coefficients depending on both variables near the singular circles. At (bracket) step two, we determine the exact observation time for sufficiently thin exterior phase collars. At higher constant step, we prove an obstruction at every time for observation sets at positive distance from an entire singular circle.

\paragraph{The phase and the observation threshold.}
Fix disjoint two-sided embedded singular circles $Z_1,\ldots,Z_m$, consisting of simple Baouendi-Grushin points, i.e. points where the horizontal rank is one and the bracket step is two. Assume that the structure and density are real analytic near these circles. The periodic normal form \cite{ABS08,BL13}, recalled in Section~\ref{sec:geometry}, gives fixed tubes with
$$
X_1 = \partial_{x_i},\qquad X_2 = x_i e^{\varphi_i(x_i,y_i)}\partial_{y_i},\qquad \varphi_i(0,y_i) = 0,\qquad y_i\in\mathbb{T}_{\ell_i}.
$$
On smaller tubes, $\vert x_i\vert = \dsr((x_i,y_i),Z_i)$. The tangential normalization leading to this normal form determines the period $\ell_i$ and makes the lowest transverse heat decay rate $k+\mathrm{o}(k)$ at momentum $k$. This normalization changes the tangential coordinate, but preserves the original time variable of the heat equation.

The geometric ingredient is the complex phase $\Phi_i$, defined on the universal cover of the tube by
\begin{equation}\label{phase-definition}
(\Phi_i)_{x_i} = i x_i e^{\varphi_i} (\Phi_i)_{y_i},\qquad \Phi_i(0,y_i) = y_i,\qquad \Phi_i = \Theta_i + i\mathcal{A}_i.
\end{equation}
It solves the principal eikonal equation $(\Phi_i)_{x_i}^2 + x_i^2e^{2\varphi_i}(\Phi_i)_{y_i}^2 = 0$ and has periodic increment $\ell_i$. Its imaginary part measures transverse decay: $\vert e^{ik\Phi_i}\vert = e^{-k\mathcal{A}_i}$. Near $Z_i$,
$$
\mathcal{A}_i(x_i,y_i) = \frac{x_i^2}{2} + \frac{\partial_{x_i}\varphi_i(0,y_i)}{3}x_i^3 + \mathrm{O}(x_i^4).
$$
The density affects lower-order terms of the operator, but not this phase.

Lemma~\ref{phase-local} constructs real analytic phase coordinates $r_i = \operatorname{sgn}(x_i)\sqrt{2\mathcal{A}_i}$ and $\theta_i = \Theta_i$ through $Z_i$. They put the principal symbol in the real normal form $c_i(r_i,\theta_i)(\xi_{r_i}^2 + r_i^2\xi_{\theta_i}^2)$, with $c_i>0$ and $c_i(0,\theta_i) = 1$. The coefficient $c_i$ may still depend on both variables, so this normal form need not be separated. This reduction connects the geometry to the two analytic parts of the proof: the realization of holomorphic heat trajectories and the observation of coupled transverse degrees.

A \emph{phase fiber} means the curve $\Theta_i = \theta$, with $\theta$ fixed. Along this curve, $r_i$ runs from one side of the circle to the other, and $\mathcal{A}_i = r_i^2/2$. The two points being compared therefore have the same phase angle, which need not mean the same original coordinate $y_i$.

Fix disjoint phase tubes defined in their respective local coordinates around $Z_i$ by $\mathcal{C}_i = \{\vert r_i\vert <R_i\}$. For continuous functions $\rho_i^\pm:\mathbb{T}_{\ell_i}\to[0,+\infty)$, prescribe the observation inside these tubes by
$$
K_i = \{-\rho_i^-(\theta_i) \leq r_i \leq \rho_i^+(\theta_i)\},\qquad \omega\cap\mathcal{C}_i = \mathcal{C}_i\setminus K_i.
$$
Thus, in each tube, both exterior sides of the collar are observed. Outside the tubes, $\omega$ need not be the whole surface. Define
\begin{equation}\label{observation-threshold}
M_\Phi(\omega) = \max_{1\leq i\leq m}\max_{\theta\in\mathbb{T}_{\ell_i}}\frac{1}{2}\min(\rho_i^+(\theta)^2,\rho_i^-(\theta)^2).
\end{equation}
For a fixed phase angle, the first observation on the two sides has attenuation exponents $\rho_i^+(\theta)^2/2$ and $\rho_i^-(\theta)^2/2$. The smaller exponent gives the stronger observed signal, since observing either side is enough. The maximum then selects the angle, and the circle, with the largest such cost. Note that the side providing the stronger signal may change with the angle.

\paragraph{Exact observation time at step two.}
Our first main result identifies $M_\Phi(\omega)$ as the exact minimal time for sufficiently thin phase collars.

\begin{maintheorem}\label{thm:exact-time}
Let $M$ be a closed connected almost-Riemannian surface with a smooth positive density. Let $Z_1,\ldots,Z_m$ be disjoint two-sided embedded singular circles consisting of simple Baouendi-Grushin points, with analytic structure and density in fixed disjoint phase tubes $\mathcal{C}_i$. There exists $r_*>0$, depending only on these tubes and their analytic data, with the following property. Suppose that $\omega\subset M$ is open, contains every singular point outside these circles, and satisfies $\omega\cap\mathcal{C}_i = \mathcal{C}_i\setminus K_i$ for some continuous nonnegative profiles $\rho_{i}^{\pm}$ with $\max_i\max(\Vert\rho_i^+\Vert_\infty,\Vert\rho_i^-\Vert_\infty)<r_*$. 

Then observability fails for $0<T<M_\Phi(\omega)$ and holds for every $T>M_\Phi(\omega)$. In particular, the minimal time is $\Tobs(\omega) = M_\Phi(\omega)$.

\end{maintheorem}

The minimal time is therefore determined by the observation near the selected singular circles. Regular unobserved regions elsewhere change the observation constant, but add no positive time. No analyticity is required outside the fixed disjoint tubes. Note that, at the critical time $T = M_\Phi(\omega)$, the leading exponential rates balance, hence the estimates below do not decide observability.

\begin{remark}[Attenuation and heat dissipation]\label{rem:threshold-interpretation}
The threshold $M_\Phi(\omega)$ compares spatial attenuation with heat dissipation. At large positive tangential momentum $k$, the factor $e^{ik\Phi}$ has modulus $e^{-k\mathcal{A}}$, while the lowest transverse heat decay rate is $k+\mathrm{o}(k)$. Heat dissipation reduces the squared norm at time $T$ by a leading factor $e^{-2kT}$, whereas observation at the level $\mathcal{A} = A_0$ contributes $e^{-2kA_0}$. When $T<A_0$, their ratio grows exponentially. Subsection~\ref{sec:sharp-lower} makes this comparison rigorous by localizing near a phase angle where observation is weakest.
\end{remark}

Section~\ref{sec:consequences} relates our threshold to the Agmon-distance formulas previously obtained for separated models, computes the curvature correction for distance collars, and gives examples on closed surfaces.

\begin{remark}[Size of the admissible collars]\label{rem:collar-width}
The radius $r_*$ bounds the unobserved widths in the phase coordinate. Its smallness comes from the local analytic estimates in Propositions~\ref{local-lower} and~\ref{weighted-observation}, applied in fixed tubes with observed regular strips near their boundaries. Shrinking a tube around a prescribed collar does not remove these analytic restrictions and may remove the required observed strips (Subsection~\ref{surface-global}).
The same $r_*$ works for all continuous profiles within this width range and all $T>M_\Phi(\omega)$, including profiles with arbitrarily rapid variations. Only the observation constant may depend on the profiles and on $T-M_\Phi(\omega)$.
\end{remark}

The lower bound can also detect an unobserved arc of a singular circle.
\begin{cor}\label{cor:local-obstruction}
Let $Z$ be a two-sided embedded singular circle consisting of simple Baouendi-Grushin points, with analytic structure and density near the whole circle. If $Z\not\subset\overline\omega$, then $\Tobs(\omega)>0$.
\end{cor}
Indeed, choose continuous nonnegative profiles supported on a small arc around an unobserved point, with their gap disjoint from $\omega$. Extend the profiles by zero elsewhere and apply the lower bound to the enlarged observation set given by the whole exterior of this gap. Adding or deleting the remaining part of the circle changes no $L^2$ observation norm. Clearly, disjointness from $Z$ alone is insufficient, since $M\setminus Z$ observes at every positive time. In separated models, related obstructions are known when observation stays away from a transverse segment \cite[Theorem~1.3]{DKR23} or from a neighbourhood of one singular point \cite[Theorem~1.5(i)]{Van26}.

\paragraph{Bounds for more general observation sets.}
The local estimates also allow holes inside a thin collar. Suppose that the full annuli $d_i<\vert r_i\vert<R_i$ are observed, with $0<d_i<r_*$. Inside $\vert r_i\vert\leq d_i$, the open set $\omega$ is arbitrary. Define
\begin{align*}
m_i(\theta) &= \inf\{\mathcal{A}_i(p)\mid p\in\omega\cap\mathcal{C}_i,\ \Theta_i(p)=\theta\}, \\
\underline m_i(\theta) &= \liminf_{\theta'\to\theta}m_i(\theta'),\qquad
\mathfrak{m}(\omega)=\max_i\sup_\theta\underline m_i(\theta).
\end{align*}
The upper bound uses a common action level observed at every angle, on at least one side:
$$
\mathfrak{A}(\omega)=\max_i\inf\big\{A\in(0,R_i^2/2)\mid
\forall\theta,\ (\sqrt{2A},\theta)\in\omega\ \text{or}\ (-\sqrt{2A},\theta)\in\omega\big\}.
$$
The points in this formula are read in the corresponding phase tube. The observed annuli make both quantities finite, with $\mathfrak{m}(\omega)\leq\mathfrak{A}(\omega)\leq\max_i d_i^2/2$.

\begin{cor}\label{cor:general-bounds}
Keep the geometric and analytic hypotheses of Theorem~\ref{thm:exact-time} and the observation of every singular point outside $Z_1\cup\cdots\cup Z_m$. Replace the exterior-collar condition by the observed annuli above. Then
$$
\mathfrak{m}(\omega)\leq\Tobs(\omega)\leq\mathfrak{A}(\omega).
$$
For continuous exterior phase collars, both bounds equal $M_\Phi(\omega)$.
\end{cor}
The proof is given in Subsection~\ref{sec:general-observation}.

\paragraph{All-time obstruction at higher step.}
Our second main result shows that, at higher constant step, positive separation from an entire singular circle prevents observability at every time.

\begin{maintheorem}\label{thm:higher-obstruction}
Let $Z$ be a two-sided compact embedded connected component of the singular set, of horizontal rank one and constant bracket step $r \geq 3$. Assume that the structure and density are real analytic near $Z$. Then every nonempty open set $\omega$ satisfying $\dsr(\omega,Z)>0$ has $\Tobs(\omega) = +\infty$.
\end{maintheorem}

The dependence on the bracket step is discussed in \cite[Section~1.2.5]{BCG14}. The conjectures recorded in \cite[Section~5.1]{Let21} predict small-time obstructions at step two and all-time obstructions at higher step. Our results establish these obstructions for unobserved singular circles under the stated geometric and analytic assumptions, and determine an exact threshold at step two. These assumptions matter: on a rank-varying structure, an arbitrary nonempty unobserved interior is insufficient: bands containing the singular set may observe at every positive time (see \cite[Remark~1 and Appendix]{BCG14} and \cite[Theorem~2.8]{Tam26}).

The constant-step normal form gives $X_2 = x^{r - 1}e^\varphi\,\partial_y$. 
The proof constructs normalized initial states concentrated near the singular circle, with energy $O(k^{2/r})$ at tangential frequency $k$. An exponentially weighted energy estimate shows that their heat evolutions have exponentially small observed mass, of order $e^{-ck}$ up to polynomial factors. Since $2/r<1$, this decay cannot be compensated by any fixed heat time.
The separated case is covered by
\cite[Theorem~1.5(ii)]{Van26}. Here the principal coefficient and the
density may depend on both variables. Constant step along the component excludes isolated tangency points. 
In fact, we show that the all-time non-observability result also holds under the weaker assumption of analyticity only in the tangential variable (Theorem~\ref{dyn:higher-order}), and even under uniform tangential Gevrey regularity of order $1<s<r/2$ (Theorem~\ref{dyn:gevrey}). 

\paragraph{Periodic spectral quasimodes.}
At step two, Section~\ref{wkb:section} also constructs periodic spectral quasimodes with exponentially small errors and an explicit limiting mass on the singular circle (Theorem~\ref{thm:periodic-quasimodes}). They yield a sequence of eigenvalues with an asymptotic expansion (Corollary~\ref{cor:spectral-sequence}).

Their superpositions give Schr\"odinger solutions concentrating near a point moving along the singular circle at unit speed in the normalized coordinate $y$. The following corollary extends the obstruction in \cite{BS22} to coupled analytic coefficients.

For an open set $\omega\subset M$, Schr\"odinger observability in time $T>0$ means that
\begin{equation}\label{schrodinger:observability}
\norm{u_0}_{L^2(M,\nu)}^2\leq C_T\int_0^T\norm{e^{-it\triangle}u_0}_{L^2(\omega,\nu)}^2\,\dd{d}t\qquad\forall u_0\in L^2(M,\nu)
\end{equation}
for some $C_T>0$.

\begin{cor}\label{cor:schrodinger-obstruction}
Let $Z$ be a two-sided compact embedded connected singular circle of rank one and bracket step two, with analytic structure and density near the whole circle. If $I\subset Z\setminus\overline\omega$ is an arc of length $L$ in the normalized coordinate $y$, then \eqref{schrodinger:observability} fails for every $0<T\leq L$. If $Z\cap\overline\omega=\varnothing$, it fails for every $T>0$.
\end{cor}
The proof is given in Section~\ref{wkb:section}. For Baouendi-Grushin surfaces, concentration of eigenfunctions already implies failure at every time when observation has positive distance from the entire singular set \cite[Sections~7.5 and~7.7]{CHT22}.

\paragraph{Strategy of the proofs.}
The complex phase is the common starting point for the two bounds in Theorem~\ref{thm:exact-time}. For the obstruction, we construct approximate heat solutions which concentrate near a chosen phase angle, then compare them with exact solutions. Runge approximation selects the angle where observation is weakest.

For observability, the phase gives local normal forms which retain the coupling between variables. A Carleman estimate first controls an angularly regularized solution. Separating the slowly and rapidly decaying transverse components reduces this preliminary loss to the attenuation at observed phase levels. Heat decay compensates this loss for every time larger than $M_\Phi(\omega)$.

Theorem~\ref{thm:higher-obstruction} uses a separate energy estimate. At higher step, high tangential frequencies produce too little heat decay to compensate for their weak signal on the observation set.

\paragraph{Structure of the paper.}
Section~\ref{sec:consequences} develops examples and geometric applications of Theorem~\ref{thm:exact-time}. Section~\ref{sec:geometry} establishes the local normal forms. Sections~\ref{sec:proof-a} and~\ref{dyn:section} prove Theorems~\ref{thm:exact-time} and~\ref{thm:higher-obstruction}, respectively. Subsection~\ref{sec:general-observation} proves Corollary~\ref{cor:general-bounds}, and Subsection~\ref{dyn:extensions} treats weaker regularity. Section~\ref{wkb:section} constructs periodic spectral quasimodes and derives their spectral and Schr\"odinger observation consequences. The appendix collects the analytic estimates used in these proofs.

\section{Applications}\label{sec:consequences}

We first interpret the threshold on flat and separated models, then compute the effect of curvature for distance collars and give examples on closed surfaces. All applications of Theorem~\ref{thm:exact-time} use collars within its fixed smallness range.

\subsection{Flat and separated models}\label{sec:model-comparison}

\paragraph{The flat model.}
For the flat Baouendi-Grushin operator $\triangle_0$, the phase is $\Phi = y + ix^2/2$, so phase fibers are the ordinary transverse lines $y = \theta$. Here ``flat'' denotes the standard normal model; its regular Riemannian metric has curvature $-2/x^2$. On $\R\times\T_\ell$, the normalized ground states $v_k = c_ke^{iky-kx^2/2}$ have eigenvalue $k$, for $k\in(2\pi/\ell)\N^*$. Their terminal squared mass is $e^{-2kT}$, whereas observation in $\{\vert x\vert>a\}$ has exponential rate $e^{-ka^2}$, giving the obstruction $T \geq a^2/2$.

For constant edges $-a_- \leq x \leq a_+$, Theorem~\ref{thm:exact-time} gives $\Tobs = \min(a_-^2,a_+^2)/2$ in its local setting. Equal widths recover the value $a^2/2$ from \cite[Theorem~1.3]{BMM15}. That theorem treats a Dirichlet rectangle with two symmetric observed strips and also proves failure at the critical time.

In contrast, variable edges illustrate why the shortest distance from the circle to the observation does not determine the exact time. For a symmetric flat gap with width $a(1+\varepsilon\cos\theta)$, $0<\varepsilon<1$, formula \eqref{observation-threshold} gives $\Tobs = a^2(1+\varepsilon)^2/2$, whereas the shortest distance is $a(1-\varepsilon)$. To see why the two sides must be compared before maximizing, take profiles with $(\rho^\pm)^2/2 = A_0\pm\delta\sin\theta$, where $0<\delta<A_0$. Their threshold is $A_0$, while the maximum cost of either side alone is $A_0+\delta$. All widths are chosen within the smallness range of the theorem.

\paragraph{Separated coefficients.}
For the separated operator $-\partial_x^2-q(x)^2\partial_y^2$, assume $q(0) = 0$ and write $\kappa = q'(0)>0$. The transverse Agmon distance used in \cite{DKR23} is $\mathrm{d}_{\mathrm{Agm}}(x) = \int_0^xq(s)\,ds$, positive on either side near zero. At angular frequency $n$ in the original coordinate $y$, the ground state has leading spatial decay $e^{-n\mathrm{d}_{\mathrm{Agm}}(x)}$ and heat decay $e^{-\kappa nt}$. Normalizing $Y = y/\kappa$ gives $\Phi = Y+i\,\mathrm{d}_{\mathrm{Agm}}/\kappa$, so $\mathcal{A} = \mathrm{d}_{\mathrm{Agm}}/q'(0)$ measures attenuation relative to the lowest heat decay rate. For analytic $q$ and thin exterior graphs $-a_-(y) \leq x \leq a_+(y)$, Theorem~\ref{thm:exact-time} reads
$$
\Tobs = \frac{1}{q'(0)}\max_y\min(\mathrm{d}_{\mathrm{Agm}}(-a_-(y)),\mathrm{d}_{\mathrm{Agm}}(a_+(y))).
$$
In the flat case, $\mathrm{d}_{\mathrm{Agm}} = x^2/2$, whereas the sub-Riemannian distance to the circle is $\vert x\vert$.

The geometric formulas in \cite[Corollary~3.5]{DK20} and \cite[Theorem~1.6]{DKR23} have the same interpretation, but concern an \emph{observed band} $\gamma_1(y)<x<\gamma_2(y)$ rather than an exterior collar. With $f^+ = \max(f,0)$ and $f^- = \max(-f,0)$, their thresholds are, respectively,
$$
\frac{1}{2}\max(\Vert\gamma_1^+\Vert_\infty^2,\Vert\gamma_2^-\Vert_\infty^2),\qquad \frac{1}{q'(0)}\max(\mathrm{d}_{\mathrm{Agm}}(\Vert\gamma_1^+\Vert_\infty),\mathrm{d}_{\mathrm{Agm}}(-\Vert\gamma_2^-\Vert_\infty)).
$$
Indeed, at fixed $y$, the infimum of $\mathrm{d}_{\mathrm{Agm}}$ over this observed band is the value at its nearer edge if the band lies on one side, and is zero if it meets $x = 0$. Maximizing this first observed value over $y$ gives exactly the displayed formulas. The variable-coefficient result allows $q\in C^3$ with positive derivative, while the present theorem requires joint analyticity near the circle. Because of the smallness restriction on the collars, the present theorem does not cover the larger widths or boundary geometries allowed in these earlier results. The new assertion concerns coupled coefficients: phase fibers replace straight transverse lines, and the upper proof controls interaction between angular frequencies. A mere change of coordinates to $(r,\theta)$ does not remove that interaction.

\paragraph{A metric determined by the phase.}
For coupled coefficients, integrating $xe^{\varphi(x,y)}$ at fixed $y$ does not generally solve \eqref{phase-definition}. Omit the circle index $i$. Lemma~\ref{phase-local} gives $g=c^{-1}(dr^2+r^{-2}d\theta^2)$ on the punctured tube. If $p$ denotes the cometric, then $p(d\Theta)=cr^2$, so the phase defines the auxiliary Agmon metric
$$
\widetilde g=p(d\Theta)g=r^2dr^2+d\theta^2.
$$
On each side, the coordinate $s=r^2/2$ gives $\widetilde g=ds^2+d\theta^2$. Its metric completion contains the singular circle at $s=0$, and the distance from this whole circle to $(r,\theta)$ is exactly $\mathcal{A}=r^2/2$, attained at fixed $\theta$. The distance from the circle to the observation set takes the infimum over all observed points. By contrast, $M_\Phi(\omega)$ maximizes the first observed action over the phase angles. These quantities need not coincide.

\subsection{Distance collars and curvature}\label{sec:geometric-consequences}

We use the normal forms of Section~\ref{sec:geometry} to express the threshold for collars defined by the sub-Riemannian distance.

\begin{cor}\label{cor:distance-time}
Let $Z_0$ be a singular circle satisfying the hypotheses of Theorem~\ref{thm:exact-time}, and set $\Omega_a = M\setminus\{\dsr(\cdot,Z_0) \leq a\}$. Write $b(y) = \partial_x\varphi(0,y)$, $B_{\min} = \min_{y\in\T_\ell}\vert b(y)\vert$, $S = \{y\in\T_\ell \,\mid\, \vert b(y)\vert = B_{\min}\}$ and $C_* = \frac{1}{8}\max_{y\in S}\big(\partial_x^2\varphi(0,y) + b(y)^2\big)$. One has
\begin{equation}\label{geom:quartic-time}
\Tobs(\Omega_a) = \frac{a^2}{2} - \frac{B_{\min}}{3}a^3 + C_*a^4+\mathrm{o}(a^4) ,
\end{equation}
as $a\downarrow0$.
There exist $d_*,C_4>0$ such that $\Tobs(\omega) \geq \frac{a^2}{2} - \frac{B_{\min}}{3}a^3 - C_4a^4>0$ for every $0<a \leq d_*$ and every nonempty open set $\omega$ with $\dsr(\omega,Z_0) \geq a$.
\end{cor}
\begin{proof}
In the normalized tube, $\vert x\vert = \dsr((x,y),Z_0)$, as stated in the introduction and proved in \eqref{geom:distance} below. Since $\Theta_y(0,y) = 1$, the map $y\mapsto\Theta(x,y)$ is a diffeomorphism for small $x$; denote its inverse by $Y(x,\cdot)$. The boundaries $x = \pm a$ are therefore graphs over the phase angle, and Theorem~\ref{thm:exact-time} gives
$$
\Tobs(\Omega_a) = \max_{\theta\in\T_\ell}\min( \A(a,Y(a,\theta)),\A(-a,Y(-a,\theta))).
$$
Set $f_\pm(\theta) = \A(\pm a,Y(\pm a,\theta))$ and $c = (\partial_x^2\varphi(0,\cdot) + b^2)/8$. Equation~\eqref{geom:phase-expansion} and $Y(\pm a,\theta) = \theta+\mathrm{O}(a^5)$ give, uniformly in $\theta$,
$\min(f_ + (\theta),f_ - (\theta)) = \frac{a^2}{2} - \frac{\vert b(\theta)\vert}{3}a^3 + c(\theta)a^4+\mathrm{O}(a^5)$.
Every limit point of maximizing angles belongs to $S$. Evaluating at a maximizer of $c$ on $S$ gives the lower bound for the coefficient of $a^4$; compactness and $\vert b\vert \geq B_{\min}$ give the matching upper bound. This proves \eqref{geom:quartic-time}. Without further assumptions on the minima of $\vert b\vert$, this argument gives a remainder $\mathrm{o}(a^4)$, not a uniform fifth-order bound. Choose $C_4>0$ and $d_*>0$ within the admissible width range so that the polynomial in the statement is positive and bounds $\Tobs(\Omega_s)$ from below for $0<s \leq d_*$. If $0<a \leq d_*$ and $\dsr(\omega,Z_0) \geq a$, then $\omega\subset\Omega_s$ for $0<s<a$. Observation monotonicity gives $\Tobs(\omega) \geq \Tobs(\Omega_s)$; letting $s\uparrow a$ proves the bound.
\end{proof}

Let $K$ be the Gaussian curvature of the regular metric $g = dx^2+q^{-2}dy^2$, with $q = xe^\varphi$. The curvature formula $K = q_{xx}/q - 2q_x^2/q^2$ \cite[Lemma~2]{ABS08} gives
$$
K(x,y) = -\frac{2}{x^2} - \frac{2b(y)}{x}+\mathrm{O}(1),\qquad d\Big\vert K + \frac{2}{d^2}\Big\vert = 2\vert b(y)\vert +\mathrm{O}(d), \quad d = \dsr((x,y),Z_0).
$$
The remainder is uniform along the circle. Thus $B_{\min}$ is half the minimum of the limiting curvature trace. Reversing the normal coorientation changes the sign of $b$ and preserves this coefficient. For distance collars around $Z_1,\ldots,Z_m$, the exact threshold is the maximum of the preceding formulas over $1 \leq i \leq m$. The quartic formula holds as well, with $B_{\min}$ and $S$ taken on the disjoint union of the selected circles.

\begin{remark}[Collars in normal coordinates]\label{operator-graphs}
In the normalized tubes of Theorem~\ref{thm:exact-time}, let $a_i,b_i \geq 0$ be continuous and sufficiently small, and set $\omega = M\setminus\bigcup_i\{-b_i(y_i) \leq x_i \leq a_i(y_i)\}$. If the boundary maps $y\mapsto\Theta_i(a_i(y),y)$ and $y\mapsto\Theta_i(-b_i(y),y)$ are orientation-preserving homeomorphisms, with inverses $y_i^+$ and $y_i^-$, each boundary meets each phase fiber once. Theorem~\ref{thm:exact-time} therefore reads
$$
\Tobs(\omega) = \max_i\max_\theta\min\big( \A_i(a_i(y_i^+(\theta)),y_i^+(\theta)), \A_i(-b_i(y_i^-(\theta)),y_i^-(\theta))\big).
$$
For $C^1$ edges, it suffices that $\Theta_y + a_i'\Theta_x>0$ on the positive edge and $\Theta_y - b_i'\Theta_x>0$ on the negative edge. Small constant widths and fixed $C^1$ profiles scaled towards the circle satisfy these conditions. Small amplitude alone does not suffice for profiles with unrestricted slopes.
\end{remark}

\subsection{Examples on closed surfaces}\label{sec:surface-examples}

\paragraph{The sphere.}
On the Baouendi-Grushin sphere, the smooth round density is $d\nu = \cos x\,dx\,dy$ and the operator is $\triangle = -\frac{1}{\cos x}\partial_x(\cos x\,\partial_x) - \tan^2x\,\partial_y^2$.
The phase is $\Phi = y-i\log(\cos x)$, with $\mathcal{A} = -\log(\cos x)$. Fix $0<b<\pi/2$ and observe $a<\vert x\vert<b$. For sufficiently small $a>0$, Theorem~\ref{thm:exact-time} applies in a fixed tube of radius less than $b$, even though the polar caps are unobserved, and yields $\Tobs = -\log(\cos a)$. This is the same threshold value as \cite[Theorem~2.6]{Tam26}. That result treats even a unilateral crown $(a,b)\times\mathbb{T}$ and proves failure at the critical time; those assertions are not consequences of our theorem. Observation of an equatorial band has zero minimal time, consistently with \cite[Theorem~2.8]{Tam26}.

\paragraph{A coupled example.}
Take $d\nu = dx\,dy$ and $X_2 = xe^{\varepsilon x\sin y}\partial_y$ near a singular circle of period $2\pi$. Its curvature is $K = -2/x^2 - 2\varepsilon\sin y/x - \varepsilon^2\sin^2y$. For $\varepsilon\ne0$, the renormalized curvature trace varies along the circle, excluding an isometry to a geometry depending only on the normal variable. For the symmetric gap $\vert x\vert \leq a$, Corollary~\ref{cor:distance-time} gives $\Tobs = a^2/2+\mathrm{o}(a^4)$: both the cubic and quartic corrections vanish at the angles where $\sin y = 0$. These local coefficients extend to a closed torus by multiplying the second field of $\partial_x,\sin x\,\partial_y$ by a positive smooth factor equal to $xe^{\varepsilon x\sin y}/\sin x$ near $x = 0$ and to one near the other singular circle; the latter is included in the observation.

\section{Normal forms near a singular circle}\label{sec:geometry}

We construct two normal forms from the complex phase at step two: a real form for the energy estimates and a holomorphic form for normal transport. Both retain the coupling between variables.

\paragraph{The normalized tube.}
We first choose a signed transverse coordinate whose absolute value is the distance to the circle.

\begin{lem}\label{geom:higher-normal-form}
Let $Z_0$ be a compact embedded two-sided singular circle of rank one and constant step $r \geq 2$. There are a periodic tube and an orthonormal frame of $E$ whose images satisfy
$$
U\simeq(-r_0,r_0)\times\mathbb{T}_\ell,\qquad X_1 = \partial_x,\qquad X_2 = x^{r - 1}e^{\varphi(x,y)}\partial_y, \qquad \varphi(0,y) = 0.
$$
On a smaller tube, $\vert x\vert = \dsr((x,y),Z_0)$. If the structure is analytic near $Z_0$, the coordinates and $\varphi$ are analytic.
\end{lem}

The proof of Lemma~\ref{geom:higher-normal-form} is given in Subsection~\ref{dyn:completion}. The construction uses unit conormal geodesics \cite[Lemma~1]{ABS08}, globalized around a two-sided step-two circle in \cite[Proposition~2.3 and Remark~2.4]{BL13}. For the sharp-time proof we work at step two, where the form reads
\begin{equation}\label{geom:normal-form}
U\simeq(-r_0,r_0)\times\T_\ell,\qquad \T_\ell = \R/\ell\Z, \qquad X_1 = \partial_x,\qquad X_2 = xe^{\varphi(x,y)}\partial_y, \qquad\varphi(0,y) = 0.
\end{equation}
The tangential normalization fixes the leading coefficient at one and determines the period $\ell$.

After shrinking the tube, one has
\begin{equation}\label{geom:distance}
\dsr((x,y),Z_0) = \vert x\vert .
\end{equation}
Indeed, the normal segment has length $\vert x\vert $, while a horizontal curve in the tube satisfies $\dot x = u_1$ and has length at least $\int\vert \dot x\vert \,\dd{d}t$. A curve from $Z_0$ leaving a fixed larger tube has length at least its normal radius, so it cannot shorten the normal segment for an endpoint in a smaller tube. This argument applies at every step in Lemma~\ref{geom:higher-normal-form}.

\paragraph{Conjugation by the density.}

The density can be removed from the local $L^2$ space without changing the principal symbol. Write $\dd{d}\nu = \rho(x,y)\,\dd{d}x\,\dd{d}y$, with $\rho>0$, and set $q = xe^\varphi$ and $Q = q^2$. The map $u\mapsto w = \rho^{1/2}u$ is unitary from $L^2(U,\nu)$ to $L^2(U,\dd{d}x\,\dd{d}y)$. Integration by parts gives, for $u\in C_c^\infty(U)$,
$$
\int(\vert u_x\vert^2 + Q\vert u_y\vert^2)\rho\,\dd{d}x\,\dd{d}y = \int(\vert w_x\vert^2 + Q\vert w_y\vert^2 + V\vert w\vert^2)\,\dd{d}x\,\dd{d}y,
$$
where $V = \rho^{-1/2}(\partial_x^2\rho^{1/2} + \partial_y(Q\partial_y\rho^{1/2}))$. Thus the positive Laplacian on Lebesgue space is $\triangle_x = \rho^{1/2}\triangle\rho^{-1/2} = -\partial_x^2 - \partial_y(Q\partial_y) + V$. The density contributes only the real potential $V$, which is analytic when the coefficients and density are analytic.

\paragraph{Real normal form in phase coordinates.}
Assume now that the structure and density are analytic on $U$. The analytic Cauchy theorem gives the phase $\Phi = \Theta+i\A$ of \eqref{phase-definition} on a fixed complex neighborhood; uniqueness gives its periodic increment. On a sufficiently small real tube, $\Theta(x,\cdot):\T_\ell\to\T_\ell$ is a diffeomorphism. Denote its inverse by $Y(x,\cdot)$. Set $b(y) = \partial_x\varphi(0,y)$ and $\varphi_2(y) = \partial_x^2\varphi(0,y)$. Comparing coefficients in \eqref{phase-definition} gives, uniformly around the circle,
\begin{equation}\label{geom:phase-expansion}
\Phi(x,y) = y + \frac{i}{2}x^2 + \frac{i}{3}b(y)x^3 + \frac{i}{8}(\varphi_2(y) + b(y)^2)x^4+\mathrm{O}(x^5).
\end{equation}
In particular, $\Theta(x,y) = y+\mathrm{O}(x^5)$ and $Y(x,\theta) = \theta+\mathrm{O}(x^5)$. This expansion gives the distance-collar correction in Corollary~\ref{cor:distance-time}.

Set $r = \operatorname{sgn}(x)\sqrt{2\A}$ and $\theta = \Theta$. In these coordinates, write $\dd{d}\nu = \rho_\Phi\,\dd{d}r\,\dd{d}\theta$, and let $\mathcal{U}_\Phi$ be the coordinate change followed by multiplication by $\rho_\Phi^{1/2}$; this is a unitary map from $L^2(U,\nu)$ to $L^2(dr\,d\theta)$.

\begin{lem}[Real normal form]\label{phase-local}
On a sufficiently small fixed tube, $\Theta_y>0$, $\A\asymp x^2$, and
$$
\frac{d}{dx}\A(x,Y(x,\theta)) = q(x,Y(x,\theta))\frac{\vert \Phi_y\vert ^2}{\Theta_y}(x,Y(x,\theta)).
$$
The phase coordinates are analytic through $Z_0$ and put the principal symbol in the real normal form $c(r,\theta)(\xi_r^2+r^2\xi_\theta^2)$. The positive Laplacian $\triangle_\Phi = \mathcal{U}_\Phi\triangle\mathcal{U}_\Phi^{-1}$ has differential expression
\begin{equation}\label{phase-operator}
\triangle_\Phi = -\partial_r(c\partial_r) + D_\theta(r^2c)D_\theta + W, \qquad c(0,\theta) = 1,\qquad D_\theta = -i\,\partial_\theta,
\end{equation}
with $c>0$ and $W$ real analytic.
\end{lem}
\begin{proof}
The expansion gives $\Theta_y>0$ and $\A\asymp x^2$ after shrinking the tube. The real and imaginary parts of \eqref{phase-definition} are $\Theta_x = -q\A_y$ and $\A_x = q\Theta_y$. Differentiating $\Theta(x,Y(x,\theta)) = \theta$ gives $Y_x = q\A_y/\Theta_y$, hence
$$
\frac{d}{dx}\A(x,Y(x,\theta)) = q\Big(\Theta_y + \frac{\A_y^2}{\Theta_y}\Big)(x,Y(x,\theta)).
$$
Also $r = x\sqrt{2\A/x^2} = x+\mathrm{O}(x^2)$, where the quotient extends analytically across $x = 0$. Thus $(x,y)\mapsto(r,\theta)$ is an analytic coordinate change. The eikonal identities give
$r_x\Theta_x + q^2r_y\Theta_y = 0$ and $\Theta_x^2 + q^2\Theta_y^2 = r^2(r_x^2 + q^2r_y^2)$.
The first identity removes the mixed term in the transformed cometric, and the second identifies its tangential coefficient. Setting $c = r_x^2 + q^2r_y^2$ gives $c(\xi_r^2 + r^2\xi_\theta^2)$ and $c(0,\theta) = 1$. Its nonvanishing analytic Jacobian transports $\nu$ to a positive analytic density in $dr\,d\theta$. Conjugating by its square root gives \eqref{phase-operator}.
\end{proof}

Thus $\A$ increases outwards along each phase fiber, and $\vert e^{in\Phi}\vert = e^{-nr^2/2}$ for $n>0$. A heat solution $u$ is represented by $v = \mathcal{U}_\Phi u$, which satisfies $\partial_tv+\triangle_\Phi v = 0$ locally.

\paragraph{Holomorphic normal form.}
To linearize the normal transport in \eqref{phase-operator}, we keep the complex phase as one coordinate. Set $w = r$ and $z = \Phi = \theta + ir^2/2$ in the complexified tube. The field $V_+ = \partial_r - ir\partial_\theta$ satisfies $V_+ z = 0$ and becomes $\partial_w$: normal differentiation now keeps the phase fixed. The principal symbol becomes $c\xi_w(\xi_w + 2iw\xi_z)$, so the second derivative in $z$ disappears.

Write $c$ for $c(w,z - iw^2/2)$. The coefficients are periodic and holomorphic on a fixed product $\mathcal{U}_{R,s} = \{\vert w\vert <R,\ \vert \operatorname{Im}z\vert <s\}$, with $R,s>0$. The coefficient of $D_z = -i\,\partial_z$ in $\triangle_\Phi$ is $c + wc_w + 2wc\partial_w$. We linearize its vector-field part and normalize its scalar part by a holomorphic coordinate $\psi$ and a factor $g_0$:
$$
wc\psi_w = \psi,\quad \psi = w+\mathrm{O}(w^2),\qquad 2wc\partial_w\log g_0 = 1 - c - wc_w,\quad g_0(0,z) = 1.
$$
Since $c(0,z) = 1$, the apparent singularities are removable. These equations give periodic holomorphic $\psi,g_0$ on a smaller fixed product, with $g_0\ne0$. Expressing $g_0^{-1}\triangle_\Phi g_0$ in $(\psi,z)$ gives, after renaming $\psi$ as $w$,
\begin{equation}\label{normal-form}
L = (1 + 2w\partial_w)D_z + \mathscr{A},\qquad \mathscr{A} = -d_2(w,z)\partial_w^2 + d_1(w,z)\partial_w + d_0(w,z), \quad d_2(0,z) = 1.
\end{equation}
Here $L$ acts on periodic holomorphic functions, and the coefficients $d_2,d_1,d_0$ are holomorphic on a fixed product. The new normal coordinate is generally complex on real points.

For $j \geq 0$ and $n\in(2\pi/\ell)\N^*$, the Euler term gives $(2j + 1)n$ on $w^je^{inz}$. For the flat operator $-\partial_r^2 + r^2D_\theta^2$, $\mathscr{A} = -\partial_w^2$ lowers degree by two, so $L$ is triangular on polynomials of degree at most $J$ times $e^{inz}$, with these diagonal entries. They give the flat transverse heat decay rates. Coupled coefficients destroy this invariance, motivating the adapted subspace of Proposition~\ref{adapted-subspace}. Differentiation at fixed phase also provides the normal derivatives used in Lemma~\ref{common-height}. Negative angular frequencies are treated by conjugation.

\section{Proof of Theorem A}\label{sec:proof-a}

We prove the local lower and upper bounds in Subsections~\ref{sec:sharp-lower} and~\ref{sec:sharp-upper}. Subsection~\ref{surface-global} combines them on the surface and fixes an admissible collar width independent of the time margin. Subsection~\ref{sec:general-observation} then applies the local estimates to more general observation sets.

\subsection{Lower bound}\label{sec:sharp-lower}

On the flat cylinder, every periodic holomorphic function $F$ gives the heat solution $F(y + i(x^2/2 + t))$. A single exponential recovers the Gaussian state, but holomorphic combinations can concentrate at a chosen angle \cite{DKR23,Ko17}. For the coupled operator, we first realize the inputs $F(\Phi + it)$ with estimates uniform in $F$. Polynomial separation then gives the local obstruction in Proposition~\ref{local-lower}, stated and proved below.

Fix one analytic normalized tube and write $\alpha = 2\pi/\ell$.

\paragraph{Holomorphic heat trajectories.}

We seek a linear realization on one fixed complex domain. The phase removes second derivatives of the input; a scalar gauge reduces the remaining angular derivative to the drift absorbed by $\Phi + it$.

The density conjugation of Section~\ref{sec:geometry} gives $\triangle_x = \rho^{1/2}\triangle\rho^{-1/2} = -\mathcal{D}_0 + V$ locally, with $\mathcal{D}_0 = \partial_x^2 + \partial_y(q^2\partial_y)$. Since $\Phi_y(0,y) = 1$, the change $(x,y)\mapsto(x,z) = (x,\Phi(x,y))$ is biholomorphic near the circle. Fix a product $\mathcal{U}_{R_0,s_0} = \{\vert x\vert <R_0,\ \vert \Im z\vert <s_0\}$ in its image, with $z\in\mathbb{C}/\ell\Z$. Normal derivatives in these coordinates hold $z$ fixed. The eikonal equation gives
$\widetilde{\mathcal{D}}_0 = \partial_x^2 + 2\Phi_x\partial_x\partial_z + (\mathcal{D}_0\Phi)\partial_z$.
Write $\widetilde{\triangle}_x = -\widetilde{\mathcal{D}}_0 + V$ for the Laplacian in these holomorphic coordinates. The phase derivatives are computed in $(x,y)$, and all coefficients, including $V$, are composed with the inverse map. There is no $\partial_z^2$ term.

Set $x\mathcal{B} = -2i\Phi_x$ and $C = -i\mathcal{D}_0\Phi - 1$. The normalization at the circle gives $\mathcal{B}(0,z) = 2$ and $C(0,z) = 0$. After decreasing $R_0,s_0$ once, the periodic nonvanishing gauge
\begin{equation}\label{lower-gauge}
g(x,z) = \exp\Big(-\int_0^x\frac{C(s,z)}{s\mathcal{B}(s,z)}\,ds\Big)
\end{equation}
is holomorphic and satisfies $\mathcal{D}_0\Phi + 2\Phi_xg_x/g = i$. The integrand has a removable singularity at zero. Consequently,
$$
g^{-1}(\widetilde{\mathcal{D}}_0 - V)g = \partial_x^2 + 2\Phi_x\partial_x\partial_z + p\partial_x + i\,\partial_z + v, \qquad p = 2g_x/g + 2\Phi_xg_z/g,\quad v = g^{-1}\widetilde{\mathcal{D}}_0g - V.
$$
To absorb $i\,\partial_z$ into the heat evolution, set $\zeta = z + it$ and choose $R,\tau,s>0$ with $R<R_0$ and $s + \tau<s_0$. All coefficients below are holomorphic on
$$
\mathcal{P}_{R,\tau,s} = \{\vert t\vert <\tau,\ \vert x\vert <R,\ \vert \Im\zeta\vert <s\}, \qquad \zeta\in\mathbb{C}/\ell\Z.
$$
Subscript $t$ means evaluation at $z = \zeta - it$. Define
$$
R_t = \partial_x^2 + 2(\Phi_x)_t\partial_x\partial_\zeta + p_t\partial_x + v_t, \qquad\mathscr{P}_t = \partial_t - R_t,\qquad L_t = x\mathcal{B}_t\partial_x.
$$
The preceding changes give the exact identities
\begin{align}
(\partial_t + \widetilde{\triangle}_x)\{g(x,z)U(t,x,z + it)\} & = g(x,z)\mathscr{P}_tU(t,x,z + it),\label{lower-conjugation} \\
\mathscr{P}_t(bF)& = (\mathscr{P}_tb)F + (L_tb)D_\zeta F, \qquad D_\zeta = -i\,\partial_\zeta.\label{lower-product}
\end{align}
Here $F$ depends only on $\zeta$. Solving the normal transport equation determines each amplitude up to its trace at $x = 0$. The source of the next equation must vanish there, which determines this trace through a time evolution equation.

\begin{lem}\label{factorial}
There are a fixed smaller product $\mathcal{P}_{R,\tau,s}$ and periodic holomorphic functions $b_j$, $j \geq 0$, such that $\inf_{\mathcal{P}_{R,\tau,s}}\vert b_0\vert >0$ and
$L_tb_0 = 0$, $L_tb_{j + 1} + \mathscr{P}_tb_j = 0$ with $\max_{\vert \gamma\vert \leq 2}\Vert \partial^\gamma b_j\Vert_\infty \leq C\Lambda^j j!$.
The norm is the supremum on this product, with $\partial^\gamma = \partial_t^{\gamma_1}\partial_x^{\gamma_2}
\partial_\zeta^{\gamma_3}$. Its radii and $C,\Lambda>0$ depend only on
the fixed analytic coefficients and domains, and are independent of $j$.
\end{lem}

The proof is given in Appendix~\ref{proof:factorial}, using the same normal inverse and factorial calculus as the stationary construction of Section~\ref{wkb:section}. For $F = e^{in\zeta}$, $n\in\alpha\N^*$, it gives the formal amplitude $\sum b_jn^{-j}$. The identity $F(\zeta + i\beta) = e^{-n\beta}F(\zeta)$ suggests replacing inverse powers of $n$ by a Laplace integral. The factorial bound makes its kernel convergent. Truncating the integral places the entire defect at one fixed positive imaginary shift. The collars below will be chosen small enough that this shift exceeds their attenuation.

\begin{prop}\label{realization}
Use the product of Lemma~\ref{factorial} and fix $\beta_{\rm B}>0$ with $\Lambda\beta_{\rm B}<1/2$. Set
$$
K(t,x,\zeta,\beta) = \sum_{j \geq 1}b_j(t,x,\zeta)\frac{\beta^{j - 1}}{(j - 1)!}, \qquad \mathcal{U}_F = b_0F(\zeta) + \int_0^{\beta_{\rm B}} K(t,x,\zeta,\beta)F(\zeta + i\beta)\,d\beta.
$$
The kernel is holomorphic on $\mathcal{P}_{R,\tau,s}\times\{\vert \beta\vert <\Lambda^{-1}\}$. For every $\ell$-periodic holomorphic $F$ on $\vert \Im\zeta\vert <s + \beta_{\rm B}$, $\mathcal{U}_F$ is periodic and holomorphic on $\mathcal{P}_{R,\tau,s}$ and
$\mathscr{P}_t\mathcal{U}_F = -L_tK(t,x,\zeta,\beta_{\rm B})F(\zeta + i\beta_{\rm B})$.
Moreover, $\Vert \mathcal{U}_F\Vert_{L^\infty(\mathcal{P}_{R,\tau,s})} \leq C\Vert F\Vert_{L^\infty(\vert \Im\zeta\vert <s + \beta_{\rm B})}$ whenever the
right-hand side is finite, with $C$ independent of $F$.
\end{prop}
\begin{proof}
Lemma~\ref{factorial} gives convergence with the derivatives in $\mathscr{P}_t$ and $L_t$. The hierarchy implies $\mathscr{P}_tK + L_t\partial_\beta K = 0$ and $L_tK\vert_{\beta = 0} = -\mathscr{P}_tb_0$. By \eqref{lower-product} and $D_\zeta F(\zeta + i\beta) = -\partial_\beta F(\zeta + i\beta)$, integration by parts cancels the interior terms and the endpoint at zero, leaving the stated defect. Only the input $F$ is shifted. All kernel coefficients remain on $\mathcal{P}_{R,\tau,s}$, so the specified larger input strip suffices. Uniform bounds on $b_0$ and $K$ over $0 \leq \beta \leq \beta_{\rm B}$ give the norm estimate.
\end{proof}

\paragraph{The local obstruction.}
Choose a smooth normal cutoff $\chi$, supported in the construction tube and equal to one on a smaller tube. Fix the construction product and cutoff so that, for some $t_*>0$, all $(t,x,\Phi(x,y) + it)$ with $0 \leq t \leq t_*$ and $x\in\operatorname{supp}\chi$ lie in a strictly smaller product. Set $b_\chi = \inf_{\operatorname{supp}\chi'\times\mathbb{T}_\ell}\A>0$ and choose $0<h_{\rm loc}<\frac{1}{2}\min(t_*,b_\chi,\beta_{\rm B})$. None of these choices depends on $T$ or $\underline m$.

\begin{prop}\label{local-lower}
Let $\omega$ be open and let $\underline m:\mathbb{T}_\ell\to[0,+\infty]$ be lower semicontinuous. Assume that $\underline m(\Theta(p)) \leq \A(p)$ for any $p\in\omega\cap\operatorname{supp}\chi$.
Observation fails for every $T\in(0,\min(h_{\rm loc},\sup\underline m))$. In particular, for the phase-graph exterior collar
$\omega = \{r>\rho_ + (\theta)\}\cup\{r<-\rho_ - (\theta)\}$
in the smaller tube, with the exterior fully observed, let $\rho_\pm$ be continuous and nonnegative. If $M = \max_\theta\min(\rho_ + (\theta)^2/2,\rho_ - (\theta)^2/2)<h_{\rm loc}$, then $T_{\rm obs}(\omega) \geq M$.
\end{prop}

\begin{proof}
Exponentiation $\xi = e^{i\alpha\zeta}$ turns heat evolution and the Borel shifts into radial contractions. We use one compact set to control observation and defects, then compare with the exact heat flow.

\emph{Polynomial separation.} For $T<\min(h_{\rm loc},\sup\underline m)$ set $\underline m_0 = \min(\underline m,h_{\rm loc})$ and
$\mathscr{K} = \{a e^{i\alpha\theta} \,\mid\, 0 \leq a \leq e^{-\alpha\underline m_0(\theta)}\}$.
Lower semicontinuity of $\underline m_0$ makes this set compact. Its complement is connected, since every exterior point can move radially to infinity. Choose $\theta_0$ with $\underline m_0(\theta_0)>T$. Then $\xi_0 = e^{i\alpha\theta_0 - \alpha T}\notin\mathscr{K}$. Runge approximation provides a polynomial $q_0$ with $\sup_{\mathscr{K}}\vert q_0\vert <1<\vert q_0(\xi_0)\vert $. Since $0\in\mathscr{K}$, taking a sufficiently large fixed power and subtracting its value at zero gives $p = q_0^m - q_0(0)^m$ such that
$p(0) = 0$, $\sup_{\mathscr{K}}\vert p\vert \leq \varrho<1$ and $P_\star = \max_{\vert \xi\vert = e^{-\alpha T}}\vert p(\xi)\vert >1$.
This is the polynomial separation used in \cite{DKR23,DK20,Ko17}. The condition $p(0) = 0$ will also force the smallest input frequency to tend to infinity, preventing the Borel correction from cancelling the terminal leading term.

\emph{Heat trajectories.} Let $d = \deg p$, $P_k = p^k$ and $F_k(\zeta) = P_k(e^{i\alpha\zeta})$. Define
$$
U_k(t,x,y) = \rho^{-1/2}\chi(x)g(x,\Phi) \mathcal{U}_{F_k}(t,x,\Phi + it), \qquad u_k(t) = e^{-t\triangle}U_k(0).
$$
The smooth functions $U_k$ extend by zero to $M$ and belong to $D(\triangle)$. The factor $\rho^{-1/2}$ restores the original density representation, so $u_k$ is the exact global trajectory in $L^2(M,\nu)$. All input frequencies lie in $\alpha\{k,\ldots,dk\}$.

\emph{Observation and defects.} At observed points of the cutoff support, $\A \geq \underline m_0(\Theta)$. Hence every argument $e^{i\alpha(\Phi + it + i\beta)}$, $\beta \geq 0$, belongs to $\mathscr{K}$. Proposition~\ref{realization} and the uniform bounds on $b_0,K$ give $\sup_{0 \leq t \leq T}\Vert U_k(t)\Vert_{L^2(\omega,\nu)} \leq C_T\varrho^k$. The disk $\vert \xi\vert \leq e^{-\alpha h_{\rm loc}}$ is contained in $\mathscr{K}$. At the Borel endpoint, the imaginary part of the input argument is at least $\beta_{\rm B}$; on the cutoff transition, it is at least $b_\chi$. Both exceed $h_{\rm loc}$, so the corresponding exponential arguments lie strictly inside this disk. Cauchy's inequality on this disk therefore bounds both $P_k$ and $P_k'$ on the transition by $C\varrho^k$, with a constant independent of $k$. In the conjugated representation, $[\triangle_x,\chi] = -2\chi'\partial_x - \chi''$, where this derivative is in the original real coordinates. Using \eqref{lower-conjugation}, restoring the density, and then applying the global contraction semigroup gives
$$
\sup_{0 \leq t \leq T}\Vert (\partial_t + \triangle)U_k(t)\Vert \leq C_T\varrho^k, \qquad\sup_{0 \leq t \leq T}\Vert u_k(t) - U_k(t)\Vert \leq C_T\varrho^k.
$$
The second estimate follows from Duhamel's formula and includes the part of the exact trajectory outside the construction tube.

\emph{Terminal mass.} Set $r_T = e^{-\alpha T}$ and $H_k = \Vert P_k(r_Te^{i\alpha y})\Vert_{L^2(\mathbb{T}_\ell,dy)}$. Parseval and the lowest frequency $k$ give
$\Vert P_k(e^{-\alpha\beta}r_Te^{i\alpha y})\Vert_{L^2} \leq e^{-\alpha k\beta}H_k$ for $\beta \geq 0$.
At $x = 0$, the multiplier $\rho^{-1/2}gb_0$ is bounded away from zero. Minkowski's inequality bounds the Borel correction by $C_TH_k\int_0^{\beta_{\rm B}}e^{-\alpha k\beta}d\beta \leq C_TH_k/k$. Thus $\Vert U_k(T,0,\cdot)\Vert_{L^2(dy)} \geq c_TH_k$ for large $k$. Finite Fourier Cauchy-Schwarz gives $H_k \geq c_dk^{-1/2}P_\star^k$.

To convert this trace into volume mass, note that $\A \geq 0$ sets every terminal input argument inside $\vert \xi\vert \leq r_T$. Bernstein's inequality for the degree-$dk$ polynomial $P_k$, together with the fixed coefficient bounds, yields $\Vert \partial_xU_k(T,x,\cdot)\Vert_{L^\infty} \leq C_T(1 + dk)P_\star^k$ on a fixed small real tube. For $\vert x\vert \leq k^{-2}$,
the $L^2(dy)$ difference from the central trace is at most $C_Tk^{-1}P_\star^k = \mathrm{o}(H_k)$. Integrating over this band, where the density is uniformly positive, gives $\Vert U_k(T)\Vert \geq c_Tk^{-1}H_k \geq c_Tk^{-3/2}P_\star^k$. The Duhamel error preserves this lower bound. Therefore
$$
\frac{\int_0^T\Vert u_k(t)\Vert_{L^2(\omega,\nu)}^2\,dt}{\Vert u_k(T)\Vert ^2} \leq C_Tk^3(\varrho/P_\star)^{2k}\longrightarrow0,
$$
as $k\to+\infty$, which contradicts observability. For the phase-graph collar take $\underline m = \min(\rho_+^2,\rho_-^2)/2$. The increase of $\mathcal{A}$ away from the circle along each phase fiber, given by Lemma~\ref{phase-local}, verifies the hypothesis and proves the second assertion.
\end{proof}

For general $\omega$, the admissible $\underline m(\theta)$ bounds from below the infimum of $\mathcal{A}$ over the observed points of the phase fiber $\Theta = \theta$ within $\operatorname{supp}\chi$. Its largest lower semicontinuous choice is
$$
m(\theta) = \inf\{\A(p) \,\mid\,  p\in\omega\cap\operatorname{supp}\chi, \ \Theta(p) = \theta\},\qquad \underline m(\theta) = \lim_{\varepsilon\downarrow0} \inf_{d(\theta',\theta)<\varepsilon}m(\theta'),
$$
where an empty infimum is $+\infty$. Only a lower obstruction is asserted for such sets.

\subsection{Upper bound}\label{sec:sharp-upper}\label{upper-functional-setting}

We work in one phase cylinder and write $M_\Phi = \max_\theta\min(\rho_+(\theta)^2/2,\rho_-(\theta)^2/2)$. The real normal form \eqref{phase-operator} is used for the weighted energy estimates; the holomorphic normal form \eqref{normal-form} is used to construct the adapted transverse subspace. We first obtain an observation estimate with an angular loss, then reduce this loss to any $h>M_\Phi$ by separating the transverse degrees. The final weighted energy estimate removes it in time arbitrarily close to $h$.

\paragraph{Angular regularization and heat dissipation.}

We first use observation in the outer strips to control the solution with an exponential angular loss. This estimate does not yet reflect the collar profile. We then show how heat decay removes an angular loss, at the cost of a corresponding waiting time.

Fix $\mathcal{C} = (-r_0,r_0)\times\T_\ell$ and $\mathscr{H} = L^2(\mathcal{C},dr\,d\theta)$, with inner products linear in the first argument. Local norms below are taken in $\mathscr{H}$ unless specified otherwise. Let $N = \vert D_\theta\vert $, so $e^{-hN}$ multiplies angular frequency $n$ by $e^{-h\vert n\vert }$. From now on, $\triangle_\Phi$ denotes the nonnegative Dirichlet realization of \eqref{phase-operator} on $\mathcal{C}$. Its form domain $\mathcal{V}$ is the completion of smooth periodic functions compactly supported in $r$ for the norm $(\Vert v\Vert ^2 + \Vert v_r\Vert ^2 + \Vert rD_\theta v\Vert ^2)^{1/2}$. Angular dependence of the coefficients couples the Fourier modes. For $\vert \eta\vert $ below a fixed analytic strip width, the formal conjugate $e^{\eta N}\triangle_\Phi e^{-\eta N}$ has a closed sectorial realization $\triangle_{\Phi,\eta}$ with form $\mathfrak{h}_\eta$ on $\mathcal{V}$. Its coefficient matrix construction is given in Appendix~\ref{app:dirichlet}. Fix an admissible width $h_*>0$ for these forms, independently of the parameters $J,\varepsilon$ and of the time margin.

To apply these cylinder estimates to the global solution, choose a fixed normal cutoff $\varkappa$, equal to one on an inner cylinder containing the gap, with its derivative supported in observed regular strips. For the local unitary representative of the global heat solution, set $U = \varkappa u$. Then
$$
U_t + \triangle_\Phi U = F_{\rm obs} = [\triangle_\Phi,\varkappa]u.
$$
For every $I\Subset(0,T)$ and integer $k \geq 0$, interior parabolic regularity on slightly larger observed regular strips gives
$$
\Vert F_{\rm obs}\Vert_{L^2(I;H^k(\mathcal{C}))} \leq C_{I,k}\obs^{1/2},\qquad \obs = \int_0^T\Vert u(t)\Vert_{L^2(\omega,\nu)}^2dt.
$$
The compact time margins are needed here; terminal bounds below use an energy identity instead of a source estimate at $T$.

\begin{prop}\label{rough-state}
If the unobserved part of $\mathcal{C}$ lies in $\vert r\vert \leq d$, then, for $d$ sufficiently small, there exists $B_0 \leq Cd^2$ such that $\sup_{t\in I}\Vert e^{-B_0N}U(t)\Vert ^2 \leq C_{d,\tau,I}\obs$ for every $\tau>0$ and $I\Subset(\tau,T)$. The loss is independent of time and of the boundary regularity of the gap.
\end{prop}
The proof is given in Appendix~\ref{proof:rough-state}.

The flat Dirichlet operator $\triangle_{\mathrm{osc}} = -\partial_r^2 + r^2N^2$ has lowest oscillator rate $\vert n\vert $, so its heat flow supplies the factor $e^{-t\vert n\vert }$. The next estimate gives the corresponding decay for the coupled operator. Raising a negative exponent $-h$ to zero at speed $1 - \varepsilon$ will convert control of $e^{-hN}U$ into control of $U$.

\begin{lem}\label{weighted-dissipation}
One has $\Re\mathfrak{h}_\eta(v,v) \geq (1 - \varepsilon)\Vert N^{1/2}v\Vert ^2 - C_{\varepsilon,h}\Vert v\Vert ^2$, for all $\varepsilon\in(0,1)$ and $h\in(0,h_*)$, uniformly with respect to $\eta\in[-h,0]$ and $v\in\mathcal{V}$.
An observation bound for $e^{-hN}U(s)$ therefore gives an unweighted bound after time $h/(1 - \varepsilon)$, provided this time is strictly before $T$. A bound at a fixed negative weight also propagates forward.
\end{lem}

\begin{proof}
Lemma~\ref{complement-coercivity}, with $J = -1$ and $\Pi_{-1} = 0$, gives the asserted coercivity on the fixed cylinder.

Set $\eta(t) = -h + (1 - \varepsilon)(t - s)$ until it reaches zero and $z = e^{\eta(t)N}U$. Then
$z_t + \triangle_{\Phi,\eta(t)}z = \eta'(t)Nz + e^{\eta(t)N}F_{\rm obs}$.
The coercivity estimate absorbs $\eta'N$. The source is acted on only by a contraction, so energy and Gronwall give the unweighted bound. The calculation at constant $\eta \leq 0$ propagates a weighted bound. On compact positive-time intervals, $U\in C^1(\mathscr{H})\cap C(D(\triangle_\Phi))$. Lemma~\ref{operator-graph} controls $NU$, and negative angular weights preserve $D(\triangle_{\mathrm{osc}}) = D(\triangle_\Phi)$, justifying the energy identity. The domain bounds for the moving projection give the same justification for the complement below.
\end{proof}

\paragraph{A subspace adapted to the heat evolution.}

Removing the rough loss by Lemma~\ref{weighted-dissipation} would cost time $B_0$. We first remove it on the rapidly decaying transverse modes. In the flat whole-line model, the functions $r^je^{-nr^2/2}$, $0 \leq j \leq J$, span the first $J + 1$ oscillator levels at frequency $n>0$. Their orthogonal complement has decay rate at least $(2J + 3)n$. Coupling destroys this invariant decomposition: the retained modes force the complement. The construction below deforms their span with time to cancel that forcing up to an exponentially small error.

The complementary loss is removed in time of order $B_0/(2J + 3)$. The complex domains and defect exponents are fixed before $J$, so increasing $J$ does not require a thinner collar.

Let $\mathcal{F}_J$ be the $\ell^2$ space of coefficients $f_{\sigma,j,n}$, indexed by $\sigma = \pm1$, $0 \leq j \leq J$ and $n\in(2\pi/\ell)\N^*$, $n \geq N_0$. Proposition~\ref{adapted-subspace} constructs $K_J(t):\mathcal{F}_J\to\mathscr{H}$ by summing these coefficients against the columns
\begin{equation}\label{upper-columns}
k_{\sigma,j,n}(t) = \chi(r)n^{j/2 + 1/4}g_\sigma e^{i\sigma n z_\sigma} \big(w_\sigma^j + w_\sigma^{J + 1}b_{\sigma,j,n}(t)\big), \qquad z_\sigma = \theta + \sigma ir^2/2.
\end{equation}
Here $w_+ = \psi(r,z_+)$ and $g_+ = g_0(r,z_+)$ are the coordinate and factor of the holomorphic normal form \eqref{normal-form}. The negative-sign expressions are their complex conjugates on the real cylinder. Each correction is evaluated at $(t,w_\sigma,z_\sigma)$. The fixed normal cutoff $\chi$ equals one near the gap. The leading term is a normal monomial times $e^{i\sigma n\theta - nr^2/2}$, normalized in $L^2$. The correction begins at degree $J + 1$, so it preserves the retained Taylor coefficients. It satisfies $\Vert b_{\sigma,j,n}(t)\Vert_{C^k} \leq C_{J,k}/n$ on one fixed smaller complex product, for every fixed spatial order $k$.

The norm equivalence $\Vert K_Jf\Vert \asymp_J\Vert f\Vert_{\mathcal{F}_J}$ gives closed ranges and the bounded left inverse $K_J^\dagger = (K_J^*K_J)^{-1}K_J^*:\mathscr{H}\to\mathcal{F}_J$. Let $P(t)$ be the orthogonal projection onto $\operatorname{Ran}K_J(t)$, ${P^\perp} = I - P$. The defect $E_J:\mathcal{F}_J\to\mathscr{H}$ satisfies
${P^\perp}(K_J'+\triangle_\Phi K_J - E_J) = 0$.
It holds first on finite coefficient sequences. The defect block from input frequency $\sigma n$ to output frequency $m$, as a map $\mathbb{C}^{J + 1}\to H^k((-r_0,r_0))$, is bounded by $C_{J,k}e^{-\gamma_* n - s_*\vert m - \sigma n\vert }$ for every fixed $k$. The constants $\gamma_*,s_*,b_*>0$ and the complex product are independent of $J$, and the family exists for $0 \leq (2J + 3)t<b_*$. Constants in norm bounds and the threshold $N_0$ may depend on $J$.

We split the actual localized solution as $U = U_J + v$, where $U_J = PU$ and $v = {P^\perp}U$. Lemmas~\ref{current-projections} and~\ref{complement-coercivity} give bounded $P_\eta = e^{\eta N}Pe^{-\eta N}$ and $P'_\eta = e^{\eta N}P'e^{-\eta N}$ on $\mathscr{H}$, for $\vert \eta\vert \leq s_*$, after a fixed reduction of $s_*$. On the complementary space they give
$$
P_\eta z = 0\ \Rightarrow\ \Re\mathfrak{h}_\eta(z,z) \geq (2J + 3 - \varepsilon)\Vert N^{1/2}z\Vert ^2 - C_{J,\varepsilon}\Vert z\Vert ^2 \qquad\forall z\in\mathcal{V}.
$$
Set $R = {P^\perp}E_JK_J^\dagger:\mathscr{H}\to\mathscr{H}$. The invariance identity and its adjoint yield the exact equation
$v_t + \triangle_\Phi v = -2P'v + (R^*-R)U + {P^\perp}F_{\rm obs}$.
Appendix~\ref{app:complement} justifies this identity and proves that, when $B_0<\min(\gamma_*,s_*)$,
$\Vert e^{\eta N}(R^*-R)U\Vert \leq C_J\Vert e^{-B_0N}U\Vert$ for $- B_0 \leq \eta \leq 0$.
To use the rough observation estimate, the defect must be controlled by the weighted state $e^{-B_0N}U$. Its factor $e^{-\gamma_*n}$ compensates the inverse weight $e^{B_0n}$. For the adjoint, the inequality $e^{B_0\vert m\vert}\leq e^{B_0n}e^{B_0\vert m-\sigma n\vert}$ also uses the decay in the frequency difference.

Start at $t_0>0$ and use $K_J(t - t_0)$. For $0<\varepsilon<1/4$, set
$\gamma_J = 2J + 3 - 2\varepsilon$ and $t_J = B_0/\gamma_J$.
Raise the weight from $-B_0$ to zero by $\eta(t) = -B_0 + \gamma_J(t - t_0)$, for $t_0 \leq t \leq t_0 + t_J$. Then $z = e^{\eta N}v$ satisfies $P_\eta z = 0$, and complementary coercivity absorbs the term $\eta'Nz$ in its energy equation. The rough estimate controls its initial state and defect forcing. Boundedness of $P'_\eta,{P^\perp}_\eta$ controls the other terms without a positive weight on the observed source. Energy and Gronwall therefore give $\Vert v(t)\Vert ^2 \leq C\obs$ after $t_0 + t_J$, within the construction interval.

On every compact regular strip $\omega'\Subset\mathcal{C}\cap\{\vert r\vert >0\}$, Lemma~\ref{tail-jets} then gives
$$
\Vert v\Vert_{L^2(I;H^k(\omega'))}^2 \leq C_{J,k,I,\omega'}\obs, \qquad I\Subset(t_0 + t_J,t_0 + b_*/(2J + 3))\cap(0,T).
$$
These derivative bounds allow us to recover $U_J = U-v$ from the observed solution. Increasing $J$ makes the complementary waiting time $t_J$ tend to zero, but the retained subspace still contains the slowest transverse degree. We therefore estimate this part directly on observed phase levels, with loss $h>M_\Phi$, before using heat dissipation once more.

\paragraph{Observation on level sets.}

It remains to recover $U_J = K_Jf$ with the geometric loss. Choose $A>M_\Phi$. At every phase angle, at least one of $(\sqrt{2A},\theta)$ and $(-\sqrt{2A},\theta)$ is observed. Both points have $z_\sigma = \theta + \sigma iA$, so each angular sign has the same attenuation $e^{-An}$ on either side. Only the gauge and normal amplitudes depend on the chosen side. Thus the same angular coefficients can be recovered even when the observed side varies with $\theta$.

A value alone cannot recover all retained transverse degrees, since they may cancel. The value and the first $J$ normal derivatives recover a degree-$J$ polynomial. Ordinary $r$-derivatives also differentiate the Gaussian factor. We instead differentiate the normal polynomial while holding the complex phase fixed. Define
$V_\sigma = \partial_r - \sigma ir\partial_\theta$ and $\mathcal{D}_\sigma = g_\sigma(V_\sigma w_\sigma)^{-1}V_\sigma g_\sigma^{-1}$.
Their coefficients are smooth and nonsingular on $\mathcal{C}$. Since $V_\sigma z_\sigma = 0$, this divides out the gauge, differentiates in $w_\sigma$ without differentiating the oscillation, and restores the gauge. These operators are combinations of spatial derivatives, which observed parabolic regularity controls on regular strips. We use both signs on each observed side. Decrease the fixed width $h_*$ once, independently of $J,\varepsilon$ and of the time margin, below the coefficient strip used for $K_J$ and so that $\chi = \varkappa = 1$ on $r^2/2 \leq h_*$. The following estimate combines polynomial recovery with observation on the level set $\mathcal{A} = A$.

\begin{lem}\label{common-height}
Let $0<A<h<h_*$, and suppose every angle has at least one observed point $(\pm\sqrt{2A},\theta)$. There is $\omega'\Subset\omega\cap\mathcal{C}$ such that, for every $f\in\mathcal{F}_J$ and $0 \leq (2J + 3)t<b_*$,
$$
\Vert e^{-hN}K_J(t)f\Vert ^2 \leq C_{J,h}\sum_{k = 0}^J \big(\Vert \mathcal{D}_+^kK_J(t)f\Vert_{L^2(\omega')}^2 + \Vert \mathcal{D}_-^kK_J(t)f\Vert_{L^2(\omega')}^2\big).
$$
For exterior phase graphs this holds for every $M_\Phi<h<h_*$, by choosing $M_\Phi<A<h$. The angular threshold may be enlarged after $J$ and the selected observed strips are fixed. The constant may depend on these strips.
\end{lem}
\begin{proof}
At each angle, choose an observed point on one of the two sides at height $A$. Openness gives a product of angular and height intervals contained in $\omega$ around that point. A finite collection of these products covers the circle at height $A$. A smooth partition subordinate to their angular intervals, grouped by side and rescaled, gives cutoffs $\vartheta_\pm$ with $\vartheta_+^2 + \vartheta_-^2 \geq 1$. Shrinking the finitely many height intervals gives $0<A_0<A<A_1<h$ such that the selected strips over the cutoff supports are compactly contained in $\omega$. For exterior graphs, choose $A_0>M_\Phi$; continuity of the edges suffices.

On the level set $\mathcal{A} = B$, with $B\in[A_0,A_1]$, set $g_{\sigma,j,B} = \sum_{n \geq N_0}n^{j/2 + 1/4}e^{-Bn}f_{\sigma,j,n}e^{i\sigma n\theta}$. The leading term for angular sign $\sigma$ is $g_\sigma\sum_{j = 0}^Jw_\sigma^jg_{\sigma,j,B}$. Its value and matching derivatives act on this vector of angular coefficients by the triangular matrix defined by
$$
(\mathsf{M}_\sigma)_{kj} = g_\sigma\frac{j!}{(j - k)!}w_\sigma^{j - k}\quad\text{if}\ j \geq k, \qquad (\mathsf{M}_\sigma)_{kj} = 0\quad\text{if}\ j<k.
$$
Its determinant is $g_\sigma^{J + 1}\prod_{j = 0}^Jj!$, bounded away from zero on the compact selected strips. The higher-degree correction in
\eqref{upper-columns} gives Fourier entries
$C_Ln^{-1}(1 + \vert m - \sigma n\vert )^{-L}$; both Schur sums tend to zero as $N_0\to + \infty$.

For a mixture of angular signs, retain both families of derivatives. All differentiations are performed before restriction to $r = s\sqrt{2B}$, where $s = \pm1$ denotes the observed side. Since $V_\sigma z_\sigma = 0$, the matching derivatives differentiate only the normal amplitude in the coordinate $w_\sigma$. For the opposite sign, $V_{-\sigma}e^{i\sigma nz_\sigma} = -2nr e^{i\sigma nz_\sigma}$. Repeated differentiation therefore preserves the angular factor $e^{i\sigma n\theta}$ and introduces at most $n^k$ at derivative order $k$.

At fixed $B$, write $a_{\sigma,j,n} = n^{j/2 + 1/4}e^{-Bn}f_{\sigma,j,n}$ and $U_{J,\sigma} = K_{J,\sigma}f_\sigma$. The triangular inverse and the correction estimate imply, uniformly in $B\in[A_0,A_1]$,
$$
c_J\sum_{j,n}\vert a_{\sigma,j,n}\vert^2 \leq \sum_{k = 0}^J\Big(\int\vartheta_+^2\vert \mathcal{D}_\sigma^kU_{J,\sigma}(\sqrt{2B},\theta)\vert^2d\theta + \int\vartheta_-^2\vert \mathcal{D}_\sigma^kU_{J,\sigma}(-\sqrt{2B},\theta)\vert^2d\theta\Big)
$$
after increasing $N_0$. To see why the correction is small in the required norm, its Fourier entries are bounded by $C_Ln^{-1}(1 + \vert m - \sigma n\vert)^{-L}$; both Schur sums are $\mathrm{O}(N_0^{-1})$. This estimate acts on the already attenuated coefficients $a_{\sigma,j,n}$ and introduces no additional exponential loss.

In the squared observation norm of $U_{J,+} + U_{J,-}$, the matching squares above are accompanied by nonnegative opposite-sign squares and by cross terms. Every cross term between frequencies $n,m \geq N_0$ contains $e^{\pm i(n + m)\theta}$. The amplitudes, the selected cutoffs, and their derivatives are uniformly bounded on the compact strips, apart from polynomial factors of degree at most $2J$ in $n,m$. Integration by parts $L$ times in $\theta$ therefore bounds the cross Gram entries by $C_{J,L}(n + m)^{2J - L}$. Taking $L>2J + 2$, their two Schur sums are bounded by $C_{J,L}N_0^{2J - L + 1}$ and tend to zero. Choose $N_0$ so that this cross Gram norm is less than $c_J/4$. Cauchy-Schwarz absorbs the mixed terms into the matching squares:
\begin{multline*}
\sum_j\big(\Vert g_{+,j,B}\Vert^2 + \Vert g_{-,j,B}\Vert^2\big) 
\leq C_J\sum_{k = 0}^J\Big(\int\vartheta_+^2\big(\vert \mathcal{D}_+^kK_Jf(\sqrt{2B},\theta)\vert^2 + \vert \mathcal{D}_-^kK_Jf(\sqrt{2B},\theta)\vert^2\big)d\theta \\
+ \int\vartheta_-^2\big(\vert \mathcal{D}_+^kK_Jf(-\sqrt{2B},\theta)\vert^2 + \vert \mathcal{D}_-^kK_Jf(-\sqrt{2B},\theta)\vert^2\big)d\theta\Big).
\end{multline*}
Integrate in $B$. Parseval makes the left-hand norms decreasing in $B$; the Jacobians of the change of variables are uniformly positive. This controls the coefficient norm $\sum_{j,n}n^{j + 1/2}e^{-2A_1n}\big(\vert f_{+,j,n}\vert ^2 + \vert f_{-,j,n}\vert ^2\big)$. To obtain the asserted $\mathscr{H}$ norm, the synthesis operator from these normalized inputs has angular entries bounded by
$C_J e^{-(h - A_1 + r^2/2)n}e^{-(s - h)\vert m - \sigma n\vert }$ for $s>h$,
using its fixed coefficient strip. Both Schur sums are finite, also after integration in $r$. This proves the lemma and its extension by closure.
\end{proof}

\paragraph{The local observation estimate.}

Observed parabolic regularity controls the required derivatives of $U$, and the complementary estimates control those of $v$. Applying Lemma~\ref{common-height} to $U_J = U - v$ therefore handles possible cancellation between the two components on the observation set. This gives a weighted observation bound at any positive time. We then remove its loss by heat decay; choosing $M_\Phi<A<h<T$ leaves time for the preceding estimates.
\begin{prop}\label{weighted-observation}
For an exterior phase-graph collar with threshold $M_\Phi$, every $M_\Phi<h<h_*$ and every $\tau\in(0,T)$ satisfy
$\Vert e^{-hN}U(\tau)\Vert ^2 \leq C_{h,\tau,T}\obs$.
More generally the conclusion holds if some level $\mathcal{A} = A<h$ meets the observation set on every phase fiber as in Lemma~\ref{common-height} and the unobserved set lies in a sufficiently thin tube. For phase graphs, the allowed thickness is independent of $h - M_\Phi$.
\end{prop}
\begin{proof}
Fix the domains and constants $b_*,\gamma_*,s_*$ in the preceding construction. Choose the containing tube $\vert r\vert \leq d$ thin enough that
$4B_0<b_*$, $B_0<\frac{1}{4}\min(\gamma_*,s_*)$ and $d^2/2<h_*$.
For phase graphs this gives $M_\Phi<h_*$. Decrease $d$ once more so that $\vert r\vert \leq 3d$ lies where $\varkappa = 1$. The selected observed strips lie where $\chi = \varkappa = 1$ by the definition of $h_*$. If $B_0 = 0$, the rough estimate itself suffices. Fix $\varepsilon\in(0,1/4)$ and choose $0<t_0<\tau/2$. With $t_J = B_0/(2J + 3 - 2\varepsilon)$ as above, take $J$ large enough that $t_0 + 3t_J<\tau$. Fix the selected strips around a level $\mathcal{A} = A<h$, and enlarge the angular threshold as required by their Gram estimate. Restart the construction at $t_0$ with $K_J(t - t_0)$. Proposition~\ref{rough-state} controls the initial weighted complement and the residual forcing. The family is defined throughout $[t_0,t_0 + 3t_J]$ since $(2J + 3)3t_J<4B_0<b_*$.

The complementary estimates above control $v = (I - P)U$ in $\mathscr{H}$ from $t_0 + t_J$ onwards. On $I = (t_0 + 2t_J,t_0 + 5t_J/2)$, they control every required spatial derivative in $L^2(I\times\omega')$. The interval $I$ stays at distance at least $t_J/2$ from the endpoints of $(t_0+t_J,t_0+3t_J)$, leaving the positive time margins required by interior parabolic regularity. On the selected strips $U = u$. Apply Lemma~\ref{common-height} to $U_J = PU = U - v = K_Jf$, where $f = K_J^\dagger U\in\mathcal{F}_J$, and integrate:
$$
\int_I\Vert e^{-hN}U_J(t)\Vert ^2dt \leq C\obs, \qquad \int_I\Vert e^{-hN}U(t)\Vert ^2dt \leq C\obs.
$$
Select a controlled time in $I$ and propagate its negative-weight state to $\tau$ by Lemma~\ref{weighted-dissipation} at constant weight.
\end{proof}

\begin{prop}\label{local-upper}
For every $T>M_\Phi$, the localized actual state satisfies $\Vert U(T)\Vert ^2 \leq C_T\obs$. For a general geometry as in Lemma~\ref{common-height}, the same conclusion holds for $T>A$ when the level $\mathcal{A} = A<h_*$ meets the observation set on every phase fiber and the unobserved set lies in the sufficiently thin tube of Proposition~\ref{weighted-observation}.
\end{prop}
\begin{proof}
Choose $M_\Phi<h<\min(T,h_*)$ and $\varepsilon,\tau>0$ such that $\tau + h/(1 - \varepsilon)<T$. For a general geometry satisfying the level-set condition, replace $M_\Phi$ by $A$ in this choice. Proposition~\ref{weighted-observation} and Lemma~\ref{weighted-dissipation} give an unweighted cylinder bound at a time $s<T$. Propagate to the terminal endpoint using the cutoff identity
$$
\frac{d}{dt}\Vert \varkappa u\Vert ^2 = -2\langle \triangle_\Phi (\varkappa u),\varkappa u\rangle + 2\int c\vert \varkappa'\vert ^2\vert u\vert ^2 \leq C\Vert u\Vert_{L^2(\omega)}^2.
$$
The integral is supported where $\varkappa'$ is nonzero, hence in $\omega$. Integrating from $s$ to $T$ therefore uses only the observed $L^2$ mass and requires no estimate of spatial derivatives of $F_{\rm obs}$ at the terminal time.
\end{proof}

\subsection{Completion of the proof on the surface}\label{surface-global}

The local obstruction already uses global heat trajectories. For the upper bound, the estimates near the selected circles are combined with elliptic heat observability on the remaining regular region, over the same time interval.

Set $M_i = \max_\theta\min((\rho_i^+)^2,(\rho_i^-)^2)/2$, so $M_\Phi(\omega) = \max_iM_i$. Choose $r_*>0$ satisfying the fixed upper analytic restrictions, such that all collars lie where the construction cutoffs equal one and $r_*^2/2<\min_i h_{{\rm loc},i}$. The cutoff transition strips lie inside the fixed tubes. These choices depend only on the tubes and their analytic data, and precede the profiles and the observation time.

\begin{proof}[Proof of Theorem~\ref{thm:exact-time}]
\emph{Lower bound.} For $T<M_\Phi(\omega)$ choose $i$ with $T<M_i$. The fixed cutoff in that tube equals one on the unobserved collar. On the support of its derivative, $\mathcal{A}_i>h_{{\rm loc},i}>M_i$. Proposition~\ref{local-lower} therefore supplies global trajectories violating observation at time $T$. Since $M_\Phi(\omega) = \max_iM_i$, this proves $T_{\rm obs}(\omega) \geq M_\Phi(\omega)$.

\emph{Upper bound.} Fix $T>\max_i M_i$. Choose normal cutoffs $\varkappa_i$, supported in $\mathcal{C}_i$, equal to one near $K_i$, with transitions in observed regular strips. If $\mathcal{U}_i$ is the local coordinate and density unitary map, set $U_i = \varkappa_i\mathcal{U}_iu$. Proposition~\ref{local-upper}, with observed interior regularity for its cutoff sources, gives $\Vert U_i(T)\Vert_{L^2(\mathcal{C}_i)}^2\leq C_{i,T}\mathcal{O}_T(u)$.
The remaining singular set is a compact subset of $\omega$. Choose $0\leq\varkappa_s\leq1$ in $C_c^\infty(\omega)$, equal to one near it, and take $0\leq\varkappa_i\leq1$. The smooth function $\vartheta_0 = 1-(1-\varkappa_s)\prod_i(1-\varkappa_i)$ equals one near every singular point, and $\{0<\vartheta_0<1\}\subset\omega$. Choose regular values $0<\delta_0<\delta_1<1$ and a smooth function $\chi_e:[0,1]\to[0,1]$ equal to one on $[0,\delta_0/2]$ and to zero on $[\delta_0,1]$. Set $\varkappa_e = \chi_e(\vartheta_0)$ and $\Omega_e = \{\vartheta_0<\delta_1\}$. Then $\operatorname{supp}\varkappa_e\Subset\Omega_e$, the closure of $\Omega_e$ is contained in the regular region, and both $\partial\Omega_e$ and $\operatorname{supp}\nabla\varkappa_e$ lie in $\omega$. Compactness gives $\varkappa_e^2+\varkappa_s^2+\sum_i\varkappa_i^2\geq c_0>0$: wherever the last terms all vanish, $\varkappa_e=1$. Every connected component of $\Omega_e$ has nonempty boundary, since $M$ is connected and a nonempty singular neighbourhood has been removed; an interior collar of that boundary is observed. If $\Omega_e$ is empty, the elliptic contribution below is absent.

On $\Omega_e$, the function $y = \varkappa_eu$ satisfies Dirichlet boundary conditions and $\partial_ty+\triangle y = f = [\triangle,\varkappa_e]u$. Let $\triangle_{\Omega_e}$ be the nonnegative Dirichlet realization, which is uniformly elliptic. Fix $0<s<t_1<T$. Homogeneous heat observability holds on each component \cite{FI96,LR95}. To include the source, set $q(t)=\int_s^t e^{-(t-\tau)\triangle_{\Omega_e}}f(\tau)\,d\tau$. Contractivity gives $\sup_{s\leq t\leq t_1}\Vert q(t)\Vert^2\leq(t_1-s)\int_s^{t_1}\Vert f(t)\Vert^2dt$. Applying homogeneous observability to $y-q$ and bounding its terminal and observed $q$ terms yields
$$
\Vert y(t_1)\Vert^2 \leq C_{s,t_1,\Omega_e}\Big(\int_s^{t_1}\Vert y(t)\Vert_{L^2(\omega\cap\Omega_e)}^2dt+\int_s^{t_1}\Vert f(t)\Vert^2dt\Big).
$$
The source is supported in a compact observed regular set. Interior parabolic regularity, with time margins inside $(0,T)$, bounds its integral by $C\mathcal{O}_T(u)$. Hence $\Vert\varkappa_eu(t_1)\Vert^2 \leq C\mathcal{O}_T(u)$.

For every smooth real cutoff $\zeta$, the horizontal form product rule yields
$$
\frac{d}{dt}\Vert\zeta u(t)\Vert^2 = -2\mathcal{E}(\zeta u(t),\zeta u(t))+2\int_M\vert\nabla_H\zeta\vert^2\vert u(t)\vert^2d\nu.
$$
Applied to $\varkappa_e$, this propagates the bound from $t_1$ to $T$, since its gradient is supported in $\omega$. Applied to $\varkappa_s$, it gives the terminal bound after choosing an initial time with controlled observed mass. These identities extend to $L^2$ initial data by approximation. Multiplying $\varkappa_e^2+\varkappa_s^2+\sum_i\varkappa_i^2 \geq c_0$ by $\vert u(T)\vert^2$ and integrating over $M$, the preceding estimates and the inequality $\Vert U_i(T)\Vert_{L^2(\mathcal{C}_i)}^2\leq C_{i,T}\mathcal{O}_T(u)$ now prove observability.
%
\end{proof}

\subsection{Beyond exterior phase collars}\label{sec:general-observation}

We prove Corollary~\ref{cor:general-bounds} using the same fixed tubes and cutoffs as Theorem~\ref{thm:exact-time}.

\begin{proof}[Proof of Corollary~\ref{cor:general-bounds}]
For every observed point in tube $i$, $\underline m_i(\Theta_i(p))\leq\mathcal{A}_i(p)$. Also $\sup\underline m_i\leq d_i^2/2<h_{{\rm loc},i}$, so Proposition~\ref{local-lower} gives the lower bound.
The observed annuli give $\mathfrak{A}(\omega)\leq\max_i d_i^2/2$. Decrease the fixed $r_*$ so that $r_*^2/2<\min_i h_{*,i}$. For $T>\mathfrak{A}(\omega)$, we can therefore choose in every tube a common observed level $A_i<\min(T,h_{*,i})$. Proposition~\ref{local-upper} and the surface argument of Subsection~\ref{surface-global} give observability in time $T$.
For an exterior collar, $m_i=\underline m_i=\min((\rho_i^+)^2,(\rho_i^-)^2)/2$. Its maximum is the infimum of common observed levels, proving the last assertion.
\end{proof}

These bounds separate two questions. The obstruction uses the first observed value of $\mathcal{A}_i$ at each phase angle. The upper estimate requires a single level of $\mathcal{A}_i$ observed on at least one side at every angle. Observation sets with holes can give different bounds. Within this class, it remains to determine which observation sets satisfy $\mathfrak{m}(\omega)=\mathfrak{A}(\omega)$ beyond exterior collars. The full observed annuli are still required; the corollary gives no criterion for arbitrary open subsets of the surface.

\section{Proof of Theorem B}\label{dyn:section}

The proof uses a weighted energy estimate on a periodic tube and states concentrated near its central circle. At higher step, their observed signal decays faster than their terminal heat mass at every fixed time. Subsection~\ref{dyn:completion} transfers this obstruction to the surface and completes the proof. Subsection~\ref{dyn:extensions} then gives the step-two and Gevrey variants, under weaker regularity assumptions.

\subsection{Weighted heat evolution}

Fix a periodic tube with orthonormal frame $X_1 = \partial_x$, $X_2 = x^me^{\varphi(x,y)}\partial_y$, where $m\in\N$, $m \geq 1$ and $\varphi(0,y) = 0$. Set $r = m + 1$ and write $\dd{d}\nu = \rho\,\dd{d}x\,\dd{d}y$. We require $\varphi$ to be real and $C^2$ in $x$ with values in $C^0(\T_\ell)$, and the conjugated potential below to be real and continuous. Density conjugation gives
$\triangle_{x,m} = -\partial_x^2 + DQ_mD + V_m$ with $D = -i\,\partial_y$, $Q_m = x^{2m}e^{2\varphi}$ and $V_m = \rho^{-1/2}\big(\partial_x^2\rho^{1/2} + \partial_y(Q_m\partial_y\rho^{1/2})\big)$.

For $0<L<r_0$, let $\triangle_L$ be its nonnegative Dirichlet realization in $\mathcal{H}_L = L^2(\Omega_L,\dd{d}x\,\dd{d}y)$, where $\Omega_L = (-L,L)\times\T_\ell$. Its form $\mathfrak{q}_L$ is $\int_{\Omega_L}(\vert \partial_xw\vert ^2 + Q_m\vert Dw\vert ^2 + V_m\vert w\vert ^2)$, with domain $\mathcal{V}_L$ the completion of $C_c^\infty((-L,L);C^\infty(\T_\ell))$ for the norm
\begin{equation}\label{dyn:form-norm}
\big(\Vert w\Vert ^2 + \Vert \partial_xw\Vert ^2 + \Vert x^mDw\Vert ^2\big)^{1/2}.
\end{equation}
All local norms below are in $\mathcal{H}_L$. The argument also applies to bounded real potentials when $\triangle_L$ has a fixed lower bound.

Let $\mathcal{H}^\infty_\ell(S_\sigma)$ be the bounded periodic holomorphic functions on $S_\sigma = \{y\in\C/\ell\Z\,\mid\, \vert \im y\vert <\sigma\}$.
\begin{hyp}[Tangential analyticity]\label{dyn:regularity}
In the prescribed tube coordinates, we assume that, for some $\sigma_0>0$,
$\varphi\in C^2([-r_0,r_0];\mathcal{H}^\infty_\ell(S_{\sigma_0}))$ and $V_m\in C^0([-r_0,r_0];\mathcal{H}^\infty_\ell(S_{\sigma_0}))$.
The coefficients are real on the real tube and $\varphi(0,y) = 0$.
\end{hyp}

Shrinking the normal interval gives
$\re e^{2\varphi(x,y)} \geq c_0>0$ when $\vert x\vert \leq r_0$ and $\vert \im y\vert <\sigma_0$,
since $\varphi(0,\cdot) = 0$ on the strip and $\Vert \varphi(x,\cdot)\Vert = \mathrm{O}(\vert x\vert )$. Joint analyticity implies Assumption~\ref{dyn:regularity}, but analyticity in $x$ is unnecessary. For example, one may take $\rho = 1$ and $\varphi(x,y) = a(x)\sin(2\pi y/\ell)$ with $a$ smooth and $a(0) = 0$. Tangential analyticity concerns these coordinates and need not survive an arbitrary smooth change of variables. The analogous $C^2$ strip regularity of a positive density $\rho$ suffices for the stated regularity of $V_m$, on a smaller strip.

Write $\alpha = 2\pi/\ell$ and set $e_\kappa(y) = \ell^{-1/2}e^{i\kappa y}$. All tangential multipliers are defined in this Fourier basis. For a Lipschitz function $\mu:[-L,L]\to[-\eta,0]$, where $0<\eta<\sigma_0$, the operator $B_\mu = e^{\mu(x)\vert D\vert }$ is a contraction on $\mathcal{H}_L$. We choose $\mu$ to satisfy
\begin{equation}\label{dyn:eikonal}
\mu'(x)^2 \leq \re Q_m(x,y\mp i\mu(x)) \qquad\text{for almost every }x\text{ and every }y\in\T_\ell.
\end{equation}
This is the normal exponential-weight inequality
\cite{Ag82,Dav89}, with imaginary translations accounting for the
coupled tangential coefficient.

\begin{lem}\label{dyn:weighted-form}
For $w\in\mathcal{V}_L\cap D(\vert D\vert )$, let $W = B_\mu w$ and let $W_\pm$ be its positive and negative Fourier parts. Then $B_\mu^2w\in\mathcal{V}_L$ and
\begin{multline}\label{dyn:weighted-form-bound}
\re\mathfrak{q}_L(w,B_\mu^2w)
\geq \Vert \partial_xW\Vert ^2 - C_\eta\Vert W\Vert ^2 + \int_{\Omega_L} \big(\re Q_m(x,y-i\mu) - \mu'^2\big)\vert DW_+\vert ^2 \\
+ \int_{\Omega_L} \big(\re Q_m(x,y+i\mu) - \mu'^2\big)\vert DW_-\vert ^2.
\end{multline}
\end{lem}

\begin{proof}
For finite Fourier sums, $\partial_xB_\mu = \mu'\vert D\vert B_\mu$ gives the normal contribution $\Vert \partial_xW\Vert ^2 - \Vert \mu'\vert D\vert W\Vert ^2$. The tangential matrix is
$\kappa\kappa'\widehat Q_{m,\kappa - \kappa'} e^{\mu(\vert \kappa\vert - \vert \kappa'\vert )}$.
On either same-sign block this is multiplication by $Q_m(x,y\mp i\mu)$ between the corresponding derivatives. On opposite-sign blocks, $\vert \kappa\kappa'\vert \leq \vert \kappa - \kappa'\vert ^2/4$. Contour translation bounds the Fourier coefficients of $Q_m,V_m$ by $Ce^{-\sigma'\vert j\vert }$ for any $\eta<\sigma'<\sigma_0$. Lemma~\ref{app:block-calculus} bounds the cross-sign and potential matrices on $L^2$, giving the remainder $-C_\eta\Vert W\Vert ^2$. Fourier truncation converges in the form norm and in $D(\vert D\vert )$. The shifted coefficients are bounded by $C\vert x\vert ^{2m}$, so the identity passes to the stated domain.
\end{proof}

\begin{prop}\label{dyn:weighted-evolution}
If \eqref{dyn:eikonal} holds, then for $w_0\in\mathcal{H}_L$ and $W(t) = B_\mu e^{-t\triangle_L}w_0$,
\begin{equation}\label{dyn:weighted-evolution-bound}
\Vert W(t)\Vert ^2 + 2\int_0^t\Vert \partial_xW(s)\Vert ^2\,\dd{d}s \leq e^{2C_\eta t}\Vert B_\mu w_0\Vert ^2.
\end{equation}
\end{prop}

\begin{proof}
Let $w_N$ be the heat solution for the closed form restricted to $\vert \kappa\vert \leq N$, with initial value the Fourier truncation of $w_0$. The test $B_\mu^2w_N$ remains in that subspace, and
\eqref{dyn:weighted-form-bound} gives
$\frac{1}{2}\partial_t\Vert B_\mu w_N\Vert ^2+
\Vert \partial_x(B_\mu w_N)\Vert ^2 \leq C_\eta\Vert B_\mu w_N\Vert ^2$.
Integration proves the estimate uniformly in $N$. Fourier truncations are dense in \eqref{dyn:form-norm}, so convergence of the restricted closed forms gives strong convergence of their heat semigroups in $\mathcal{H}_L$. Boundedness of $B_\mu$ and weak lower semicontinuity of the normal derivative pass the estimate to the limit. Strong continuity supplies $t = 0$ for arbitrary initial data.
\end{proof}

\subsection{Concentration and the local higher-step obstruction}

The following localization argument will also be used with $e^{\mu(x)\vert D\vert ^\beta}$, $0<\beta<1$. The weight is small near the singular circle and equals one before the unobserved tube ends. Its evolution estimate controls the error from cutting the local solution off there.

Fix a nonnegative $C^1$ function $F$ on $[-L,L]$ satisfying $F(x) \geq c\vert x\vert ^r$, and, near zero, $
F(x) = \frac{\vert x\vert ^r}{r}+\mathrm{O}(\vert x\vert ^{r + 1})$ and $F'(x) = \operatorname{sgn}(x)\vert x\vert ^m+\mathrm{O}(\vert x\vert ^{m + 1})$.
Suppose $\omega$ misses $(-a_-,a_+)\times\T_\ell$, where $\max(a_+,a_ -)<L$. Choose $\vartheta\in(0,\min(F(a_+),F(-a_-)))$ and set $\mu = \min(F,\vartheta) - \vartheta$.

\begin{lem}\label{dyn:concentrated-states}
With $F,\mu,\vartheta$ as above, let $0<\beta \leq 1$. If
\eqref{dyn:weighted-evolution-bound} holds with
$B_\mu = e^{\mu(x)\vert D\vert ^\beta}$, then for positive $k\in\alpha\Z$ tending to infinity there are unit vectors $u_{0,k}$ in the form domain of $\triangle$, supported in the unobserved tube, such that
\begin{equation}\label{dyn:global-leak}
\int_0^T\Vert e^{-t\triangle}u_{0,k}\Vert_{L^2(\omega,\nu)}^2\,\dd{d}t \leq C_T k^{1/r}e^{-2\vartheta k^\beta}, \qquad \mathcal{E}(u_{0,k}) \leq C_E k^{2/r}.
\end{equation}
For $m = 1$, their energy is $k+\mathrm{o}(k)$.
\end{lem}

\begin{proof}
Choose real $\chi_0,\chi_1\in C_c^\infty((-a_-,a_+))$ with $\chi_0 = 1$ near zero, $\chi_1 = 1$ near $\supp\chi_0$, and $\supp\chi_1'$ contained in the interior of $\{\mu = 0\}$. For $k\in\alpha\Z$, $k\to + \infty$, normalize $w_{0,k} = c_k\chi_0e^{-kF}e_k$. Rescaling $x = k^{-1/r}s$ gives $c_k^2\asymp k^{1/r}$ and $\int\vert x\vert ^{2m}\vert w_{0,k}\vert ^2 = \mathrm{O}(k^{-2m/r})$. The cutoff terms are exponentially small. Since $\vert F'\vert \leq C\vert x\vert ^m$ on its support, $Q_m \leq C\vert x\vert ^{2m}$ and $V_m$ is bounded, the normal and tangential energies are $\mathrm{O}(k^{2/r})$. When $m = 1$, $F'^2 + \overline Q_1 = 2x^2+\mathrm{O}(\vert x\vert ^3)$, where the bar denotes the tangential average. Gaussian rescaling gives $k\int x^2\vert w_{0,k}\vert ^2\to1/2$, hence $\mathfrak{q}_L(w_{0,k},w_{0,k}) = k+\mathrm{o}(k)$.
For $k \geq 1$, $k \geq k^\beta$ and $\mu k^\beta - kF \leq (\mu - F)k^\beta \leq - \vartheta k^\beta$. The assumed evolution estimate consequently yields
\begin{equation}\label{dyn:local-leak}
\sup_{0 \leq t \leq T}\Vert B_\mu w_k(t)\Vert ^2 + \int_0^T\Vert \partial_x(B_\mu w_k(t))\Vert ^2\,\dd{d}t \leq C_Tk^{1/r}e^{-2\vartheta k^\beta}, \qquad w_k(t) = e^{-t\triangle_L}w_{0,k}.
\end{equation}

Extend $u_{0,k} = \rho^{-1/2}w_{0,k}$ and $\widetilde u_k = \rho^{-1/2}\chi_1w_k$ by zero to $M$. The first is a unit form-domain vector with the computed energy. The second has the same initial value and satisfies, weakly,
$$
(\partial_t + \triangle)\widetilde u_k = f_k = -\rho^{-1/2}(\chi_1''w_k + 2\chi_1'\partial_xw_k).
$$
There is no boundary distribution since $\chi_1$ vanishes near $\partial\Omega_L$. On the support of its derivatives, $\mu = \mu' = 0$, so \eqref{dyn:local-leak} gives $\Vert f_k\Vert_{L^2((0,T);L^2(M,\nu))} \leq C_Tk^{1/(2r)}e^{-\vartheta k^\beta}$.
Duhamel's formula and contractivity of $e^{-t\triangle}$ give
\eqref{dyn:global-leak}, since $\widetilde u_k$ vanishes on $\omega$.
\end{proof}

For these unit states, Jensen's inequality for the spectral measure of $\triangle$ gives
\begin{equation}\label{dyn:jensen}
\Vert e^{-T\triangle}u_{0,k}\Vert ^2 \geq e^{-2T\mathcal{E}(u_{0,k})}.
\end{equation}
For each fixed $T>0$, the observed integral is bounded above by $C_Tk^{1/r}e^{-2\vartheta k^\beta}$, whereas the terminal squared norm is bounded below by $e^{-2C_ETk^{2/r}}$. If $\beta>2/r$, the ratio of the former to the latter tends to zero as $k\to+\infty$, which contradicts observability.

Under Assumption~\ref{dyn:regularity}, define
$$
F_{m,\eta}(x) = \Big\vert \int_0^x \vert s\vert ^m \inf_{\substack{y\in\T_\ell\\ \vert v\vert \leq \eta}} \sqrt{\re e^{2\varphi(s,y + iv)}}\,\dd{d}s\Big\vert , \qquad 0<\eta<\sigma_0.
$$
The infimum is a positive Lipschitz function of $s$. Hence $F_{m,\eta}\in C^1$, is strictly increasing away from zero on each side, and
$F_{m,\eta}(x) = \frac{\vert x\vert ^r}{r}+\mathrm{O}(\vert x\vert ^{r + 1})$ and $F_{m,\eta}(x) \geq c\vert x\vert ^r$.
Its derivative has the expansion required in Lemma~\ref{dyn:concentrated-states}.

\begin{thm}\label{dyn:higher-order}
Let $m \geq 2$ and suppose Assumption~\ref{dyn:regularity} holds. If $\omega$ misses $(-a_-,a_+)\times\T_\ell$ for some $0<a_\pm<r_0$, then $\Tobs(\omega) = +\infty$.
\end{thm}

\begin{proof}
Choose $0<\vartheta<\min(\eta,F_{m,\eta}(a_+), F_{m,\eta}(-a_-))$. The associated weight satisfies
\eqref{dyn:eikonal}. Lemma~\ref{dyn:concentrated-states} and \eqref{dyn:jensen} give an observed norm squared at most $C_Tk^{1/r}e^{-2\vartheta k}$ and a final norm squared at least $e^{-2C_ETk^{2/r}}$. Since $2/r<1$, the ratio of terminal mass to observation tends to $+\infty$ for every fixed $T>0$.
\end{proof}

\subsection{Completion of the proof}\label{dyn:completion}
To deduce Theorem~\ref{thm:higher-obstruction}, we justify the constant-step normal form stated in Lemma~\ref{geom:higher-normal-form}.

\begin{proof}[Proof of Lemma~\ref{geom:higher-normal-form}]
Set $m = r - 1$. Straighten a local nonzero horizontal field and write generators of the horizontal module as $X = \partial_u$ and $Y = g\partial_v$. At a point of $Z_0$, the step is one plus the least $j \geq 1$ for which $\partial_u^jg\ne0$. Indeed, brackets with $X$ successively differentiate $g$, while other brackets have coefficients in the ideal generated by the preceding derivatives.

The points where the horizontal line is transverse to $Z_0$ are dense: otherwise every horizontal field and its brackets would be tangent to an open arc, contradicting bracket generation. At a transverse point, constant step gives determinant vanishing of normal order $m$. For a local defining function $\sigma$ of $Z_0$ and a determinant $\delta_f$ of $f$, the normal derivatives of order less than $m$ therefore vanish on a dense subset of $Z_0$, hence everywhere. Smooth division gives $\delta_f = \sigma^m b$. If $X$ were tangent at $p\in Z_0$, its integral curve through $p$ would satisfy $\sigma = \mathrm{O}(t^2)$. Since $g$ differs from $\delta_f$ by a nonvanishing factor, this would give $g = \mathrm{O}(t^{2m})$ and $X^jg(p) = 0$ for $j \leq m$, contrary to step $m + 1$. Thus transversality holds everywhere, and the step condition gives $b\ne0$ on $Z_0$. The punctured tube is regular after shrinking it.

The conormal-geodesic construction used in Section~\ref{sec:geometry} applies once transversality is known: the conormal flow is noncharacteristic, independently of the determinant order. The Gauss lemma gives a signed coordinate $x$ with $\vert f^*dx\vert_E = 1$ and $ff^*dx = \partial_x$, where $f^*$ is the adjoint for the bundle metric. Thus $e_1 = f^*dx$ is a global unit section of $E$. The bundle $E$ is orientable on the collar: on a parallel regular circle it is identified by $f$ with the tangent bundle of the annulus, and the collar retracts onto that circle. Complete $e_1$ to an orthonormal frame $(e_1,e_2)$. Then $dx(f(e_2)) = 0$, and the determinant order gives $f(e_2) = x^ma(x,\theta)\partial_\theta$ with $a>0$ after a sign choice. The reparametrization $y(\theta) = \int_0^\theta a(0,s)^{-1}\,\dd{d}s$ replaces $a$ by $a(x,\theta)/a(0,\theta)$, whose logarithm is $\varphi$. The distance argument of \eqref{geom:distance} applies without change.

For analytic data, choose analytic $X$ and $\delta_f$. The function $X^{m - 1}\delta_f$ vanishes on $Z_0$ and has nonzero differential there, since $X^m\delta_f\ne0$. Hence $Z_0$ is analytic. The conormal flow, the oriented orthonormal frame, division by $x^m$ and tangential reparametrization are then analytic.
\end{proof}

\begin{proof}[Proof of Theorem~\ref{thm:higher-obstruction}]
Lemma~\ref{geom:higher-normal-form} supplies an analytic periodic tube with $m = r - 1 \geq 2$. The analytic density gives the required strip regularity of $V_m$ after shrinking the tube. Positive separation from the circle leaves a smaller tube unobserved. Apply Theorem~\ref{dyn:higher-order}.
\end{proof}

\subsection{Obstructions under weaker regularity}\label{dyn:extensions}

The same energy estimates apply without joint analyticity. We record their consequences at step two and for tangential Gevrey coefficients at higher step.

\paragraph{Step two under tangential analyticity.}
For $m = 1$, write $F_\eta = F_{1,\eta}$.

\begin{thm}\label{dyn:lower-bound}
Let $m = 1$ and suppose Assumption~\ref{dyn:regularity} holds. If $\omega$ misses $(-a_-,a_+)\times\T_\ell$, with $0<a_\pm<r_0$, then
$\displaystyle \Tobs(\omega) \geq \sup_{0<\eta<\sigma_0}\min(\eta,F_\eta(a_+),F_\eta(-a_-))$.
\end{thm}

\begin{proof}
Fix $\eta$ and $0<\vartheta<\min(\eta,F_\eta(a_+), F_\eta(-a_-))$. The weight $\mu = \min(F_\eta,\vartheta) - \vartheta$ satisfies
\eqref{dyn:eikonal}, so Proposition~\ref{dyn:weighted-evolution}
and Lemma~\ref{dyn:concentrated-states}, with $\beta = 1$, apply. Observability would imply $e^{-2T(k+\mathrm{o}(k))} \leq C_Tk^{1/2}e^{-2\vartheta k}$. Taking logarithms and dividing by $2k$ gives $T \geq \vartheta$. Taking the supremum over $\vartheta$ and $\eta$ proves the result.
\end{proof}

For a symmetric gap of sufficiently small width $a$, set $b(y) = \partial_x\varphi(0,y)$ and choose $\eta = a^2$. On a smaller fixed strip, $\varphi(x,y) = xb(y)+\mathrm{O}(x^2)$ and $\re b(y + iv) = b(y)+\mathrm{O}(v^2)$ uniformly. Thus $F_{a^2}(a) \geq a^2/2 + (\min b)a^3/3 - Ca^4$ and $F_{a^2}(-a) \geq a^2/2 - (\max b)a^3/3 - Ca^4$, giving
$\Tobs(\omega) \geq \frac{a^2}{2} - \frac{\Vert b\Vert_\infty}{3}a^3 - Ca^4$.
The sharper coefficient $B_{\min}$ in Corollary~\ref{cor:distance-time} uses holomorphic angular localization under joint analyticity. The present $F_\eta$ provides a lower bound under weaker regularity, rather than an exact-time formula.

\paragraph{Tangential Gevrey coefficients.}
At higher step, the same comparison permits subexponential Fourier weights. The following theorem states the regularity in the prescribed normalized coordinates.

\begin{thm}\label{dyn:gevrey}
Let $m \geq 2$ and suppose $a = e^{2\varphi}$ and $V_m$ satisfy uniform tangential Gevrey bounds of order $1<s \leq r/2$:
$$
\sup_{\vert x\vert \leq r_0} \big(\Vert \partial_y^na(x,\cdot)\Vert_\infty + \Vert \partial_y^nV_m(x,\cdot)\Vert_\infty\big) \leq CR^n(n!)^s\qquad\forall n \geq 0.
$$
If $\omega$ misses $(-a_-,a_+)\times\T_\ell$ for some $0<a_\pm<r_0$, then $\Tobs(\omega) = +\infty$ for $s<r/2$ and $\Tobs(\omega)>0$ for $s = r/2$.
\end{thm}

\begin{proof}
Integration by parts and optimization in $n$ give Fourier decay $Ce^{-c\vert j\vert ^{1/s}}$. In fact, the proof only needs the more general condition
\begin{equation}\label{dyn:gevrey-fourier}
\sup_{\vert x\vert \leq r_0}\sum_{j\in\alpha\Z} \big((1 + \vert j\vert )^2\vert \widehat a_j(x)\vert + \vert \widehat{V_m}_j(x)\vert \big) e^{\eta_0\vert j\vert ^\beta}<+\infty, \qquad 2/r \leq \beta<1,
\end{equation}
which gives all-time failure when $\beta>2/r$ and a positive-time obstruction when $\beta = 2/r$. The Gevrey hypotheses imply this with $\beta = 1/s$ and sufficiently small $\eta_0>0$.

Fix $0<\eta<\eta_0$ and a Lipschitz $\mu\in[-\eta,0]$ with $\vert \mu'\vert \leq C_\mu\vert x\vert ^m$. Set $B_\mu = e^{\mu(x)\vert D\vert ^\beta}$ and $W = B_\mu w$. The normal form contribution is $\Vert \partial_xW\Vert ^2 - \Vert \mu'\vert D\vert ^\beta W\Vert ^2$. At fixed $x$, reality of $a$ shows that the Hermitian part of the tangential matrix differs from the unweighted form by $x^{2m}R_x$, where
$(R_x)_{\kappa,\kappa'} = \kappa\kappa'\widehat a_{\kappa - \kappa'} \left(\cosh\left(\mu(\vert \kappa\vert ^\beta - \vert \kappa'\vert ^\beta)\right) - 1\right)$.
We estimate this error relative to $\langle D\rangle^{2\beta}$, where $\langle D\rangle = (1 + D^2)^{1/2}$. Set $j = \kappa - \kappa'$ and $K = \max(\vert \kappa\vert ,\vert \kappa'\vert)$. If $\vert j\vert \leq K/2$ and $K>0$, the frequencies have the same sign and comparable magnitudes, so
$\vert \vert \kappa\vert ^\beta - \vert \kappa'\vert ^\beta \vert \leq C\vert j\vert K^{\beta - 1}$.
Using $\cosh z - 1 \leq z^2e^{\vert z\vert }/2$, this power of $K$ cancels the frequency growth after conjugation by $\langle D\rangle^{-\beta}$. If $\vert j\vert >K/2$, both frequencies are bounded by $2\vert j\vert $, and concavity gives $\big\vert \vert \kappa\vert ^\beta - \vert \kappa'\vert ^\beta\big\vert \leq \vert j\vert ^\beta$. In both cases,
$\left\vert (\langle D\rangle^{-\beta}R_x \langle D\rangle^{-\beta})_{\kappa,\kappa'}\right\vert \leq C_\eta(1 + \vert j\vert )^2\vert \widehat a_j(x)\vert e^{\eta\vert j\vert ^\beta}$.
Lemma~\ref{app:block-calculus} and \eqref{dyn:gevrey-fourier} imply $\vert \langle R_xW,W\rangle\vert \leq C_\eta
\Vert \langle D\rangle^\beta W\Vert ^2$. The weighted potential matrix
is bounded on $L^2$ by the same Fourier condition.

Since $a \geq c_0>0$ on the real tube, the tangential remainder and normal cost are bounded by $C\Vert x^m\langle D\rangle^\beta W\Vert ^2$. The inequality $\langle\kappa\rangle^{2\beta} \leq \varepsilon\kappa^2 + C_\varepsilon$, valid for $\beta<1$, absorbs them into the principal form and gives
$\re\mathfrak{q}_L(w,B_\mu^2w) \geq \Vert \partial_xW\Vert ^2 + \frac{c_0}{2}\Vert x^mDW\Vert ^2 - C\Vert W\Vert ^2$.
The bound on $\mu'$ makes $B_\mu$ and $B_\mu^2$ bounded on the form domain \eqref{dyn:form-norm}. Fourier truncation and the Galerkin argument of Proposition~\ref{dyn:weighted-evolution} therefore give
\eqref{dyn:weighted-evolution-bound} for this weight.

Take $F(x) = \vert x\vert ^r/r$ and $0<\vartheta<\min(\eta,F(a_+),F(-a_-))$. Its associated weight satisfies $\vert \mu'\vert \leq \vert x\vert ^m$. Lemma~\ref{dyn:concentrated-states} and \eqref{dyn:jensen} reduce observability to $e^{-2C_ETk^{2/r}} \leq C_Tk^{1/r}e^{-2\vartheta k^\beta}$.
This is impossible for every fixed $T>0$ when $\beta>2/r$. At equality it implies $T \geq \vartheta/C_E>0$, with the weight and energy constant fixed before $k\to + \infty$.
\end{proof}

At step two, tangential analyticity is the limiting regularity in this comparison between observation and terminal mass. At higher step, the strict inequality $s<r/2$ excludes every observation time, while equality only yields a positive lower bound. These arguments give no converse under the weaker regularity hypotheses.

It remains open here whether a smooth coupled structure on a closed Baouendi-Grushin surface can have $\Tobs(\omega) = 0$ when $\omega$ has positive distance from an entire simple singular circle. Separated models and the preceding estimates show that joint analyticity is unnecessary for an obstruction. At step two, however, tangential Gevrey regularity of order $s>1$ gives only stretched exponential Fourier decay, insufficient for the present comparison with heat dissipation. Failure of these estimates alone does not establish observability in the smooth setting.

\section{Periodic spectral quasimodes}\label{wkb:section}

At step two, the complex phase also gives stationary quasimodes concentrated near the singular circle. We use the semiclassical parameter $h = \frac{1}{k}\to0$, where $k$ is the tangential momentum. We compute their energy corrections and limiting mass, then derive consequences for the spectrum of $\triangle$ and Schr\"odinger observation.

Complex phases and factorial estimates are standard in analytic microlocal analysis \cite{HS18,Sj82}. Exponentially accurate WKB constructions include \cite{BR20,Del23,KLS22}, with transport vanishing at an isolated well in \cite[Proposition~3.4]{Del23}. Periodic wells are studied in \cite{DV21}, and separated Baouendi-Grushin quasimodes in \cite{AS23}. Here the transport vanishes on the whole singular circle. A holomorphic normal coordinate reduces it to an Euler operator; periodic compatibility fixes the traces and real energy corrections.

\begin{thm}[Periodic spectral quasimodes]\label{thm:periodic-quasimodes}
Let $Z$ be a two-sided compact connected singular component of rank one and constant step two, with analytic structure and density near $Z$. For $h>0$, set $P_h = h^2\triangle$ and let $\Phi$ be the phase in \eqref{phase-definition}, with period $\ell$. There are periodic holomorphic amplitudes $a_j$ and real coefficients $\mu_j$, with $\mu_0 = 1$ and $a_0(0,y)\ne0$, having the following properties. Along $h = h_n = \ell/(2\pi n)$, $n\to + \infty$, and for every fixed smooth normal cutoff $\chi$, compactly supported in a sufficiently small interval and equal to one near $Z$, one can choose $N(h)\asymp h^{-1}$ and $c_h>0$ such that
$$
u_h = c_h\chi(x)e^{i\Phi(x,y)/h}\sum_{j = 0}^{N(h)}h^ja_j(x,y), \qquad E(h) = h\sum_{j = 0}^{N(h)}h^j\mu_j = h+\mathrm{O}(h^2).
$$
Extended by zero, $u_h\in D(\triangle)$ and, for some $C,\beta>0$,
$$
\Vert u_h\Vert_{L^2(M,\nu)} = 1, \qquad\Vert (P_h - E(h))u_h\Vert_{L^2(M,\nu)} \leq Ce^{-\beta/h}.
$$
The probability measures $\vert u_h\vert ^2\nu$ converge weakly to the normalized measure $\dd{d}y/\ell$ on $Z$.
\end{thm}

\begin{proof}
We solve the periodic transport equations, truncate their factorial expansion and compute the limiting mass.

\smallskip
\noindent
\emph{Step 1. Phase and normal transport.} Use the analytic tube \eqref{geom:normal-form}, write $\dd{d}\nu = \rho\,\dd{d}x\,\dd{d}y$ and set
$Q = q^2 = x^2e^{2\varphi}$ and $\mathcal{D}_\rho = \rho^{-1}\partial_x(\rho\partial_x) + \rho^{-1}\partial_y(\rho Q\partial_y)$.
The operator $P_h$ acts in $L^2(M,\nu)$ with domain $D(\triangle)$, and $\triangle = -\mathcal{D}_\rho$ on smooth functions supported in the tube. All quasimode norms are taken in this space.

The principal symbol is $p(x,y;\xi,\eta) = \xi^2 + Q(x,y)\eta^2$. Its zero-energy eikonal equation is \eqref{phase-definition}. The phase satisfies $\Phi(x,y + \ell) = \Phi(x,y) + \ell$, so $e^{i\Phi/h}$ is periodic along $h = \ell/(2\pi n)$, $n\to + \infty$. By \eqref{geom:phase-expansion}, uniformly in $y$ and for bounded $t$,
\begin{equation}\label{wkb:positive-phase}
e^{i\Phi(\sqrt h\,t,y)/h} = e^{iy/h}e^{-t^2/2}(1 + \mathrm{O}(\sqrt h)),\qquad \im\Phi(x,y) \geq cx^2
\end{equation}
for some $c>0$ on a sufficiently small real tube.

The phase coordinate $z = \Phi$ used in Subsection~\ref{sec:sharp-lower} transforms the principal symbol into
$$
\widetilde p = (\xi')^2 + 2\widetilde{\Phi_x}\xi'\zeta = \xi'(\xi'+ix\mathcal{B}\zeta),\qquad \mathcal{B} = \widetilde{2e^\varphi\Phi_y},\qquad\mathcal{B}(0,z) = 2.
$$
A tilde denotes the change of coordinates. The eikonal equation removes the $\zeta^2$ term. This identity reduces the transport equations to a normal equation with a free tangential trace.

Set $\mathcal{L} = -i(2\Phi_x\partial_x + 2Q\Phi_y\partial_y+
\mathcal{D}_\rho\Phi) - 1$. Substitution of
$a\sim\sum_{n \geq 0}h^na_n$ and $E(h)\sim h\sum_{j \geq 0}h^j\mu_j$ gives
\begin{equation}\label{wkb:hierarchy}
\mathcal{L}a_0 = 0,\qquad \mathcal{L}a_n = \mathcal{D}_\rho a_{n - 1} + \sum_{j = 1}^n\mu_j a_{n - j}\quad\forall n\geq1, \qquad\mu_0 = 1.
\end{equation}
The value $\mu_0 = 1$ follows from $\mathcal{D}_\rho\Phi(0,y) = i$ and the nonzero leading trace.

On a fixed product $\mathcal{U}_{R_0,s_0} = \{\vert x\vert <R_0,\ \vert \Im z\vert <s_0\}$, with $z\in\mathbb{C}/\ell\Z$, let $\widetilde\rho$ be the scalar coefficient $\rho$ composed with the inverse of $(x,y)\mapsto(x,\Phi(x,y))$. Shrink the product once so that $\widetilde\rho$ is nonvanishing. The density conjugation $\mathcal{D}_\rho = \rho^{-1/2}(\mathcal{D}_0 - V)\rho^{1/2}$ gives $C_\rho = C + \frac{x\mathcal{B}}{2}\partial_x\log\widetilde\rho$, where the derivative is taken at fixed $z$, and $C,g$ are defined in Subsection~\ref{sec:sharp-lower} and \eqref{lower-gauge}. Thus
\begin{equation}\label{wkb:gauge}
\widetilde{\mathcal{L}} = x\mathcal{B}\partial_x + C_\rho, \qquad g_\rho(x,z) = \left(\frac{\widetilde\rho(0,z)}{\widetilde\rho(x,z)}\right)^{1/2}g(x,z), \qquad g_\rho^{-1}\widetilde{\mathcal{L}}g_\rho = x\mathcal{B}\partial_x.
\end{equation}
The square root is the holomorphic branch equal to one at $x = 0$. In particular, $g_\rho(0,z) = 1$ and $g_\rho$ is nonvanishing on this product. The transport source must vanish at $x = 0$. After division by $\mathcal{B}$, the zero-trace inverse of $x\partial_x$ divides the coefficient of $x^k$ by $k$. The normal operator has kernel consisting of functions of $z$. Appendix~\ref{app:wkb} gives the solution formula.

\smallskip
\noindent
\emph{Step 2. Periodic compatibility.} We use the trace calculation of Appendix~\ref{app:wkb}, which identifies the effective tangential operator and the leading periodic amplitude. For a periodic holomorphic trace $v$ on a complex strip, let $\mathsf{E}v$ be the homogeneous solution whose expression in the new coordinates is $g_\rho(x,z)v(z)$. The trace of $\mathcal{D}_\rho\mathsf{E}v$ determines the next compatibility equation:
$$
(\mathcal{D}_\rho\mathsf{E}v)\vert_{x = 0} = H_{\mathrm{eff}}v = iv'+\frac{i}{2}(\partial_y\log\rho_0)v + \mathcal{V}v,\qquad\rho_0(y) = \rho(0,y),
$$
Here $H_{\mathrm{eff}}$ acts in $L^2(\T_\ell,\rho_0\,\dd{d}y)$ with domain $H^1(\T_\ell)$, and $\mathcal{V}$ is the real potential in \eqref{wkb:effective-operator}. The equation $(H_{\mathrm{eff}} + \mu_1)v_0 = 0$ has the nonvanishing solutions \eqref{wkb:leading-trace}. Periodicity is equivalent to
\begin{equation}\label{wkb:main-compatibility}
\mu_1 = \frac{2\pi m}{\ell} - \frac{1}{\ell}\int_0^\ell\mathcal{V}(y)\,\dd{d}y,\qquad m\in\Z.
\end{equation}
Fix one branch, for instance $m = 0$. Since $\mathcal{V}$ and $\mu_1$ are real, $\rho_0\vert v_0\vert^2$ is constant. At every subsequent order, division by $v_0$ reduces the trace equation to $i\,\partial_y$. The next energy coefficient cancels its mean, and the zero-mean periodic primitive determines the trace. The recursions \eqref{wkb:recursion} give the induction.

The construction uses the whole circle. Cutting off a nonzero leading trace in $y$ produces only a polynomially small error, since $\im\Phi = 0$ on $Z_0$. For example, a fixed nonconstant cutoff $\chi_{\rm tan}(y)$ applied to $v_h = e^{iy/h - x^2/(2h)}$ in the flat model gives the commutator $[P_h,\chi_{\rm tan}]v_h = -2ihx^2\chi_{\rm tan}'v_h - h^2x^2\chi_{\rm tan}''v_h$, whose norm divided by $\norm{v_h}$ has order $h^2$. We therefore solve the periodic transport equations before inserting a cutoff away from $Z_0$.

\smallskip
\noindent
\emph{Step 3. Optimal truncation.} With the phase, branch and traces fixed above, Proposition~\ref{app:wkb-hierarchy} gives the factorial estimates \eqref{wkb:theorem-factorial} on one fixed complex tube, with real energy coefficients. For a fixed normal cutoff $\chi$, equal to one near zero and supported in that tube, set
$$
b_\chi = \inf_{\supp\chi'\times\T_\ell}\im\Phi>0,\qquad0<\beta<\min(\Lambda^{-1},b_\chi),
$$
where $\Lambda$ is the constant in the factorial bounds. Proposition~\ref{app:wkb-realization} gives a truncation $N(h) = \lfloor\delta/h\rfloor$, with $0<\delta<\Lambda^{-1}$, such that $a^{(N(h))} = a_0 + \mathrm{O}(h)$ uniformly and the residual before multiplication by $\chi$ is exponentially small. Set $w_h = \chi e^{i\Phi/h}a^{(N(h))}$ and $u_h = w_h/\norm{w_h}$. The commutator with $\chi$ is supported where $\im\Phi \geq b_\chi$. The Gaussian profile \eqref{wkb:positive-phase} and the nonzero leading trace give $\norm{w_h}\asymp h^{1/4}$, so normalization costs a factor of order $h^{-1/4}$. Hence
$E(h) = h + \mathrm{O}(h^2)\in\R$ and $\norm{(P_h - E(h))u_h} \leq C_\beta e^{-\beta/h}$.

\smallskip
\noindent
\emph{Step 4. Concentration.} By \eqref{wkb:positive-phase}, the change $x = \sqrt h\,t$ and dominated convergence give, for $f\in C(M)$,
$$
h^{-1/2}\int_M f\vert w_h\vert^2\,\dd{d}\nu \rightarrow \sqrt{\pi}\int_0^\ell f(0,y)\rho_0(y)\vert a_0(0,y)\vert^2\,\dd{d}y.
$$
The limit for $f = 1$ is positive. The same computation with $x^2$ gives $\norm{xu_h}_{L^2(U,\nu)} = \mathrm{O}(\sqrt h)$. Since $\rho_0\vert a_0(0,\cdot)\vert^2$ is constant, the normalized mass satisfies $\int_M f\vert u_h\vert^2\,\dd{d}\nu\longrightarrow\frac{1}{\ell}\int_0^\ell f(0,y)\,\dd{d}y$.
\end{proof}

\paragraph{Spectral consequences.}
\begin{cor}\label{cor:spectral-sequence}
Under the hypotheses of Theorem~\ref{thm:periodic-quasimodes}, there is a sequence of eigenvalues of $\triangle$ such that
$$
\lambda_n = \frac{2\pi n}{\ell} + \mu_1 + \frac{\mu_2\ell}{2\pi n} + \mathrm{O}(n^{-2}),\qquad n\to+\infty.
$$
More precisely, for $h_n = \ell/(2\pi n)$, these eigenvalues can be chosen with $\vert\lambda_n - E(h_n)/h_n^2\vert \leq Ch_n^{-2}e^{-\beta/h_n}$.
\end{cor}
\begin{proof}
The structure is bracket-generating on the compact surface, so $\triangle$ has compact resolvent. The spectral theorem applied to the normalized quasimodes bounds $\operatorname{dist}(E(h)/h^2,\sigma(\triangle))$ by $Ch^{-2}e^{-\beta/h}$. The factorial bounds give the stated finite-order expansion of $E(h)/h^2$.
\end{proof}

For the coupled example of Subsection~\ref{sec:surface-examples}, with $\varphi(x,y) = \varepsilon x\sin y$ and $d\nu = dx\,dy$, formula \eqref{wkb:effective-operator} gives $\mathcal{V} = 5\varepsilon^2\sin^2y/4$. Corollary~\ref{cor:spectral-sequence} therefore yields $\lambda_n = n-5\varepsilon^2/8+\mathrm{O}(n^{-1})$.


\begin{remark}[Stationary quasimodes and observation]
The stationary quasimodes of Theorem~\ref{thm:periodic-quasimodes} give a direct lower bound when a tube around the whole singular circle is unobserved. Choose their cutoff $\chi$ with support disjoint from $\omega$, and let $b_\chi$ and $\Lambda$ be as in Proposition~\ref{app:wkb-realization}. Then $\Tobs(\omega) \geq \min(\Lambda^{-1},b_\chi)$. Indeed, for every $0<\beta<\min(\Lambda^{-1},b_\chi)$, Duhamel's formula bounds $\Vert e^{-t\triangle}u_h\Vert_{L^2(\omega,\nu)}$ by $C_Th^{-1}e^{-\beta/h}$ for $0 \leq t \leq T$, whereas the terminal norm is $e^{-T/h+\mathrm{O}(1)}$ if $T<\beta$. Observability therefore fails for such $T$. These states have limiting mass along the whole circle; the sharp lower bound in Theorem~\ref{thm:exact-time} uses the angular localization developed in Subsection~\ref{sec:sharp-lower}.
\end{remark}

\paragraph{Proof of Corollary~\ref{cor:schrodinger-obstruction}.}\label{sec:schrodinger-observation}

\begin{proof}
Only a finite WKB expansion is needed. Fix the phase, branch, amplitudes and cutoff in Appendix~\ref{app:wkb}. With $\alpha = 2\pi/\ell$ and $h_n = (\alpha n)^{-1}$, the normalized states $u_n = u_{h_n}^{(1)}$ of \eqref{wkb:quasimode} satisfy $\norm{(\triangle-\Lambda_n)u_n} = \mathrm{O}(n^{-1})$, where $\Lambda_n = \alpha n+\mu_1$, by \eqref{wkb:exact-residual}. Gaussian rescaling gives $\vert u_n\vert^2\nu\rightharpoonup\dd{d}y/\ell$ on $Z$ and, for each fixed $m\in\Z$,
\begin{equation}\label{schrodinger:coherence}
\norm{u_{n+m}-e^{i\alpha my}u_n}_{L^2(M,\nu)}\longrightarrow0.
\end{equation}
Indeed, the leading amplitudes agree, the normalizing constants have ratio tending to one, and $e^{i\alpha m(\Phi-y)}\to1$ on the concentrating mass.

For $y_0\in\T_\ell$ and $J\in\N$, define
$$
F_J(s)=\frac{1}{\sqrt{2J+1}}\sum_{\vert m\vert\leq J}e^{i\alpha ms},\qquad
w_{n,J}=\frac{1}{\sqrt{2J+1}}\sum_{\vert m\vert\leq J}e^{-i\alpha my_0}u_{n+m},\qquad n>J.
$$
For fixed $J$ and $T$, Duhamel's formula and \eqref{schrodinger:coherence} imply
$$
\sup_{\vert t\vert\leq T}\norm{e^{-it\triangle}w_{n,J}-e^{-it\Lambda_n}u_nF_J(y-y_0-t)}_{L^2(M,\nu)}\longrightarrow0.
$$
The mass limit gives $\norm{w_{n,J}}\to1$. Since $\vert F_J\vert^2\dd{d}y/\ell$ is a Fej\'er approximate identity, a diagonal over a countable dense subset of $C(M)$ and bounded time intervals gives normalized states $\psi_J = w_{n(J),J}/\norm{w_{n(J),J}}$ such that
\begin{equation}\label{schrodinger:packet}
\sup_{\vert t\vert\leq T}\Big\vert\int_M f\vert e^{-it\triangle}\psi_J\vert^2\,\dd{d}\nu-f(0,y_0+t)\Big\vert\longrightarrow 0  \qquad\forall f\in C(M).
\end{equation}
For $T<L$, place the trajectory for $0\leq t\leq T$ inside $I$ and choose $f\in C(M;[0,1])$ equal to zero there and one on $\overline\omega$. Then \eqref{schrodinger:packet} contradicts observability. For $T=L$, make the observed integral at most $\delta$ on $[0,L-\delta]$; unitarity bounds the remaining integral by $\delta$. If $Z\cap\overline\omega = \varnothing$, apply the same cutoff argument to $u_n$ and Duhamel's formula for any fixed time.
\end{proof}

\begin{remark}\label{rem:schrodinger-microlocalization}
At fixed $J$, the Gaussian profile of $b_{n,J}(t) = u_nF_J(y-y_0-t)$ gives $h_nD_xb_{n,J}\to0$ and $(h_nD_y-1)b_{n,J}\to0$ in $L^2$, uniformly on bounded time intervals, where $D_x = -i\partial_x$. Including these limits in the diagonal and using the $L^2$ approximation above identifies the unique semiclassical measure of $e^{-it\triangle}\psi_J$ at scale $h_{n(J)}$ as $\mu_t = \delta_{(0,y_0+t;0,1)}$ (see \cite[proof of Proposition~3.1]{AS23} for the definition). Conjugate initial states give $\delta_{(0,y_0-t;0,-1)}$.

The spacing $\Lambda_{n+1}-\Lambda_n = \alpha$ accounts for the unit speed: $L$ measures a travel time in the normalized coordinate, rather than a sub-Riemannian arc length. Although $H_p = 0$ on $\Sigma = \{x = \xi = 0,\ \eta\ne0\}$, physical time $t$ corresponds to semiclassical time $t/h$. This provides the concentration on the characteristic cone sought in \cite[Section~3.5]{Let21} for the present analytic class.
\end{remark}

\appendix

\section*{Appendix}
\addcontentsline{toc}{section}{Appendix}
\renewcommand{\thesubsection}{\Alph{subsection}}
\counterwithin{thm}{subsection}
\renewcommand{\theHsubsection}{appendix.\Alph{subsection}}
\renewcommand{\theHthm}{\theHsubsection.\arabic{thm}}

Appendices~\ref{current-appendix}--\ref{rough-appendix} supply the estimates used for the upper bound in Subsection~\ref{sec:sharp-upper}. Appendix~\ref{current-appendix} constructs the adapted transverse subspace and derives the equation and decay estimate for its complement. Appendix~\ref{graph-appendix} justifies the operator domains and the spatial derivatives required in these arguments. Appendix~\ref{rough-appendix} proves the initial weighted observation estimate. Appendix~\ref{app:factorial} supplies the factorial bounds used for the lower bound in Subsection~\ref{sec:sharp-lower} and the stationary quasimodes of Section~\ref{wkb:section}.

\subsection{Adapted transverse subspace and complementary decay}
\label{current-appendix}

In the flat whole-line model, the first $J+1$ transverse oscillator modes at angular frequency $n>0$ have decay rates $n,3n,\ldots,(2J+1)n$, while their orthogonal complement decays at rate at least $(2J+3)n$. Angular dependence destroys this invariant decomposition: the retained modes force the complement. Proposition~\ref{adapted-subspace} constructs a time-dependent subspace close to their span, whose failure of invariance is exponentially small in the angular frequency. The complement retains the leading dissipation rate $(2J+3)n$, up to an arbitrarily small relative loss and lower-order terms, and the forcing from the retained component is exponentially small. The analytic domains and defect exponents are fixed before $J$, allowing the complementary waiting time in Subsection~\ref{sec:sharp-upper} to tend to zero.

The construction has three steps. Appendix~\ref{app:compression} writes the holomorphic normal form \eqref{normal-form} as a Taylor--Fourier matrix and truncates normal degrees at a level proportional to the angular frequency. Appendix~\ref{app:evolving-graph} constructs a time-dependent graph over degrees $0,\ldots,J$, cancelling their forcing into higher degrees in this truncated system. Appendix~\ref{app:real-evaluation} evaluates the graph on the real cylinder and proves Proposition~\ref{adapted-subspace}. Appendix~\ref{app:complement} then establishes the projection, coercivity and evolution estimates for the complementary component.

We use the cylinder, operator and coefficient space of Subsection~\ref{sec:sharp-upper}. On $\mathcal{F}_J$, let $Nf = (nf_{\sigma,j,n})$ and $D(N) = \{f \,\mid\, \sum_{j,n}n^2\big(\vert f_{+,j,n}\vert ^2 + \vert f_{-,j,n}\vert ^2\big)<+\infty\}$. A synthesis block maps $\mathbb{C}^{J + 1}$ at frequency $\sigma n$ to output frequency $m$. Its Sobolev norm includes the factor $e^{im\theta}$; its angular decay is measured in $m - \sigma n$.

We first record the angular matrix estimates used throughout the appendix.

\begin{lem}\label{app:block-calculus}
Let $T = (T_{mn})$ be an operator-valued matrix between Hilbert direct sums indexed by subsets of $(2\pi/\ell)\Z$. Extend absent blocks by zero, and, for $s \geq 0$, set
$$
q_s(k) = \frac{e^{-s\vert k\vert }}{(1 + \vert k\vert )^2},\qquad \Vert T\Vert_{q_s} = \sup_{m,n}\frac{\Vert T_{mn}\Vert }{q_s(m - n)}.
$$
The operator norm is bounded by the geometric mean of the maximal row and column sums of $\Vert T_{mn}\Vert $. In particular, finite $\Vert T\Vert_{q_s}$ implies boundedness on these direct sums. Products satisfy $\Vert TU\Vert_{q_s} \leq C_\ell\Vert T\Vert_{q_s}\Vert U\Vert_{q_s}$, and adjoints have the same norm. Let $N$ act by $\vert n\vert $ on each frequency space. Then
$\Vert e^{\eta N}Te^{-\eta N}\Vert_{q_{s - \vert \eta\vert }} \leq \Vert T\Vert_{q_s}$ for $\vert \eta\vert <s$.
For $0<s'<s$, polynomial factors in $m - n$ are absorbed by the smaller exponent. With $\langle N\rangle = (1 + N^2)^{1/2}$, one also has
$\Vert \langle N\rangle^kT\langle N\rangle^{-k}\Vert_{q_{s'}} \leq C_{s - s',k,\ell}\Vert T\Vert_{q_s}$ for every $k\in\R$.
On one such direct sum, a uniformly bounded, uniformly coercive positive block-diagonal operator remains invertible after a sufficiently small perturbation in $\Vert \cdot\Vert_{q_s}$, with inverse bounded in this norm.
\end{lem}
\begin{proof}
The Schur test uses $\sum_k q_s(k)<+\infty$. Splitting the convolution according to $\vert j\vert \geq \vert k\vert /2$ or $\vert k - j\vert \geq \vert k\vert /2$ gives $q_s*q_s \leq C_\ell q_s$, uniformly for $s \geq 0$, hence the product bound. The conjugation estimate follows from $\big\vert \vert m\vert - \vert n\vert \big\vert \leq \vert m - n\vert $. Moreover, $\langle m\rangle/\langle n\rangle \leq C\langle m - n\rangle$, so a strict decrease of $s$ absorbs polynomial factors. The product bound then justifies the Neumann series for the inverse. The conjugations are first defined on finite sequences and extend by these estimates.
\end{proof}

\begin{prop}\label{adapted-subspace}
There are fixed smaller complex domains and constants $b_*,\gamma_*,s_*>0$ such that, for every $J \geq 0$ and sufficiently large $N_0$ depending on $J$, the columns \eqref{upper-columns} define $K_J\in C^1([0,b_*/(2J + 3)];\mathcal{L}(\mathcal{F}_J,\mathscr{H}))$. Their corrections satisfy $\Vert b_{\sigma,j,n}\Vert_{C^k} \leq C_{J,k}/n$ and $\Vert \partial_tb_{\sigma,j,n}\Vert_{C^k} \leq C_{J,k}$ on these domains. Moreover, $\Vert K_Jf\Vert \asymp_J\Vert f\Vert $ and $K_J:D(N)\to D(\triangle_\Phi)$ is bounded. There is a defect $E_J:\mathcal{F}_J\to\mathscr{H}$ such that, if ${P^\perp}$ is the orthogonal projection onto $(\operatorname{Ran}K_J)^\perp$,
\begin{equation}\label{current-intertwining}
{P^\perp}(K_J'+\triangle_\Phi K_J - E_J) = 0\quad\hbox{on }D(N),\qquad \Vert (E_J)_{m,\sigma n}\Vert_{\mathcal{L}(\mathbb{C}^{J + 1},H^k(\mathcal{C}))} \leq C_{J,k}e^{-\gamma_* n}q_{s_*}(m - \sigma n),
\end{equation}
All bounds are uniform in time. The domains and $b_*,\gamma_*,s_*$ are independent of $J$. Other constants may depend on $J,k,N_0$.
\end{prop}

The proof occupies Appendices~\ref{app:compression}--\ref{app:real-evaluation}: we restrict the Taylor-Fourier matrix, construct its evolving graph, then evaluate it on the real cylinder and estimate the discarded coefficients. Appendix~\ref{app:complement} derives the complementary evolution and decay estimates. The operator-domain bounds used below are proved in Appendix~\ref{app:dirichlet}.

\subsubsection{The compressed coefficient matrix}\label{app:compression}

In the holomorphic normal form, one has
$L = (1 + 2w\partial_w)D_z + \mathscr{A}$ with $\mathscr{A} = -d_2\partial_w^2 + d_1\partial_w + d_0$,
the Euler term is diagonal on $w^je^{inz}$, with entry $(2j + 1)n$. The second normal derivative grows like $j^2$. For degrees above $J$, the diagonal gap supplies a factor $1/((j - J)n)$ at matching angular frequencies. Restricting to $j \leq \kappa n$ makes their ratio small; on a smaller normal disc the discarded degrees are exponentially small in $n$.

Fix a coefficient disc of radius $R_c$ and an angular strip. Choose $0<\rho<R_c$, set $\vartheta_0 = \rho/R_c$, and use the basis $\rho^{-j}w^je^{inz}$. For $\kappa>0$, set
$\mathcal{I}_\kappa = \{(j,n)\,\mid\, n\in\alpha\N, n \geq N_0, \ 0 \leq j \leq \lfloor\kappa n\rfloor\}$, $\mathcal{E} = \ell^2(\mathcal{I}_\kappa)$, $\alpha = 2\pi/\ell$.
The compression $\Pi_\kappa$ retains these coefficients. Its matrix is $L_\kappa = \Lambda_{\mathrm{E}} + \mathscr{A}_\kappa$, where $(\Lambda_{\mathrm{E}})_{j,n} = (2j + 1)n$. Matrix identities below are first read on finite coefficient families.

\begin{lem}\label{current-compression}
There are constants $C,s_c>0$, depending only on the chosen analytic domains, such that
\begin{equation}\label{graph-entry-bound}
\vert (\mathscr{A}_\kappa)_{(l,m),(j,n)}\vert \leq C(1 + j)^2\vartheta_0^{(l - j)_+}e^{-s_c\vert m - n\vert }, \qquad l \geq j - 2.
\end{equation}
All other entries vanish. If $\kappa N_0 \geq J + 2$, decompose
\begin{equation*}
\mathcal{E} = \mathcal{F}_J^+\oplus\mathcal{E}_{>J},\qquad
L_\kappa = \begin{pmatrix}
\Lambda_{\mathrm{E},r} + \mathscr{A}_{rr}&D_1\\
\mathsf{C}&\Lambda_{\mathrm{E},f} + \mathscr{A}_{ff}
\end{pmatrix}.
\end{equation*}
Here $\mathcal{F}_J^+$ retains degrees $0,\ldots,J$, and $\mathcal{E}_{>J}$ retains the higher degrees. The angular entries of $\mathscr{A}_{rr}$ and $D_1$ are bounded by $C_Je^{-s_c\vert m - n\vert }$; $D_1$ uses only input degrees $J + 1,J + 2$. For any fixed $\vartheta_0<\vartheta<1$, the entries of $\mathsf{C}$ satisfy
$\vert \mathsf{C}_{(l,m),(j,n)}\vert \leq C_J\vartheta^{l - J - 1}e^{-s_c\vert m - n\vert }$ for $l>J$ and $j \leq J$.
\end{lem}
\begin{proof}
Writing $(d_a)_{p,k}$ for a Taylor-Fourier coefficient, with negative Taylor indices set to zero, gives
$$
(\mathscr{A}_\kappa)_{(l,m),(j,n)} = \rho^{l - j}\left\{-j(j - 1)(d_2)_{l - j + 2,m - n} + j(d_1)_{l - j + 1,m - n} + (d_0)_{l - j,m - n}\right\}.
$$
Cauchy estimates give \eqref{graph-entry-bound}. The remaining claims follow from the finite retained degrees and the downward bandwidth two.
\end{proof}

\subsubsection{The evolving graph}\label{app:evolving-graph}

The graph $g = S(t)f$ must cancel the forcing $\mathsf{C}f$ into higher degrees. At equal angular frequencies the diagonal gap is positive. For unequal frequencies it may change sign; we compensate by decreasing the angular strip at speed $2J + 1$.

\begin{lem}\label{current-graph}
There are $\kappa,b_*>0$ and a fixed smaller complex product, independent of $J$, with the following property. For every $J$ and sufficiently large $N_0$ depending on $J$, there is a bounded graph $S(t):\mathcal{F}_J^+\to\mathcal{E}_{>J}$, with $S(0) = 0$, for $0 \leq (2J + 3)t \leq b_*$. On finite coefficient families it satisfies
\begin{equation}\label{current-riccati}
S'+(\Lambda_{\mathrm{E},f} + \mathscr{A}_{ff})S - S(\Lambda_{\mathrm{E},r} + \mathscr{A}_{rr}) = -\mathsf{C} + SD_1S.
\end{equation}
After multiplying column $(j,n)$ of $(I,S)$ by $\rho^j$, its holomorphic expression on the smaller product is
$e^{inz}\{w^j + w^{J + 1}b_{j,n}(t,w,z)\}$, with $\Vert b_{j,n}\Vert_{C^k} \leq C_{J,k}/n$ and $\Vert \partial_tb_{j,n}\Vert_{C^k} \leq C_{J,k}$.
For each $j,n$ and each fixed $k$, the amplitude $b_{j,n}$ belongs to $C^1([0,b_*/(2J + 3)];C^k)$. The displayed bounds are uniform in $t,n$.
\end{lem}
\begin{proof}
Choose $0<2b_*<s_0<s_c$, all before $J$, and set $\nu = 2J + 1$ and $s(t) = s_0 - \nu t$.
We use the convolution bound of Lemma~\ref{app:block-calculus}. For a matrix family $F$, define the diagonal integral
$$
(\mathcal{T}F)_{(l,m),(j,n)}(t) = \int_0^t e^{-(t - \tau)[(2l + 1)m - (2j + 1)n]} F_{(l,m),(j,n)}(\tau)\,d\tau.
$$
For $l>J$ and $j \leq J$, we use
\begin{align*}
(2l + 1)m - (2j + 1)n & \geq 2(l - J)m - \nu(n - m)_+, \\
e^{-(t - \tau)[(2l + 1)m - (2j + 1)n]}q_{s(\tau)}(m - n) & \leq e^{-2(t - \tau)(l - J)m}q_{s(t)}(m - n).
\end{align*}
When $m<n$, the decreasing angular strip cancels the possible growth from the negative part of the diagonal gap; the remaining positive gap yields the factor $[2(l-J)m]^{-1}$ after time integration.

Take the Banach space of blockwise continuous families with norm
$$
\Vert S\Vert_* = \sup_{t,l,m,j,n} \frac{m\,\vartheta^{-(l - J - 1)} \vert S_{(l,m),(j,n)}(t)\vert }{q_{s(t)}(m - n)}.
$$
Continuity here is imposed on each matrix entry. The norm makes these families complete, and the exponential integral defining $\mathcal{T}$ preserves entrywise continuity by dominated convergence on each finite time interval. The graph equation is the fixed-point equation
$S = \mathcal{T}\left(-\mathsf{C} - \mathscr{A}_{ff}S + S\mathscr{A}_{rr} + SD_1S\right)$.
We give the summation estimate for its only term with growing normal degrees. If $R = \Vert S\Vert_*$, the normal sum is bounded using
$\sum_{k = J + 1}^{l + 2}(k + 1)^2 \vartheta_0^{(l - k)_+}\vartheta^{k - J - 1} \leq C(l + 3)^2\vartheta^{l - J - 1}$.
The strict angular reserve $s_c - s_0$ gives, for $s \leq s_0$,
$$
\sum_{p \geq N_0}e^{-s_c\vert m - p\vert }\frac{q_s(p - n)}{p} \leq \frac{C}{m} q_s(m - n).
$$
Indeed, $m/p \leq 1 + \vert m - p\vert /N_0$ and $(1 + \vert m - n\vert ) \leq (1 + \vert m - p\vert )(1 + \vert p - n\vert )$; the remaining polynomial factors are summable against $e^{-(s_c - s_0)\vert m - p\vert }$. Consequently,
$$
\vert (\mathscr{A}_{ff}S)_{(l,m),(j,n)}\vert \leq CR\frac{(l + 3)^2}{m} \vartheta^{l - J - 1}q_{s(t)}(m - n).
$$
Since $l \leq \kappa m$, the diagonal integral then yields
$\Vert \mathcal{T}(\mathscr{A}_{ff}S)\Vert_* \leq C\left(\kappa + \frac{(J + 3)^2}{N_0}\right)R$.
The finite normal block bounds from Lemma~\ref{current-compression}, together with the same angular convolution, give
$$
\Vert \mathcal{T}\mathsf{C}\Vert_* \leq C_J,\qquad \Vert \mathcal{T}(S\mathscr{A}_{rr})\Vert_* \leq C_JN_0^{-1}R,\qquad \Vert \mathcal{T}(SD_1S)\Vert_* \leq C_JN_0^{-2}R^2.
$$
The corresponding difference estimates have the same constants, with $R^2$ replaced by $2R\Vert S_1 - S_2\Vert_*$. Choose $\kappa$ so that its linear contribution is less than $1/8$, and then choose $N_0$ depending on $J$. The map is a contraction on the ball $\Vert S\Vert_* \leq 2C_J$. The coefficient sums converge absolutely. Set $\underline s = s_0-b_*>0$. Summing the geometric normal factor in the definition of $\Vert S\Vert_*$ gives $\Vert S_{m,n}(t)\Vert \leq C_Jm^{-1}q_{\underline s}(m-n)$, where the block maps the retained normal coefficients at frequency $n$ to the higher normal coefficients at frequency $m$. Let $P_L$ and $Q_L$ restrict the input and output angular frequencies to at most $L$. Schur summation gives $\sup_t\Vert S(t)-Q_LS(t)P_{2L}\Vert \leq C_J(L^{-1}+e^{-\underline sL/2})$: output frequencies above $L$ contribute the factor $L^{-1}$, whereas $m\leq L$ and $n>2L$ imply $\vert m-n\vert>L$. The compression restricts each output degree to $l\leq\kappa m$, so $Q_LS(t)P_{2L}$ is a finite matrix with continuous entries. The uniform tail bound therefore proves operator-norm continuity of $S$.

Finally, fix $\rho_1<\rho$ and $s_1<s_0 - b_*$, before $J$. Evaluation on $\vert w\vert <\rho_1$, $\vert \im z\vert <s_1$ sums the normal and angular series absolutely. Every output degree is at least $J + 1$, so division by $w^{J + 1}$ is allowed. The inequality $1/m \leq (1 + \vert m - n\vert /N_0)/n$ gives the $C^k/n$ bound after factoring $e^{inz}$; strict domain margins absorb every fixed derivative. Differentiating each mild entry gives \eqref{current-riccati}. Its diagonal term multiplies $S$ by $(2l + 1)m - (2j + 1)n$, and $n/m \leq 1 + \vert m - n\vert /N_0$. The estimates of Lemma~\ref{current-compression} and the preceding convolution bounds control each differentiated entry by $C_J(1+l)^2(1+\vert m-n\vert)^2\vartheta^{l-J-1}q_{\underline s}(m-n)$. After factoring $e^{inz}w^{J+1}$ and evaluating on the smaller product, the normal and angular margins absorb these polynomial factors and every fixed spatial derivative. The resulting summable majorant is uniform in time and in the input frequency, so it proves the stated $C_{J,k}$ bound and, by dominated convergence for each column, continuity of $\partial_tb_{j,n}$ up to $t=0$.
\end{proof}

With $B = \Lambda_{\mathrm{E},r} + \mathscr{A}_{rr}$, the graph identity is
$$
\partial_t\binom{I}{S} + L_\kappa\binom{I}{S} = \binom{I}{S}(B + D_1S).
$$
Thus the retained generator is $B + D_1S$. Before the normal cutoff, the only error in the evaluated identity comes from discarded coefficients.
\subsubsection{Evaluation and proof of Proposition~\ref{adapted-subspace}}\label{app:real-evaluation}

Evaluate the graph columns at $z_+ = \theta + ir^2/2$ and $w_+ = \psi(r,z_+)$, restore $g_+$, and multiply column $(j,n)$ by $\chi n^{j/2 + 1/4}$. Complex conjugation gives the negative sign. These are the columns \eqref{upper-columns}. The cylinder and cutoff are fixed before $J$, strictly inside the graph domains. Let $K_J^0$ denote synthesis by the flat columns $\chi n^{j/2 + 1/4}r^je^{-nr^2/2}e^{i\sigma n\theta}$.

\begin{lem}\label{current-evaluation}
The angular blocks of the evaluated graph satisfy
$$
\Vert (K_J - K_J^0)_{m,\sigma n}\Vert + \Vert (K'_J)_{m,\sigma n}\Vert \leq C_Jn^{-1/2}q_{s_e}(m - \sigma n)
$$
for a fixed $s_e>0$, with norms from $\mathbb{C}^{J + 1}$ to the output $L^2$ Fourier mode. If $\Gamma_J^0$ is the positive Gram matrix of the normalized Gaussian monomials, repeated at every frequency and sign, then $K_J^*K_J = \Gamma_J^0 + \mathrm{O}_J(N_0^{-1/2})$.
\end{lem}
\begin{proof}
Since $w_+ = r+\mathrm{O}(r^2)$ and $g_+ = 1+\mathrm{O}(r)$, the difference from the flat column gains one factor of $r$. Rescaling $r = n^{-1/2}x$ gives the $n^{-1/2}$ bound. The correction has normal degree at least $J + 1$, so the bounds for $b_{j,n}$ and $\partial_tb_{j,n}$ give, respectively, smaller errors and $\Vert k'_{\sigma,j,n}\Vert \leq C_Jn^{(j - J - 1)/2} \leq C_Jn^{-1/2}$. The reserved angular strip gives the displayed off-diagonal bounds. Each application of $\triangle_\Phi$ costs at most $\langle n\rangle$ by the same Gaussian scaling. Angular Schur summation thus proves $\Vert K_Jf\Vert + \Vert \triangle_\Phi K_Jf\Vert \leq C_J\Vert \langle N\rangle f\Vert $. The compact normal cutoff supplies the Dirichlet boundary condition, and Lemma~\ref{operator-graph} identifies this domain with $D(\triangle_\Phi)$.

Let $B_J(t)$ denote synthesis by the differentiated columns and let $P_L$ retain input frequencies $n\leq L$ for both signs. The displayed block bound and Schur summation imply $\sup_t\Vert T(t)(I-P_L)\Vert \leq C_JL^{-1/2}$ for $T = K_J-K_J^0$ and $T = B_J$. For fixed $L$, these maps have finitely many input columns, each continuous in $L^2(\mathcal{C})$; their infinitely many output coefficients have already been summed on the smaller complex product. Hence $T(t)P_L$ is continuous in operator norm, and the uniform tail estimate proves the same for $T$. The columnwise fundamental theorem of calculus extends by density to $K_J(t)-K_J(s)=\int_s^tB_J(\tau)\,d\tau$ in operator norm. Therefore $K_J$ is $C^1$, with derivative $B_J$, as asserted in Proposition~\ref{adapted-subspace}. The flat synthesis has Gram matrix $\Gamma_J^0$, up to an exponentially small cutoff error. The bound for $K_J - K_J^0$ gives the stated Gram estimate and, for large $N_0$, the norm equivalence.
\end{proof}

We now estimate the discarded part, keeping its exponential constants independent of $J$.

\begin{lem}\label{current-defect}
The evaluated graph satisfies \eqref{current-intertwining}, with $\gamma_*,s_*>0$ chosen independently of $J$. The defect has compact normal support and satisfies its block estimate for every fixed Sobolev order $k$.
\end{lem}
\begin{proof}
We treat the positive angular sign; complex conjugation gives the other sign. Set $a_{j,n} = e^{inz}(w^j+w^{J+1}b_{j,n})$. Since $a_{j,n}$ and its time derivative belong to the compressed coefficient space, the graph identity gives the defect column
$$
E_{+,j,n} = n^{j/2+1/4}\left\{\chi g_+[(I-\Pi_\kappa)La_{j,n}](w_+,z_+)+[\triangle_\Phi,\chi]\big(g_+a_{j,n}(w_+,z_+)\big)\right\}.
$$
All estimates below are uniform over the construction time interval. Choose fixed radii $R_1<R_2$ strictly inside the graph disc, with $\vert w_+\vert<R_1$ on the real cylinder and on a slightly larger complex angular strip. Leave a further normal margin beyond $R_2$. Cauchy estimates on these fixed domains give
$$
\sum_{l \geq 0}R_2^l\vert(La_{j,n})_{l,q}\vert \leq C_J(1+n+\vert q\vert)^{a_J}e^{-s\vert q-n\vert},\qquad q\in\alpha\Z,
$$
for a fixed $s>0$. The same estimate holds for $a_{j,n}$ and every fixed derivative, with a possibly larger polynomial power. The Fourier coefficients of the evaluated factor $g_+w_+^l$, and their fixed normal derivatives, satisfy
$\Vert\partial_r^k(g_+w_+^l)_{m-q}\Vert_{L^\infty_r} \leq C_kR_1^le^{-s_e\vert m-q\vert}$
for a fixed $s_e>0$. The strict margin in $\vert w_+\vert<R_1$ absorbs the powers of $l$ produced by differentiation. The remaining factor $e^{iqz_+}$ contributes $e^{iq\theta-qr^2/2}$; its derivatives contribute only polynomial factors in $q$. Sobolev norms of the output block also introduce polynomial factors in $m$.

Write $D_q = \vert q-n\vert$ and $E_q = \vert m-q\vert$ for the angular frequency changes before and after evaluation, and set $\delta_0 = \kappa\log(R_2/R_1)>0$. Decrease the fixed cylinder so that $A_{\max} = r_0^2/2<s/4$. There are three ranges of discarded coefficients.

If $q \geq \max(N_0,n/2)$, every discarded degree satisfies $l>\kappa q$. Normal evaluation therefore supplies
$(R_1/R_2)^l \leq e^{-\delta_0q} \leq e^{-\delta_0n/2}$.
Here $q>0$, so $e^{-qr^2/2} \leq 1$. After summing the normal degrees, this part is bounded by a polynomial in $n,\vert q\vert,E_q$ times $e^{-\delta_0n/2-sD_q-s_eE_q}$.

If $q<n/2$, then $D_q = n-q \geq n/2$, and angular decay supplies the required factor in the input frequency. Negative $q$ are included: since $\vert q\vert \leq D_q$ when $q<0$, their apparent Gaussian growth satisfies
$e^{-sD_q}e^{A_{\max}(-q)_+} \leq e^{-3sD_q/4} \leq e^{-sn/8}e^{-sD_q/2}$.
No normal truncation gain is needed in this range.

The remaining range is $n/2 \leq q<N_0$. It consists of nonnegative discarded frequencies and implies $N_0 \leq n<2N_0$. There are finitely many such input and intermediate frequencies. The same fixed angular evaluation strip gives their required off-diagonal decay. Multiplying the constant by $e^{2\gamma_*N_0}$ gives the input factor $e^{-\gamma_*n}$ without changing either exponent. Dependence of this constant on $N_0$ is allowed.

For the cutoff term, set $A_\chi = \inf_{\operatorname{supp}\chi'}r^2/2>0$. The coefficients of $a_{j,n}$ have only positive frequencies $q \geq N_0$. If $q \geq n/2$, the cutoff support supplies $e^{-A_\chi q} \leq e^{-A_\chi n/2}$; if $q<n/2$, use the preceding angular estimate. The identity $[\triangle_\Phi,\chi] = -2c\chi'\partial_r-(c\chi')'$ shows that this term has the same bounds, up to polynomial factors and a fixed reduction of the evaluation strip to include $c$.

For completeness, these estimates yield both required exponentials simultaneously. Set
$\gamma_0 = \min\big(\frac{\delta_0}{2},\frac{s}{8},\frac{A_\chi}{2}\big)$, $\gamma_* = \frac{\gamma_0}{2}$, $0<s_*<\min\left(\frac{s}{16},\frac{s_e}{8}\right)$.
Outside the finite range, every contribution is bounded by a polynomial in $n,D_q,E_q$ times $e^{-\gamma_0n-sD_q/2-s_eE_q}$. This includes the column normalization, since $j \leq J$, and every fixed Sobolev derivative. Fixed halves of these three exponential reserves absorb all such polynomial factors. Summing over $q$ then uses
$\sum_q e^{-sD_q/4-s_eE_q/2} \leq C e^{-2s_*\vert m-n\vert} \leq C' q_{s_*}(m-n)$,
where $D_q+E_q \geq \vert m-n\vert$ and the unused exponential factors are summable on $\alpha\Z$. Thus
$$
\Vert(E_J)_{m,+n}\Vert_{\mathcal{L}(\mathbb{C}^{J+1},H^k(\mathcal{C}))} \leq C_{J,k}e^{-\gamma_*n}q_{s_*}(m-n).
$$
The radii, strips, $\kappa$ and cutoff were chosen before $J$. Hence $\gamma_*,s_*$ are independent of $J$ and of the fixed derivative order $k$; only the constants and $N_0$ change. Both defect terms have compact normal support. For finite input sequences, the compressed graph identity places $K_J'+\triangle_\Phi K_J-E_J$ in $\operatorname{Ran}K_J$; the normalized retained coefficients are square summable by angular exponential decay. Boundedness of $K_J:D(N)\to D(\triangle_\Phi)$, $K_J'$ and $E_J$ extends the projected identity to $D(N)$.
\end{proof}

\begin{proof}[Proof of Proposition~\ref{adapted-subspace}]
Fix the compression, complex domains and cylinder in the preceding lemmas. For each $J$, enlarge $N_0$ for the contraction and for the Gram error to be smaller than half the least eigenvalue of $\Gamma_J^0$. Lemmas~\ref{current-graph}-\ref{current-defect} give the asserted family, norm equivalence, domain bound and invariance error.
\end{proof}

\subsubsection{Projection and complementary estimates}\label{app:complement}

We pass from the constructed columns to the orthogonal projection onto their span. The projection bounds and the oscillator coercivity estimate then justify the complementary heat equation used in Subsection~\ref{sec:sharp-upper}.

Set $\Gamma_J = K_J^*K_J\in\mathcal{L}(\mathcal{F}_J)$, $K_J^\dagger = \Gamma_J^{-1}K_J^*$, $P = K_JK_J^\dagger$ and ${P^\perp} = I - P$. Let $\Pi_J$ be the flat Dirichlet oscillator projection onto levels $0,\ldots,J$ at frequencies $\vert n\vert \geq N_0$. For any of these maps, a subscript $\eta$ denotes angular conjugation by $e^{\eta N}$ on the indicated source and target spaces.

\begin{lem}\label{current-projections}
On a fixed exponent interval, independent of $J$, $P_\eta$ and $P'_\eta$ are bounded uniformly in time, and
$P_\eta z = 0 \Rightarrow \Vert N^{1/2}\Pi_Jz\Vert \leq C_J\Vert z\Vert$.
\end{lem}
The synthesis, Gram inverse and pseudoinverse have bounded conjugates on the same interval. Moreover, $K_J^\dagger:D(\triangle_\Phi)\to D(N)$ and $P D(\triangle_\Phi)\subset D(\triangle_\Phi)$.
\begin{proof}
The block bounds of Lemma~\ref{current-evaluation} give all conjugated synthesis and Gram bounds. Increase $N_0$ so that the Gram perturbation is small also on the fixed exponent interval. Lemma~\ref{app:block-calculus} gives the inverse and its polynomially weighted bounds. Since $K_J^*$ maps $D(N)\subset\mathscr{H}$ to $D(N)\subset\mathcal{F}_J$, Lemma~\ref{operator-graph} gives the pseudoinverse domain bound. Together with $K_J:D(N)\to D(\triangle_\Phi)$, this shows that $P$ preserves $D(\triangle_\Phi)$. Differentiating $P = K_J\Gamma_J^{-1}K_J^*$ gives the asserted bound for $P'$.

If $P_\eta z = 0$, then $(K_{J,-\eta})^*z = 0$. Consequently $(K_J^0)^*z = (K_J^0 - K_{J,-\eta})^*z$. The $n^{-1/2}$ block estimate controls this coefficient vector after multiplication by $N^{1/2}$. The flat columns span the first Hermite levels. Their cutoff errors are exponentially small, and the neighboring gaps are of order $n$, so their projection differs exponentially from $\Pi_J$. The positive flat Gram matrix now proves the last assertion.
\end{proof}

The following estimate gives rate one without excluded levels and rate $2J + 3$ on the complement.

\begin{lem}\label{complement-coercivity}
Set $\Pi_{-1} = 0$. If $J\in\{-1,0,1,\ldots\}$, $A \geq 0$ and $\Vert N^{1/2}\Pi_Jz\Vert \leq A\Vert z\Vert $, then, for every $\varepsilon>0$,
$$
\Re\mathfrak{h}_\eta(z,z) \geq (2J + 3 - \varepsilon)\Vert N^{1/2}z\Vert ^2 - C_{J,A,\varepsilon}\Vert z\Vert ^2, \qquad z\in\mathcal{V}.
$$
The exponent interval and cylinder are fixed independently of $J,A$ and $\varepsilon$.
\end{lem}
\begin{proof}
Write $\mathfrak{h}_{\mathrm{osc}}(z,z) = \Vert z_r\Vert ^2 + \Vert rD_\theta z\Vert ^2$. The fixed cylinder can be chosen so that $\Re\mathfrak{h}_\eta \geq (1 - \delta_0)\mathfrak{h}_{\mathrm{osc}} - C\Vert \cdot\Vert ^2$ with $\delta_0<1/10$ throughout the fixed exponent interval. Split at the first $L + 1$ oscillator levels, where $L = 3J + 6$. On these levels Lemma~\ref{fixed-level-perturbation} bounds the perturbation and its adjoint by $C_L(1 + N^{1/2})$. Young's inequality bounds the finite-block error and the two cross terms by $\varepsilon\Vert N^{1/2}\Pi_Lz\Vert ^2 + C_{J,\varepsilon}\Vert z\Vert ^2$. On the remaining levels, the relative form bound leaves coefficient $(1 - \delta_0)(2L + 3)>2J + 3$ in front of $N$. Levels $J + 1,\ldots,L$ already have rate at least $2J + 3$, and the assumed bound controls the lower levels. The finitely many angular frequencies excluded from $\Pi_J$ cost only an additive constant.
\end{proof}

We now apply these facts to the actual localized solution. On compact positive-time intervals, $U\in H^1(\mathscr{H})\cap L^2(D(\triangle_\Phi))$. The domain bounds above justify differentiating $v = {P^\perp}U$. With $R = {P^\perp}E_JK_J^\dagger$, the projected invariance identity gives $P'P = -{P^\perp}\triangle_\Phi P + R$ on $D(\triangle_\Phi)$. Its adjoint gives $P'{P^\perp} = -P\triangle_\Phi {P^\perp} + R^*$. Since $P' = P'P + P'{P^\perp}$, we obtain
\begin{equation}\label{tail-equation}
v_t + \triangle_\Phi v = -2P'v + (R^*-R)U + {P^\perp}F_{\rm obs}.
\end{equation}
Since $R=RP$ and $R^*=R^*{P^\perp}$, the retained component $PU$ enters this equation only through $-RPU$. If $E_J=0$, the equation depends only on $v$ and the observed source, even when the coefficients depend on the angular variable.

Decrease the fixed $s_*$ below the projection exponent interval. For $B_0<\min(\gamma_*,s_*)$, the defect and its adjoint satisfy
\begin{equation}\label{weighted-defect}
\Vert e^{\eta N}(R^*-R)U\Vert \leq C_J\Vert e^{-B_0N}U\Vert ,\qquad - B_0 \leq \eta\leq0.
\end{equation}
Indeed, input exponential decay makes $E_Je^{B_0N}$ bounded, and $e^{B_0\vert m\vert } \leq e^{B_0n}e^{B_0\vert m - \sigma n\vert }$ makes $E_J^*e^{B_0N}$ bounded. The conjugated projection and pseudoinverse bounds then give \eqref{weighted-defect}.

The weight-raising calculation in Subsection~\ref{sec:sharp-upper} is therefore justified for the actual complementary state: $P$ preserves $D(\triangle_\Phi)$, and negative angular weights preserve the common domain $D(\triangle_{\Phi,\eta}) = D(\triangle_{\mathrm{osc}})$. Lemma~\ref{tail-jets} supplies the local spatial derivatives after any further positive time margin.

\subsection{Operator domains and observed derivatives}\label{graph-appendix}

Appendix~\ref{app:dirichlet} identifies the common domain of the angularly conjugated operators and their comparison with the flat oscillator on finitely many levels. These facts justify the weighted energy estimates and projection identities of Appendix~\ref{current-appendix}. Appendix~\ref{app:regular-derivatives} then recovers the spatial derivatives of the complementary state needed for observation on level sets. We retain $\mathcal{C}$, $\mathscr{H}$ and $\mathcal{V}$ from Subsection~\ref{upper-functional-setting}, and write $D = D_\theta$, $N = \vert D\vert $.

\subsubsection{Dirichlet and conjugated operators}\label{app:dirichlet}

For a periodic analytic coefficient $a$, define its angular matrix by
$(a_\eta)_{mn}(r) = e^{\eta(\vert m\vert - \vert n\vert )}\widehat a_{m - n}(r)$ for $m,n\in(2\pi/\ell)\Z$.
On a fixed compact exponent interval inside the analytic strip, the Schur test bounds these matrices and every fixed coefficient derivative. In particular, $c(0,\theta) = 1$ gives $\Vert c_\eta(r) - I\Vert \leq C\vert r\vert $. Since $c$ is real, the Hermitian part has entries
$\big(\tfrac12(c_\eta+c_\eta^*)\big)_{mn} = \widehat c_{m-n}\cosh\big(\eta(\vert m\vert-\vert n\vert)\big)$.
The difference from $c$ is bounded in operator norm by $C\eta^2$: use $\cosh t-1 \leq t^2e^{\vert t\vert}/2$, $\big\vert\vert m\vert-\vert n\vert\big\vert \leq \vert m-n\vert$, and the reserved analytic strip. Thus a small fixed exponent interval preserves ellipticity of the angular coefficient. On the fixed small cylinder the forms
$$
\mathfrak{h}_\eta(v,w) = \langle c_\eta v_r,w_r\rangle + \langle c_\eta rDv,rDw\rangle + \langle W_\eta v,w\rangle, \qquad v,w\in\mathcal{V},
$$
are closed and sectorial after a constant shift. Let $\triangle_{\Phi,\eta}$ be their associated operators. At $\eta = 0$ this is the Dirichlet realization $\triangle_\Phi$, and $\triangle_{\Phi,\eta}^* = \triangle_{\Phi,-\eta}$. The coefficient-matrix definition avoids applying an unbounded positive exponential to arbitrary data.

Let $\triangle_{\mathrm{osc}} = -\partial_r^2 + r^2N^2$ be the flat Dirichlet operator. The graph norm controls the angular derivative as follows.

\begin{lem}\label{operator-graph}
Uniformly on the fixed exponent interval, $D(\triangle_{\Phi,\eta}) = D(\triangle_{\mathrm{osc}})$. For $f\in D(\triangle_\Phi)$,
\begin{equation}\label{actual-graph}
\Vert Nf\Vert + \Vert f_{rr}\Vert + \Vert r^2N^2f\Vert + \Vert rNf_r\Vert \leq C(\Vert \triangle_\Phi f\Vert + \Vert f\Vert ).
\end{equation}
All derivatives are distributional.
\end{lem}
\begin{proof}
At frequency $n$, integration by parts for a Dirichlet function gives
$$
\Vert f_{rr}\Vert ^2 + \Vert n^2r^2f\Vert ^2 + 2\Vert nr f_r\Vert ^2 = \Vert (-\partial_r^2 + n^2r^2)f\Vert ^2 + 2n^2\Vert f\Vert ^2.
$$
The boundary terms vanish. The factorization $\Vert f_r + \vert n\vert rf\Vert ^2 = \langle(-\partial_r^2 + n^2r^2)f,f\rangle - \vert n\vert \Vert f\Vert ^2$ also gives $\triangle_{\mathrm{osc},n} \geq \vert n\vert $. Summing over frequencies yields
\begin{equation}\label{first-graph}
\Vert Nf\Vert + \Vert f_{rr}\Vert + \Vert r^2N^2f\Vert + \Vert rNf_r\Vert \leq C\Vert \triangle_{\mathrm{osc}}f\Vert .
\end{equation}
Fourier truncation and one-dimensional Dirichlet approximation extend this estimate to $D(\triangle_{\mathrm{osc}})$.

On its smooth Dirichlet core, one has
$\triangle_{\Phi,\eta} = c_\eta \triangle_{\mathrm{osc}} - (c_r)_\eta\partial_r - ir^2(c_\theta)_\eta D + W_\eta$.
The principal perturbation has small relative bound on the fixed cylinder. The remaining terms have arbitrarily small relative bound, using $\Vert f_r\Vert ^2 + \Vert rNf\Vert ^2 = \langle \triangle_{\mathrm{osc}}f,f\rangle$ and $\Vert r^2Df\Vert \leq r_0\Vert rNf\Vert $. For large $\lambda$, the Neumann inverse of $I + (\triangle_{\Phi,\eta} - \triangle_{\mathrm{osc}})(\triangle_{\mathrm{osc}} + \lambda)^{-1}$ therefore gives a closed realization on $D(\triangle_{\mathrm{osc}})$. Its form is $\mathfrak{h}_\eta$, so it agrees with the associated sectorial operator. At $\eta = 0$, one has $D(\triangle_\Phi) = D(\triangle_{\mathrm{osc}})$ and $\Vert \triangle_{\mathrm{osc}}f\Vert \leq C(\Vert \triangle_\Phi f\Vert + \Vert f\Vert )$, which proves \eqref{actual-graph}.
\end{proof}

Finitely many oscillator levels concentrate on the scale $\vert n\vert ^{-1/2}$. On these levels the coefficient error $c_\eta - I = \mathrm{O}(r)$ therefore has angular order $1/2$, rather than order one.

\begin{lem}\label{fixed-level-perturbation}
Let $\Pi_L$ retain oscillator indices $0,\ldots,L$ for $\triangle_{\mathrm{osc},n}$ at every angular frequency. For every fixed $L$, uniformly on the fixed exponent interval,
$$
\Vert (\triangle_{\Phi,\eta} - \triangle_{\mathrm{osc}})\Pi_Lf\Vert + \Vert (\triangle_{\Phi,\eta}^*-\triangle_{\mathrm{osc}})\Pi_Lf\Vert \leq C_L(\Vert N^{1/2}f\Vert + \Vert f\Vert )\qquad\forall f\in D(N^{1/2}).
$$
The products are initially defined on finite Fourier sums and extended by this estimate.
\end{lem}

\begin{proof}
Rescaling $r = \vert n\vert ^{-1/2}s$ gives the Hermite estimates
$\Vert r^p\partial_r^q\Pi_{L,n}f\Vert \leq C_L\langle n\rangle^{(q - p)/2}\Vert f\Vert$, $p + q \leq 3$, $q \leq 2$.
They hold for the Dirichlet spectral subspaces as well: cutoff Hermite functions have exponentially small residuals, and the neighboring oscillator gaps are of order $\vert n\vert $. Min-max and
\eqref{first-graph} control the corresponding spectral subspaces
and their first two normal derivatives. Finitely many frequencies, including zero, are absorbed into $C_L$.

In the perturbation formula used above, write $c_\eta - I = r\widetilde c_\eta$. Each input Fourier block is then bounded by $C_L\langle n\rangle^{1/2}$. Exponential decay of the coefficient matrices in $\vert m - n\vert $ gives the asserted bound by the weighted Schur test. Replacing $\eta$ by $-\eta$ proves the adjoint estimate.
\end{proof}

\subsubsection{Derivatives on regular strips}\label{app:regular-derivatives}

Although the projections are nonlocal, their Gaussian columns make them smoothing on strips separated from $r = 0$.

\begin{lem}\label{tail-jets}
Let $v = {P^\perp}U$ satisfy \eqref{tail-equation}, and let $0<B_0<\min(\gamma_*,s_*)$. For compact time intervals $I\Subset I_0$ inside the construction interval and every $\omega'\Subset\mathcal{C}\cap\{\vert r\vert >0\}$,
\begin{equation}\label{tail-graphs}
\Vert v\Vert_{L^2(I;H^k(\omega'))} \leq C_{J,k,I,\omega'}\left( \Vert v\Vert_{L^2(I_0;\mathscr{H})} + \Vert e^{-B_0N}U\Vert_{L^2(I_0;\mathscr{H})} + \Vert F_{\rm obs}\Vert_{L^2(I_0;H^k(\mathcal{C}))}\right)
\end{equation}
for every integer $k \geq 0$. In particular, after the complementary energy estimate, the right-hand side is bounded by $C\obs^{1/2}$.
\end{lem}
\begin{proof}
Choose $\omega''$ with $\omega'\Subset\omega''\Subset\mathcal{C}\cap\{\vert r\vert >0\}$. The column formula \eqref{upper-columns} and its time derivative give $K,K':\mathcal{F}_J\rightarrow H^k(\omega'')$ boundedly for every fixed $k$.
Indeed, on $\vert r\vert \geq \delta>0$, each spatially differentiated column is bounded by $C_{J,k}n^{j/2 + 1/4 + k}e^{-n\delta^2/4}$. Its squared sum over the input frequencies is finite. The same argument applies to $K'$, since the spatial derivatives of $\partial_t b_{\sigma,j,n}$ are uniformly bounded. Writing $\Gamma_J = K^*K$, differentiation gives
$P' = K'\Gamma_J^{-1}K^*+K\Gamma_J^{-1}K'^* -K\Gamma_J^{-1}\Gamma_J'\Gamma_J^{-1}K^*$.
All coefficient-space factors are bounded, so $\Vert P'v\Vert_{H^k(\omega'')} \leq C_{J,k}\Vert v\Vert $.

The exponential block estimate for $E$ gives $Ee^{B_0N}:\mathcal{F}_J\to H^k(\mathcal{C})$ and $E^*e^{B_0N}:\mathscr{H}\to\mathcal{F}_J$ boundedly. For the adjoint this uses both exponents, since $e^{B_0\vert m\vert } \leq e^{B_0n}e^{B_0\vert m - \sigma n\vert }$. The fixed angular conjugates of ${P^\perp}$ and $K^\dagger$ are bounded. Consequently
$\Vert EK^\dagger U\Vert_{H^k(\mathcal{C})} + \Vert E^*{P^\perp}U\Vert_{\mathcal{F}_J} \leq C_{J,k}\Vert e^{-B_0N}U\Vert$.

Using $R = {P^\perp}EK^\dagger$ and $R^* = K\Gamma_J^{-1}E^*{P^\perp}$, the local bounds for $K$ imply
$\Vert (R^*-R)U\Vert_{H^k(\omega'')} \leq C_{J,k}\Vert e^{-B_0N}U\Vert$.
Similarly, ${P^\perp}F_{\rm obs} = F_{\rm obs} - K\Gamma_J^{-1}K^*F_{\rm obs}$ is controlled locally in $H^k$ by the corresponding source norm. Thus the entire right-hand side of \eqref{tail-equation} lies in $L^2(I_0;H^k(\omega''))$, with the displayed bound.

The operator is uniformly elliptic on $\omega''$. Interior parabolic estimates on nested spatial and time intervals give \eqref{tail-graphs}: first use the $L^2$ estimate, then differentiate in space and absorb the commutators on the larger intervals. No higher time derivative of $K$ is required. The rough estimate controls $e^{-B_0N}U$, and observed parabolic regularity controls all fixed spatial derivatives of $F_{\rm obs}$.
\end{proof}

\subsection{The rough weighted observation estimate}\label{rough-appendix}
We prove Proposition~\ref{rough-state}, which supplies the first observation estimate in Subsection~\ref{sec:sharp-upper}. Appendix~\ref{app:angular-weight} controls the commutators of the angular weight, and Appendix~\ref{app:carleman} proves the Carleman estimate. Appendix~\ref{proof:rough-state} applies it to the heat solution. The resulting angular loss $B_0 \leq Cd^2$ depends on the thickness of the unobserved tube, independently of time and of the collar edges.

Recall that $\triangle_\Phi = -\partial_r(c\partial_r) + D_\theta(r^2c)D_\theta + W$, $N = \vert D_\theta\vert$, $c(0,\theta) = 1$.
We use the Dirichlet realization and form domain of Subsection~\ref{upper-functional-setting}. The coefficients are real analytic on a fixed tube, with $c>0$. If the unobserved set lies in $\vert r\vert \leq d$, the strips $2d<\vert r\vert <b$, where $b = 3d$, are fully observed. All coefficient bounds below hold on a fixed smaller complex angular strip; $d$ is chosen so that the exponential weights fit within this strip.

\subsubsection{The angular weight}\label{app:angular-weight}
All Fourier indices belong to $(2\pi/\ell)\Z$. We use $e_n(\theta) = \ell^{-1/2}e^{in\theta}$ and $\widehat a_j(r) = \ell^{-1}\int_0^\ell a(r,\theta)e^{-ij\theta}d\theta$. The angular weight is exponential at high frequencies and vanishes below the Carleman parameter, where the scalar weight provides coercivity.

Fix $\kappa_0>2$ and a smooth function $f:[0,+\infty)\to\R$ with $f(z) = 0$ for $z \leq 1$, $f(z) = z$ for $z \geq 2$, and $0 \leq f(z) \leq z$. For $\lambda \geq 1$, set
$g_\lambda(\nu) = \lambda f(\vert \nu\vert /\lambda)$, $\mathsf{O} = g_\lambda(D_\theta)$ and $\Lambda_\lambda = \lambda + N$.

Thus $\mathsf{O} = 0$ below $\lambda$ and $\mathsf{O} = N$ above $2\lambda$, while $\vert g_\lambda'\vert \leq \ell_g$ and $\vert g_\lambda''(\nu)\vert \leq C/(\lambda + \vert \nu\vert )$. The operators $\mathsf{O}$ and $\Lambda_\lambda$ are selfadjoint on $D(N)$. The cutoff is needed in \eqref{rough:commutator-gain}: after normalization of the normal second derivative, $[C_R,C_I]$ multiplies $\partial_t^2$. Its absorption requires a gain $(\lambda + \vert \nu\vert )^{-1}$, which $\vert \nu\vert $ alone does not provide near zero. On the block $\vert m\vert ,\vert n\vert \leq \lambda$, conjugation preserves selfadjointness. Elsewhere, its cost $e^{B\ell_g\vert m - n\vert }$ is absorbed by analytic Fourier decay.

\begin{lem}\label{rough-comm}
Let $a$ be a real periodic analytic coefficient and let $\psi\in\R$ satisfy $\vert \psi\vert \leq B$, with $B\ell_g$ smaller than its analytic strip width. Define $C$ initially on finite Fourier sums by $C = e^{\psi\mathsf{O}}ae^{-\psi\mathsf{O}}$, and extend it through the matrix $C_{mn} = \widehat a_{m - n}e^{\psi(g_\lambda(m) - g_\lambda(n))}$. This matrix defines a bounded operator on $L^2(\mathbb{T}_\ell)$. Set $C_R = (C + C^*)/2$ and $C_I = (C - C^*)/(2i)$. Then
\begin{equation}\label{rough:commutator-gain}
\Vert\Lambda_\lambda^{1/2}[C_R,C_I]\Lambda_\lambda^{1/2}\Vert_{\mathcal{L}(L^2(\mathbb{T}_\ell))} \leq C\vert \psi\vert \Big(\sum_{j\ne0}(1 + \vert j\vert )^3 e^{B\ell_g\vert j\vert }\vert \widehat a_j\vert \Big)^2.
\end{equation}
The constant is uniform in $\lambda \geq 1$ and on regions where $a$ has fixed analytic bounds. An additional strip weight $e^{\eta\vert m - n\vert }$ is allowed if $B\ell_g + \eta$ stays below the coefficient strip width, adding $e^{\eta\vert j\vert }$ to the Fourier sum.
\end{lem}

\begin{proof}
Write $g = g_\lambda$. The two entries are
$(C_R)_{mn} = \widehat a_{m - n}\cosh(\psi(g(m) - g(n)))$ and $(C_I)_{mn} = -i\widehat a_{m - n}\sinh(\psi(g(m) - g(n)))$.
Pairing the intermediate index $q$ with $m + n - q$ gives
$[C_R,C_I]_{mn} = -i\sum_q\widehat a_{m - q}\widehat a_{q - n} \sinh(\psi u)\cosh(\psi v)$,
where $u = g(q) + g(m + n - q) - g(m) - g(n)$ and $v = g(q) - g(m + n - q)$. For $j_1 = m - q$, $j_2 = q - n$, $\mathfrak{j} = \vert j_1\vert + \vert j_2\vert $, the second-difference formula and the bound on $g_\lambda^{\prime\prime}$ imply
$$
\vert u\vert \leq \frac{C\vert j_1j_2\vert (1 + \mathfrak{j})}{ \sqrt{(\lambda + \vert m\vert )(\lambda + \vert n\vert )}}, \qquad \vert u\vert + \vert v\vert \leq \ell_g\mathfrak{j}.
$$
The Schur bound of Lemma~\ref{app:block-calculus} proves \eqref{rough:commutator-gain}. The factors $j_1j_2$ remove the angular mean; a further fixed strip weight costs only $e^{\eta\mathfrak{j}}$.
\end{proof}

\subsubsection{The Carleman estimate}\label{app:carleman}
A scalar Carleman weight alone leaves the negative term $-r\lambda\upsilon'N^2$ uncontrolled on $r>0$ at high angular frequencies. We add an angular exponent $\psi$ with $p=\psi'>0$ and $p'=\kappa_0>0$. In the spatial energy, its leading high-frequency contribution is $p(pp'-r)N^3$, which is positive for the choice below. A quadratic $\psi$ has size $\mathrm{O}(b^2)$ on the tube, so the resulting angular loss has that size. The scalar exponential controls the low frequencies and vanishes at both time endpoints.

On a time interval $(0,t_0)$ choose $p_0>2$ and a parameter $s \geq 1$, and set
$$
\lambda = s[t(t_0 - t)]^{-p_0},\quad p(r) = \kappa_0(r + 2b),\quad \psi(r) = \frac{\kappa_0}{2}\big((r + 2b)^2 - 9b^2\big),\quad B_0 = 4\kappa_0b^2.
$$
Take $s$ large enough that $\min_t\lambda(t) \geq 1$. On $\vert r\vert \leq b$, $p>0$ and $-B_0 \leq \psi \leq 0$. Choose $K \geq 1$ later and set $\upsilon' = e^{Kr}$, $\upsilon = K^{-1}e^{Kr} - E_0$, with $E_0>e$. Thus $\upsilon<0$ whenever $Kb \leq 1$. Define
$\Gamma = \lambda\upsilon$ and $\mathcal{W} = e^{\psi\mathsf{O}}$.
The multipliers $\mathcal{W}(t,r)$ are contractions on $L^2(\mathbb{T}_\ell)$. The test class is $\mathcal{D} = C_c^\infty((0,t_0)\times(-b,b);C^\infty(\mathbb{T}_\ell))$.

\begin{prop}\label{rough-carleman}
There exists $K_*>0$, depending on the coefficients, $\kappa_0$ and $p_0$, such that, for every $K \geq K_*$, one can choose $0<b_1(K) \leq K^{-1}$ with the following property. For $0<b<b_1(K)$ and $t_0>0$, there are $s_1>0$ and $c_*>0$ such that, for all $s \geq s_1$ and $v\in\mathcal{D}$, setting $z = e^\Gamma\mathcal{W}v$ gives
$$
\Vert e^\Gamma\mathcal{W}(\partial_t + \triangle_\Phi)v\Vert^2 \geq c_*\big(\Vert\Lambda_\lambda^{1/2}z_r\Vert^2+ \Vert\lambda^{3/2}z\Vert^2 + b^2\Vert N^{3/2}z\Vert^2\big).
$$
The norms are those of $L^2((0,t_0)\times(-b,b)\times\mathbb{T}_\ell)$. The constants $s_1,c_*$ may depend on $b,K,t_0$ and the fixed coefficient bounds. The choice of $b_1(K)$ is independent of $t_0$ and $s$.
\end{prop}
We first compute the conjugated operator. Lemma~\ref{rough-spatial-energy} gives the positive spatial energy, and Lemma~\ref{rough-coupling-energy} controls the terms involving time derivatives. The proof of Proposition~\ref{rough-carleman} follows these two lemmas.

\paragraph{Conjugation.}
All norms below are space-time norms on this cylinder, and angular inner products are integrated in $(t,r)$. Compact support removes the boundary terms. For fixed $t_0$, $\vert \lambda'\vert \lesssim\lambda^{1 + 1/p_0}$, $\vert \lambda''\vert \lesssim\lambda^{1 + 2/p_0}$ and $\Vert\mathsf{O}_t\Vert\lesssim\lambda^{1 + 1/p_0}$, since $\partial_\lambda g_\lambda = f(z) - zf'(z)$ is supported in $1 \leq z \leq 2$.

Set $\Xi = \lambda\upsilon'+p\mathsf{O}$ and $F = \Gamma_t + \psi\mathsf{O}_t$. For fixed $t\in(0,t_0)$, $\Xi(t,r)$ has domain $D(N)$ and $F(t,r)$ is bounded.

On a fixed smaller complex tube, $c^{-1}$ is analytic. Define $\widetilde c = \mathcal{W}c\mathcal{W}^{-1}$ and $\mathfrak{a} = \mathcal{W}c^{-1}\mathcal{W}^{-1} = \widetilde c^{-1}$ by their bounded Fourier matrices, as in Lemma~\ref{rough-comm}, and set $A = (\mathfrak{a} + \mathfrak{a}^*)/2$, $B_1 = (\mathfrak{a} - \mathfrak{a}^*)/(2i)$. Since $\mathcal{W}^{-1}$ is unbounded, conjugation is first computed on functions with finite angular Fourier support. The resulting differential expression extends to $\mathcal{D}$ by the Fourier estimates; no inverse weight is applied to general test functions. Since $\widetilde c_r = \mathcal{W}c_r\mathcal{W}^{-1}+[\Xi,\widetilde c]$, expansion gives
\begin{equation}\label{rough:conjugated-operator}
Q_\Gamma = \mathfrak{a} e^\Gamma\mathcal{W}(\partial_t + \triangle_\Phi)\mathcal{W}^{-1}e^{-\Gamma} = \mathfrak{a}\partial_t - \partial_r^2 + 2\Xi\partial_r + \Xi_r + V_0 + \mathcal{R} - \mathfrak{a} F
\end{equation}
with $V_0 = r^2N^2 - \Xi^2$ and $\Vert\mathcal{R}z\Vert \leq C\big(\Vert z_r\Vert + \Vert(\lambda + bN)z\Vert + \Vert z\Vert\big)$.
Explicitly, on $\mathcal{D}$,
$$
\mathcal{R} = -h\partial_r + h\Xi + r^2\mathfrak{a}[N,\widetilde c]N + r^2\mathfrak{a}(D_\theta\widetilde cD_\theta - N\widetilde cN) + \mathfrak{a}\mathcal{W}W\mathcal{W}^{-1},
$$
with $h = \mathfrak{a}\mathcal{W}c_r\mathcal{W}^{-1}$.
The matrix $D_\theta\widetilde cD_\theta - N\widetilde cN$ vanishes for equal frequency signs, and for opposite signs $2\vert mn\vert \leq \vert m - n\vert ^2/2$. Fourier moments give the stated bound on $\mathcal{R}$.

To estimate \eqref{rough:conjugated-operator}, set
$S_0 = -\partial_r^2 + V_0$, $K_0 = 2\Xi\partial_r + \Xi_r$, $T_A = A\partial_t + A_t/2$, $T_B = B_1\partial_t + (B_1)_t/2$.
Then $S = S_0 + iT_B$ is symmetric, $\mathcal{K} = K_0 + T_A$ is skew-symmetric, and the spatial commutator is expressed through
\begin{equation}\label{rough:spatial-energy}
\E(z) = \int\langle \Xi_rz_r,z_r\rangle + \int\langle(\Xi^2\Xi_r - r\Xi N^2)z,z\rangle.
\end{equation}

\paragraph{Spatial coercivity.}
\begin{lem}\label{rough-spatial-energy}
For fixed $K$ and sufficiently small $b$, there is $c_K>0$ such that, for every $z\in\mathcal{D}$,
\begin{equation}\label{rough:energy-coercivity}
\E(z) \geq c_K\big(\Vert\Lambda_\lambda^{1/2}z_r\Vert^2+ \Vert\Lambda_\lambda^{1/2}(\lambda + bN)z\Vert^2\big) 
\geq c'_K\big(\Vert\Lambda_\lambda^{1/2}z_r\Vert^2+ \Vert\lambda^{3/2}z\Vert^2 + b^2\Vert N^{3/2}z\Vert^2\big).
\end{equation}
Moreover, $\langle[S_0,K_0]z,z\rangle = 4\E(z) - \langle \Xi_{rrr}z,z\rangle$, and the final term is $\mathrm{o}_s(1)\E(z)$ for fixed $K,b$ as $s\to + \infty$.
\end{lem}
\begin{proof}
Integration by parts gives the commutator identity. For positivity, we compare the two ranges of the angular multiplier. Writing $n$ for the absolute value of an angular frequency, for $n \geq 2\lambda$ one has $\mathsf{O}_n = n$ and $\Xi\Xi_r \geq \kappa_0^2bn^2$; for $n \leq 2\lambda$, $rn^2 \leq 4b\lambda^2$ and $\Xi\Xi_r \geq \lambda^2\upsilon'\upsilon''$. Thus, for small $b$, one has
$\Xi^2\Xi_r - r\Xi N^2 \geq \frac{1}{2}\Xi^2\Xi_r$, $\Xi_r\asymp_K\Lambda_\lambda$, and $\Xi/\Xi_r \leq K^{-1} + 3b$.
Since $\Xi \geq c_K\lambda$ everywhere and $\Xi \geq \kappa_0b n$ for $n \geq 2\lambda$, the first two inequalities imply \eqref{rough:energy-coercivity}. For $n \leq 2\lambda$, its angular term follows from the scalar $\lambda^3$ term. Since $\Xi_{rrr} = \lambda\upsilon^{(4)}$ and $\E$ controls $\Vert\lambda^{3/2}z\Vert^2$, its commutator contribution is lower order as $s\to + \infty$.
\end{proof}

\paragraph{Time derivatives and mixed commutators.}
The spatial energy $\E(z)$ does not control time derivatives. We keep the positive term $\Vert\mathcal{K}z\Vert^2$ and use the identity defining $\mathcal{K}$ to estimate $\Lambda_\lambda^{-1/2}z_t$ in Lemma~\ref{rough-coupling-energy}. This is the norm needed to control the mixed commutators. Their coefficient bounds are expressed in the weighted Schur norms
$$
\Vert T\Vert_{\eta,q} = \max\Big(\sup_m\sum_n e^{\eta\vert m - n\vert }(1 + \vert m - n\vert )^q\vert T_{mn}\vert , \ \sup_n\sum_m e^{\eta\vert m - n\vert }(1 + \vert m - n\vert )^q\vert T_{mn}\vert \Big).
$$
Here $\eta>0$ is a fixed strip width and $q$ is any fixed nonnegative integer. Lemma~\ref{app:block-calculus} controls operator norms by these norms. Put $u_{mn} = g_\lambda(m)-g_\lambda(n)$. The entries of $A$ and $B_1$ are $\widehat{c^{-1}}_{m-n}\cosh(\psi u_{mn})$ and $-i\widehat{c^{-1}}_{m-n}\sinh(\psi u_{mn})$. The nonconstant Fourier coefficients of $c^{-1}$ are $\mathrm{O}(b)$ because $c(0,\theta) = 1$. Thus the cosh correction is $\mathrm{O}(b^5)$ and the sinh term is $\mathrm{O}(b^3)$, since $\psi = \mathrm{O}(b^2)$. Normal differentiation uses $\psi' = \mathrm{O}(b)$ and $\psi'' = \kappa_0$; time differentiation uses
$$
\vert (u_{mn})_t\vert \leq C\frac{\vert\lambda'\vert}{\lambda}\vert m-n\vert,\qquad \vert (u_{mn})_{tt}\vert \leq C\left(\frac{\vert\lambda''\vert}{\lambda}+\frac{\vert\lambda'\vert^2}{\lambda^2}\right)\vert m-n\vert.
$$
The extra frequency differences are absorbed by fixed Fourier moments. This proves the following $\mathrm{O}(\cdot)$ bounds, uniform in $r$:
\begin{equation}\label{rough:coefficient-bounds}
\begin{gathered}
\begin{array}{c@{\qquad}ccccc} \text{coefficient}&\text{size}&\partial_r&\partial_r^2&\partial_t&\partial_t^2 \\
\hline A - I&b&1&1&b^5\lambda^{1/p_0}&b^5\lambda^{2/p_0} \\
B_1&b^3&b^2&b&b^3\lambda^{1/p_0}&b^3\lambda^{2/p_0} \end{array} \\[4pt]
(A_t)_r = \mathrm{O}(b^4\lambda^{1/p_0}),\qquad (A_t)_{rr} = \mathrm{O}(b^3\lambda^{1/p_0}),\qquad (B_1)_{rt} = \mathrm{O}(b^2\lambda^{1/p_0}).
\end{gathered}
\end{equation}
The corresponding commutator bounds are 
$[\Xi,A] = \mathrm{O}(b^2)$, $[\Xi,B_1] = \mathrm{O}(b^4)$, $[\Xi_r,B_1] = \mathrm{O}(b^3)$, 
$\Vert\Lambda_\lambda^{1/2}[A,B_1]\Lambda_\lambda^{1/2}\Vert \leq Cb^4$, $\Vert\Lambda_\lambda^{1/2}[V_0,A]z\Vert \leq C_Kb^2\E(z)^{1/2}$.
The first three use the same Fourier moments; the bound on $[A,B_1]$ uses \eqref{rough:commutator-gain}. For the last bound expand $[V_0,A] = r^2[N^2,A]-[\Xi^2,A]$; its order-one coefficient is $\mathrm{O}(b^2\lambda + b^3N)$. The energy controls $\Vert\Lambda_\lambda^{1/2}(\lambda + bN)z\Vert^2$ by retaining $\Xi^2\Xi_r$ in \eqref{rough:spatial-energy}. This proves the bound also in the intermediate range $\lambda\ll N\ll\lambda/b$.

\begin{lem}\label{rough-coupling-energy}
For fixed $K$, sufficiently small $b$, and $z\in\mathcal{D}$, the temporal derivative satisfies
\begin{equation}\label{rough:temporal-derivative}
\mathcal{T}_t = \Vert\Lambda_\lambda^{-1/2}z_t\Vert^2 \leq C_K\big(\lambda_{\min}^{-1}\Vert\mathcal{K}z\Vert^2 + \E(z)\big),
\qquad\lambda_{\min} = \min_t\lambda(t).
\end{equation}
For every fixed $\varepsilon>0$, the mixed commutators satisfy
\begin{equation}\label{rough:mixed-commutators}
\begin{aligned}
\vert \langle[S_0,T_A]z,z\rangle\vert & \leq \varepsilon\Vert\mathcal{K}z\Vert^2+ [C(K^{-1} + b) + C_Kb^2+\mathrm{o}_s(1)]\E(z), \\
\vert \langle[iT_B,K_0]z,z\rangle\vert & \leq \mathrm{o}_s(1)\Vert\mathcal{K}z\Vert^2+[C_Kb^2+\mathrm{o}_s(1)]\E(z), \\
\vert \langle[iT_B,T_A]z,z\rangle\vert & \leq \mathrm{o}_s(1)\Vert\mathcal{K}z\Vert^2+[C_Kb^4+\mathrm{o}_s(1)]\E(z).
\end{aligned}
\end{equation}
Here $\mathrm{o}_s(1)\to0$ as $s\to + \infty$ after $K,b$ have been fixed. The constant in $C/K$ is uniform when $Kb \leq 1$.
\end{lem}
\begin{proof}
We first prove \eqref{rough:temporal-derivative}. The Schur bounds in \eqref{rough:coefficient-bounds}, together with $\langle m\rangle/\langle n\rangle \leq 1 + \vert m - n\vert $, imply that $\Lambda_\lambda^{\pm1/2}A^{-1}\Lambda_\lambda^{\mp1/2}$ are uniformly bounded: a Neumann series applies to $A - I = \mathrm{O}(b)$ in the same weighted matrix norms. Substitute $z_t = A^{-1}(\mathcal{K}z - 2\Xi z_r - \Xi_rz - A_tz/2)$. The first term costs $C\lambda_{\min}^{-1}\Vert\mathcal{K}z\Vert^2$; the others are bounded by \eqref{rough:energy-coercivity}, since $\Xi^2/\Lambda_\lambda\lesssim_K\Lambda_\lambda$. This proves \eqref{rough:temporal-derivative}.

We now prove the three bounds in \eqref{rough:mixed-commutators}. Write $E = \E(z)$ and $X = \Vert\mathcal{K}z\Vert$. Besides \eqref{rough:temporal-derivative}, we use the two consequences of \eqref{rough:energy-coercivity}: 
$\Vert\Lambda_\lambda^{1/2}z\Vert \leq C_K\lambda_{\min}^{-1}E^{1/2}$ and $\Vert\lambda^{3/2}z\Vert \leq C_KE^{1/2}$.
All matrix estimates below include the fixed Fourier moments needed to move $\Lambda_\lambda^{1/2}$ across a coefficient.

\smallskip
\noindent
\emph{The commutator $[S_0,T_A]$.} Expanding gives
$[S_0,T_A] = -2A_r\partial_r\partial_t-A_{rr}\partial_t+[V_0,A]\partial_t-A(V_0)_t+\frac{1}{2}[S_0,A_t]$.
Integrating the first two terms in $r$ leaves $2\langle A_rz_t,z_r\rangle+\langle A_{rr}z_t,z\rangle$. Substitute the formula for $z_t$ in the first term. The four resulting contributions satisfy
\begin{align*}
2\vert\langle A_rA^{-1}\mathcal{K}z,z_r\rangle\vert &\leq \varepsilon X^2+C_\varepsilon\lambda_{\min}^{-1}E,\\
4\vert\langle A_rA^{-1}\Xi z_r,z_r\rangle\vert &\leq C(K^{-1}+b)E,\\
2\vert\langle A_rA^{-1}\Xi_rz,z_r\rangle\vert+\vert\langle A_rA^{-1}A_tz,z_r\rangle\vert &\leq C_{K,b}\lambda_{\min}^{-1+1/p_0}E.
\end{align*}
The second bound uses $\Xi/\Xi_r \leq K^{-1}+3b$; its constant is uniform for $Kb \leq 1$. The other terms containing $z_t$ are bounded by
$\vert\langle A_{rr}z_t,z\rangle\vert \leq C_K\lambda_{\min}^{-1}E^{1/2}\mathcal{T}_t^{1/2}$ and $\vert\langle[V_0,A]z_t,z\rangle\vert \leq C_Kb^2E^{1/2}\mathcal{T}_t^{1/2}$.
For the remaining terms, $(V_0)_t = -2\Xi\Xi_t$ and
$\frac{1}{2}[S_0,A_t] = -(A_t)_r\partial_r-\frac{1}{2}(A_t)_{rr}+\frac{1}{2}[V_0,A_t]$.
Since $\vert\Xi_t\vert\lesssim_K\lambda^{1+1/p_0}$, \eqref{rough:energy-coercivity}-\eqref{rough:coefficient-bounds} give
$\vert\langle A(V_0)_tz,z\rangle\vert+\frac{1}{2}\vert\langle[S_0,A_t]z,z\rangle\vert \leq C_{K,b}\lambda_{\min}^{-1+1/p_0}E$.
In the last term, expand $[V_0,A_t] = r^2[N^2,A_t]-[\Xi^2,A_t]$ as above; its order-one coefficients are bounded by $C_{K,b}\lambda^{1/p_0}(\lambda+N)$. These bounds and \eqref{rough:temporal-derivative} prove the first line of \eqref{rough:mixed-commutators}.

\smallskip
\noindent
\emph{The commutator $[iT_B,K_0]$.} Its full expansion is
\begin{multline*}
[iT_B,K_0]/i = \big(2[B_1,\Xi]\partial_r-2\Xi(B_1)_r+[B_1,\Xi_r]\big)\partial_t \\
+ B_1(2\Xi_t\partial_r+(\Xi_r)_t)+[(B_1)_t,\Xi]\partial_r-\Xi(B_1)_{rt}+\frac{1}{2}[(B_1)_t,\Xi_r].
\end{multline*}
One integration in $r$ transforms the first line of its quadratic form into the sum of the three terms $-2\langle[B_1,\Xi]z_t,z_r\rangle$ and $-2\langle\Xi(B_1)_rz_t,z\rangle$ and $\langle([B_1,\Xi_r]-2[B_1,\Xi]_r)z_t,z\rangle$.
Their absolute values are bounded, respectively, by $C_Kb^4E^{1/2}\mathcal{T}_t^{1/2}$, $C_Kb^2E^{1/2}\mathcal{T}_t^{1/2}$ and $C_Kb^3\lambda_{\min}^{-1}E^{1/2}\mathcal{T}_t^{1/2}$.
Here $[B_1,\Xi]_r = [(B_1)_r,\Xi]+[B_1,\Xi_r] = \mathrm{O}(b^3)$. Each of the four terms in the second line of the expansion contains at most one normal derivative. Their forms are bounded in sum by $C_{K,b}\lambda_{\min}^{-1+1/p_0}E$, using \eqref{rough:coefficient-bounds} and $\vert\Xi_t\vert+\vert(\Xi_r)_t\vert\lesssim_K\lambda^{1+1/p_0}$. Substitution of \eqref{rough:temporal-derivative} proves the second line of \eqref{rough:mixed-commutators}.

\smallskip
\noindent
\emph{The commutator $[iT_B,T_A]$.} Here a second time derivative occurs:
\begin{equation*}
[iT_B,T_A] = i[B_1,A]\partial_t^2 + i([B_1,A]_t/2 + B_1A_t - A(B_1)_t)\partial_t  + i(B_1A_{tt} - A(B_1)_{tt})/2 + i[(B_1)_t,A_t]/4.
\end{equation*}
After one time integration, its form is the sum of $-i\langle[B_1,A]z_t,z_t\rangle$, $i\langle F_1z_t,z\rangle$ and $i\langle F_0z,z\rangle$, where
$F_1 = B_1A_t-A(B_1)_t-\frac{1}{2}[B_1,A]_t$ and $F_0 = \frac{1}{2}(B_1A_{tt}-A(B_1)_{tt})+\frac{1}{4}[(B_1)_t,A_t]$.
The gain in \eqref{rough:commutator-gain} is essential for the first term. Together with \eqref{rough:coefficient-bounds}, it gives
$$
\vert\langle[B_1,A]z_t,z_t\rangle\vert\leq Cb^4\mathcal{T}_t,\quad
\vert\langle F_1z_t,z\rangle\vert\leq C_Kb^3\lambda_{\min}^{-1+1/p_0}E^{1/2}\mathcal{T}_t^{1/2},\quad
\vert\langle F_0z,z\rangle\vert\leq C_Kb^3\lambda_{\min}^{-3+2/p_0}E.
$$
Indeed $F_1 = \mathrm{O}(b^3\lambda^{1/p_0})$ and $F_0 = \mathrm{O}(b^3\lambda^{2/p_0})$ in the weighted Schur norms. Applying \eqref{rough:temporal-derivative} proves the last line of \eqref{rough:mixed-commutators}. All negative powers of $\lambda_{\min}$ tend to zero as $s\to+\infty$ with $K,b$ fixed.
\end{proof}

\begin{proof}[Proof of Proposition~\ref{rough-carleman}]
Apply $\Vert(S + \mathcal{K})z\Vert^2 = \Vert Sz\Vert^2 + \Vert\mathcal{K}z\Vert^2+
\langle[S,\mathcal{K}]z,z\rangle$.
Lemma~\ref{rough-spatial-energy} gives the positive spatial term, and Lemma~\ref{rough-coupling-energy} bounds the other three commutators. Fix a small $\varepsilon$ in \eqref{rough:mixed-commutators}, then choose $K$ to make the term $C/K$ small. With $K$ fixed, decrease $b$ so that $Kb\leq1$ and the terms $Cb$, $C_Kb^2$ and $C_Kb^4$ can be absorbed by the spatial energy. Finally increase $s$ to absorb the remaining $\mathrm{o}_s(1)$ terms into $\E(z)$ and $\Vert\mathcal{K}z\Vert^2$. The difference $Q_\Gamma - (S + \mathcal{K}) = \mathcal{R} - \mathfrak{a} F - (A_t + i(B_1)_t)/2$ has squared norm $\mathrm{o}_s(1)\E(z)$, since $2 + 2/p_0<3$. We obtain, for every $z\in\mathcal{D}$,
\begin{equation}\label{rough:conjugated-carleman}
\Vert Q_\Gamma z\Vert^2 \geq c\big(\Vert Sz\Vert^2 + \Vert\mathcal{K}z\Vert^2 + \E(z)\big).
\end{equation}
For $z = e^\Gamma\mathcal{W}v$, the conjugation identity gives $Q_\Gamma z = \mathfrak{a} e^\Gamma\mathcal{W}(\partial_t + \triangle_\Phi)v$. The operator $\mathfrak{a}$ is bounded uniformly in $s$, so \eqref{rough:energy-coercivity} and \eqref{rough:conjugated-carleman} prove the stated estimate.
\end{proof}

\subsubsection{Proof of Proposition~\ref{rough-state}}\label{proof:rough-state}
A cutoff in the observed strips makes \eqref{rough:conjugated-carleman} applicable to the heat solution. A local energy estimate controls its source; propagation at a fixed negative angular weight then proves Proposition~\ref{rough-state}.

Fix $0<t_0<\tau$. Choose $\chi_d\in C_c^\infty(-b,b)$ equal to one on $\vert r\vert \leq 2d$, with derivative support strictly inside observed strips, where the fixed localization cutoff $\varkappa$ equals one. Let $J_d\Subset\{2d<\vert r\vert <b\}$ contain $\operatorname{supp}\chi_d'$. For the local unitary representative $u$ of the global solution $e^{-t\triangle}u_0$, with $u_0\in L^2(M,\nu)$, set $z = e^\Gamma\mathcal{W}\chi_du$. Then
$Q_\Gamma z = \mathfrak{a} e^\Gamma\mathcal{W}[\triangle_\Phi,\chi_d]u$ and $[\triangle_\Phi,\chi_d]u = -2c\chi_d'u_r - (c\chi_d')_ru$.
At each $r$, $\mathcal{W}$ is a contraction. With $\varpi (t) = e^{2\lambda(t)\upsilon_{\max}}$, where $\upsilon_{\max} = \max_{[-b,b]}\upsilon<0$, the source is bounded by $C\int_0^{t_0}\varpi \Vert u_r,u\Vert_{L^2(J_d\times\T_\ell)}^2dt$. Choose $\eta\in C_c^\infty(-r_0,r_0)$, supported in the fully observed strips and equal to one on $J_d$. Testing the heat equation with $\varpi \eta^2\bar u$ gives
\begin{equation}\label{rough:observed-energy}
\int_0^{t_0}\varpi \Vert u_r\Vert_{L^2(J_d\times\T_\ell)}^2dt \leq C\int_0^{t_0}(\varpi + \vert \varpi '\vert )\Vert u\Vert_{L^2(\omega)}^2dt \leq C\obs.
\end{equation}
The angular form is nonnegative because $\eta$ depends only on $r$. Compact support removes the normal boundary terms; $\varpi$ removes the temporal terms. Both $\varpi$ and $\vert \varpi'\vert $ are bounded.

We justify the application of \eqref{rough:conjugated-carleman} in three steps.

\smallskip
\noindent
\emph{Time regularization.} Replace $u(t)$ by $u_\delta(t) = e^{-(t+\delta)\triangle}u_0$, with $\delta>0$. Then $\varkappa u_\delta\in C^1([0,t_0];\mathscr{H})\cap C([0,t_0];D(\triangle_\Phi))$: its form energy is finite and $\triangle_\Phi (\varkappa u_\delta) = [\triangle_\Phi,\varkappa]u_\delta-\varkappa\partial_tu_\delta\in L^2$. Lemma~\ref{operator-graph} gives $\Vert v_{rr}\Vert+\Vert Nv\Vert \leq C(\Vert \triangle_\Phi v\Vert+\Vert v\Vert)$. Moreover, $b_d = -\max_{\operatorname{supp}\chi_d}\psi>0$. On this support, $\mathcal{W}$ exponentially smooths frequencies above $2\lambda$; the remaining frequencies satisfy $N \leq 2\lambda$. Every fixed angular or weight derivative therefore costs only a fixed polynomial in $\lambda$.

\smallskip
\noindent
\emph{Fourier truncation and convolution.} On a compact time interval in $(0,t_0)$, truncate the angular Fourier series of $z_\delta = e^\Gamma\mathcal{W}\chi_du_\delta$ and convolve in $(t,r)$, keeping support in the interior. This approximates $z_\delta$ simultaneously for $\Vert z\Vert+\Vert Q_\Gamma z\Vert+\Vert Sz\Vert+\Vert\mathcal{K}z\Vert+\E(z)^{1/2}$. The preceding derivative bounds and the Fourier moments of the coefficients control each term. Time cutoffs may then approach the endpoints, since $e^\Gamma$ dominates both their derivative costs and the fixed polynomials in $\lambda$.

\smallskip
\noindent
\emph{Passage to the limit.} Apply \eqref{rough:conjugated-carleman} to these test functions and pass through the approximations. Letting $\delta\downarrow0$ uses strong semigroup continuity for the observed solution in $L^2((0,T)\times\omega)$ and for the weighted solution on compact interior time intervals; Fatou's lemma gives the left-hand estimate. This proves \eqref{rough:conjugated-carleman}-\eqref{rough:observed-energy} for the original $L^2$ initial datum.

On $I_0 = [t_0/3,2t_0/3]$, $\Gamma$ is bounded below and $e^{\Gamma + \psi\mathsf{O}} \geq c_{d,t_0}e^{-B_0N}$ frequencywise. Equations \eqref{rough:energy-coercivity}, \eqref{rough:conjugated-carleman}-\eqref{rough:observed-energy} imply
$\int_{I_0}\Vert e^{-B_0N}\chi_du(t)\Vert^2dt \leq C\obs$.
The difference $U - \chi_du = (\varkappa - \chi_d)u$ is supported in observed normal strips, preserved by angular contraction. The same average therefore holds for $e^{-B_0N}U$.

To propagate the estimate, use the closed form realization $\triangle_{\Phi,-B_0}$ on $\mathcal{V}$ defined in Subsection~\ref{upper-functional-setting}. For $v = e^{-B_0N}U$, the equation $v_t + \triangle_{\Phi,-B_0}v = e^{-B_0N}F_{\rm obs}$ holds in $\mathcal{V}'$ on compact time intervals in $(0,T)$. The symmetric coefficient of $e^{-B_0N}ce^{B_0N}$ equals $c+\mathrm{O}(B_0^2)$ in angular operator norm, by the same cosh formula. For small $d$ its normal and angular principal forms remain positive; the bounded potential gives $\Re\mathfrak{h}_{-B_0}(v,v) \geq - C\Vert v\Vert^2$. The source remains contracted. Ordinary local parabolic regularity on its fully observed support yields $\int_{t_0/3}^{\sup I}\Vert F_{\rm obs}\Vert^2dt \leq C\obs$. Energy and Gronwall, averaged over the initial time in $I_0$, give
$\sup_{t\in I}\Vert e^{-B_0N}U(t)\Vert^2 \leq C_{d,\tau,I}\obs$.
Here $F_{\rm obs}\in L^2(\mathscr{H})$ and $v\in L^2(\mathcal{V})\cap H^1(\mathcal{V}')$, so time regularization justifies the energy identity. The loss $B_0 = 4\kappa_0b^2 \leq Cd^2$ is independent of $\tau$, $t_0$ and $I$.

\subsection{Factorial bounds for heat transport and stationary WKB}\label{app:factorial}
This appendix supplies the factorial bounds of Lemma~\ref{factorial} and the quasimode estimates used in Theorem~\ref{thm:periodic-quasimodes}. The two constructions use the same normal inverse, whose undetermined trace is fixed by a time evolution equation for heat and by periodic compatibility for stationary WKB. After the common estimates in Appendix~\ref{app:normal-inversion}, Appendix~\ref{proof:factorial} proves the heat transport bound. Appendices~\ref{app:wkb} and~\ref{app:exponential-quasimodes} construct the stationary hierarchy and truncate it to obtain exponentially small errors.

\subsubsection{Normal inversion and factorial weights}\label{app:normal-inversion}
On the fixed complex product
$\mathcal{P}_{R_0,\tau_0,s_0} = \{\vert x\vert<R_0,\ \vert t\vert<\tau_0,\ \vert\im z\vert<s_0\}$, $z\in\C/\ell\Z$,
all functions are holomorphic and periodic in $z$. The stationary construction omits $t$. For a nonvanishing holomorphic coefficient $\mathcal{B}$, define
$$
J_x\sum_{k \geq 0}x^kf_k = \sum_{k \geq 1}\frac{x^k}{k}f_k,\qquad \mathcal{S}f = J_x(\mathcal{B}^{-1}f).
$$
These maps preserve the product and have zero trace. Since
$x\mathcal{B}\partial_x\mathcal{S}f = f - \frac{\mathcal{B}(x,t,z)}{\mathcal{B}(0,t,z)}f(0,t,z)$,
$\mathcal{S}$ is the zero-trace inverse on zero-trace sources. Defining it on arbitrary sources also permits estimates of summands whose sum alone has zero trace.

Fix $R<R_0$, $\tau<\tau_0$ and $s<s_0$. For $u = \sum_{k,m \geq 0}x^kt^mu_{k,m}(z)$, set
$$
W_{n,k,d} = \frac{(2n + k + d)!}{(n + k + d)!},\qquad \mathcal{N}_n(u) = \sum_{k,m,p \geq 0}\frac{R^k\tau^ms^p}{W_{n,k,m + p}}\frac{\norm{\partial_z^pu_{k,m}}_{L^\infty(\T_\ell)}}{p!}.
$$
Take $m = 0$ in the stationary case. Write $\mathcal{N}_n^{\mathrm{tr}}$ for functions independent of $x$, and $\langle g\rangle = \ell^{-1}\int_0^\ell g(y)\dd{d}y$ for the stationary mean. The smaller radii make the coefficient norms finite; the weights will absorb the derivative losses without further shrinking at each order.

\begin{lem}\label{app:transport-calculus}
Let $\mathscr{P} = \partial_x^2 + xa\partial_x\partial_z + b\partial_x + c\partial_z + d + e\partial_t$ have fixed holomorphic coefficients on the larger product. There is $C>0$, independent of the hierarchy indices, such that
\begin{gather*}
\mathcal{N}_{n + 1}(u) \leq \mathcal{N}_n(u),\qquad \mathcal{N}_{a + b}(uv) \leq \mathcal{N}_a(u)\mathcal{N}_b(v),\qquad \mathcal{N}_n^{\mathrm{tr}}(u_j) \leq (2/R)^j\mathcal{N}_n(u), \\
\mathcal{N}_n(\mathcal{S}u) \leq C\mathcal{N}_n(u),\qquad \mathcal{N}_n(\mathcal{S}\mathscr{P}u) \leq C\mathcal{N}_{n - 1}(u)\quad(n\geq1),\qquad \norm{u}_{\infty,\mathcal{P}_{R/2,\tau/2,s/2}}\leq4^nn!\mathcal{N}_n(u).
\end{gather*}
Here $u_j$ is the coefficient of $x^j$. For a zero-trace function $r$, the map $\mathsf{T}r = (\mathscr{P}r)\vert_{x = 0} = 2r_2 + b(0,t,z)r_1$ satisfies $\mathcal{N}_n^{\mathrm{tr}}(\mathsf{T}r) \leq C\mathcal{N}_n(r)$. Time integration from $0$ to $t$ has norm at most $\tau$ on traces. The stationary mean has norm at most one, and the periodic primitive $J_z$ defined by $i\,\partial_zJ_zg = g - \langle g\rangle$, $\langle J_zg\rangle = 0$, is uniformly bounded in $\mathcal{N}_n^{\mathrm{tr}}$.
\end{lem}
\begin{proof}
The weights increase with $n,k,d$. Since $W_{n,k,d} = n!\binom{2n + k + d}{n}$, Vandermonde's identity gives
$W_{a,k_1,d_1}W_{b,k_2,d_2} \leq W_{a + b,k_1 + k_2,d_1 + d_2}$.
Taylor convolution and Leibniz's formula prove the product bound. The inequality $W_{n,j,d}/W_{n,0,d} \leq 2^j$ proves coefficient extraction and the bound for $\mathsf{T}$. Multiplication by $\mathcal{B}^{-1}$ is bounded, and division by $k \geq 1$ bounds $\mathcal{S}$.

For $\mathcal{S}\mathscr{P}$, absorb $\mathcal{B}^{-1}$ into the coefficients. At output normal degree $k \geq 1$, with $d = m + p$, the derivative ratios are
\begin{align*}
\frac{(k + 2)(k + 1)}{k}\frac{W_{n - 1,k + 2,d}}{W_{n,k,d}}& = \frac{(k + 2)(k + 1)}{k(n + k + d + 1)}\leq2, \\
\frac{k + 1}{k}\frac{W_{n - 1,k + 1,d}}{W_{n,k,d}}& = \frac{k + 1}{k(2n + k + d)}\leq1, \\
(p + 1)\frac{W_{n - 1,k,d + 1}}{W_{n,k,d}}& = \frac{p + 1}{2n + k + d}\leq1,\qquad (m + 1)\frac{W_{n - 1,k,d + 1}}{W_{n,k,d}}\leq1.
\end{align*}
A coefficient of normal degree $l$, time degree $a$ and angular derivative order $q$ lowers the input indices to $(k - l + 2,m - a,p - q)$ for $\partial_x^2$, and similarly for the other terms. Its weights and derivative factors are no larger than those just estimated. For $x\partial_x\partial_z$, the remaining factor $(k - l)/k$ is at most one. Summing the convolutions costs the finite coefficient $\mathcal{N}_0$ norms and fixed factors $R^{-2},R^{-1},s^{-1},\tau^{-1}$. This proves the transport bound. The absence of $\partial_z^2$ is essential.

Time integration raises $m$ and divides by $m + 1$. For the periodic primitive, subtraction of the mean costs at most two. If $\langle g\rangle = 0$ and $v = J_zg$, then $\norm{v}_\infty \leq 2\ell\norm{g}_\infty$ and, for $p \geq 1$, one has $\frac{v^{(p)}}{p!} = \frac{g^{(p - 1)}}{ip(p - 1)!}$ and $\frac{W_{n,0,p - 1}}{W_{n,0,p}} \leq 1$.
This gives a bound independent of $n$. Finally, $W_{n,k,d} \leq 4^n2^{k + d}n!$ and Taylor expansion about $\re z$ prove the supremum estimate.
\end{proof}

\subsubsection{Proof of Lemma~\ref{factorial}}\label{proof:factorial}
Work on a product strictly inside the coefficient domain of Subsection~\ref{sec:sharp-lower}, where $\mathcal{B}_t^{-1}$ is holomorphic. The normal differential operator is
$R_t = \partial_x^2 + xg_1\partial_x\partial_\zeta + p_t\partial_x + v_t$, where $g_1 = 2(\Phi_x)_t/x$ is holomorphic since $\Phi_x(0,z) = 0$. For zero-trace $r$, set $\mathsf{T}_tr = 2r_2 + p_t(0,\zeta)r_1$. Define successively
\begin{align*}
b_0(t,\zeta)& = \exp\Big(\int_0^t v_a(0,\zeta)\dd{d}a\Big),\qquad r_n = \mathcal{S}_t(R_t - \partial_t)b_{n - 1}, \\
c_n(t,\zeta)& = b_0(t,\zeta)\int_0^t b_0(a,\zeta)^{-1}\mathsf{T}_ar_n(a,\zeta)\dd{d}a,\qquad b_n = r_n + c_n\quad(n\geq1).
\end{align*}
Extend $b_0,c_n$ constantly in $x$. The first source has zero trace by the equation for $b_0$. At every subsequent order,
$(\mathscr{P}_tb_n)\vert_{x = 0} = (\partial_t - v_t(0,\zeta))c_n - \mathsf{T}_tr_n = 0$.
Thus the normal inverse solves the hierarchy, with $b_0(0,\zeta) = 1$ and $c_n(0,\zeta) = 0$ for $n \geq 1$.

All fixed coefficients and $b_0^{\pm1}$ have finite $\mathcal{N}_0$. Lemma~\ref{app:transport-calculus} gives
$\mathcal{N}_n(r_n) \leq C\mathcal{N}_{n - 1}(b_{n - 1})$, $\mathcal{N}_n(c_n) \leq C\mathcal{N}_n(r_n)$, and $\mathcal{N}_n(b_n) \leq C_1^n\mathcal{N}_0(b_0)$.
The supremum estimate yields $\norm{b_n}_\infty \leq C\Lambda^nn!$ on one fixed smaller product. A further fixed Cauchy shrink gives derivatives through order two. The exponential formula bounds $b_0$ away from zero there. Every domain choice is independent of $n$.

\subsubsection{The periodic WKB hierarchy}\label{app:wkb}

We now determine the stationary traces and energy corrections by periodic compatibility. The resulting factorial bounds hold on one fixed complex tube and allow truncation at an order proportional to the inverse semiclassical parameter.

\begin{prop}\label{app:wkb-hierarchy}
Fix a branch in \eqref{wkb:main-compatibility}, a nonzero leading constant, and the zero-mean normalization of every higher trace after division by the leading trace. The hierarchy \eqref{wkb:hierarchy} has periodic holomorphic solutions $a_n$ on a common complex neighborhood of $Z_0$ and real coefficients $\mu_n$, with $\mu_0 = 1$. On a fixed smaller complex tube, there are $C,\Lambda>0$ such that
\begin{equation}\label{wkb:theorem-factorial}
\sum_{\vert\alpha\vert\leq2}\norm{\partial^\alpha a_n}_{L^\infty} + \vert\mu_n\vert \leq C\Lambda^nn!\qquad\forall n\geq0.
\end{equation}
The leading trace is \eqref{wkb:leading-trace}. In particular, $\rho_0\vert a_0(0,\cdot)\vert^2$ is constant.
\end{prop}

\begin{proof}
\emph{Normal transport.}
Set $b = \partial_x\varphi(0,\cdot)$ and $\varphi_2 = \partial_x^2\varphi(0,\cdot)$. The phase expansion \eqref{geom:phase-expansion} gives \eqref{wkb:positive-phase}. Direct conjugation gives
$$
e^{-i\Phi/h}(P_h - h\mu)e^{i\Phi/h}a = p(x,y,\dd{d}\Phi)a + h\mathcal{L}a - h^2\mathcal{D}_\rho a - h(\mu - 1)a.
$$
The eikonal equation and coefficient comparison give \eqref{wkb:hierarchy}.

Use the product $\mathcal{U}_{R_0,s_0}$ of \eqref{wkb:gauge}, its inverse image $\widehat U$ under $(x,y)\mapsto(x,\Phi)$, and the strip $S_{s_0} = \{\vert\im z\vert<s_0\}\subset\C/\ell\Z$. Write $\mathcal{O}_\ell$ for periodic holomorphic functions on the indicated domain and $\mathcal{O}_{\ell,0}$ for those with zero normal trace. For $f\in\mathcal{O}_{\ell,0}(\widehat U)$ and $v\in\mathcal{O}_\ell(S_{s_0})$, the unique solution of $\mathcal{L}u = f$, $u(0,\cdot) = v$, is
$$
\widetilde u(x,z) = g_\rho(x,z)\Big(v(z) + \int_0^x\frac{\widetilde f(t,z)}{t\mathcal{B}(t,z)g_\rho(t,z)}\dd{d}t\Big).
$$
The source has zero trace, so the integrand is holomorphic at zero. For $f = 0$, this defines the extension $\mathsf{E}v(x,y) = g_\rho(x,\Phi(x,y))v(\Phi(x,y))$.

\smallskip
\noindent
\emph{Periodic compatibility.}
Set $\rho_0(y) = \rho(0,y)$ and
$d_x = (\partial_x\log\rho)(0,\cdot)$, $d_{xx} = (\partial_x^2\log\rho)(0,\cdot)$, $d_y = (\partial_y\log\rho)(0,\cdot)$.
In the original coordinates,
$$
\mathcal{L} = (2x + 2bx^2)\partial_x - 2ix^2\partial_y + (2b + d_x)x + \kappa_2x^2+\mathrm{O}(x^3),\qquad \kappa_2 = \frac{3}{2}(\varphi_2 + b^2) + d_xb + d_{xx} - id_y,
$$
where the remainder is a first-order operator whose coefficients vanish to order three. Writing $\mathsf{E}v = v + xv_1 + x^2v_2+\mathrm{O}(x^3)$ gives $2v_1 + (2b + d_x)v = 0$ and $4v_2 + (4b + d_x)v_1 - 2iv'+\kappa_2v = 0$, hence
$$
v_1 = -(b + d_x/2)v,\qquad v_2 = \frac{i}{2}v'+\Big(\frac{5}{8}b^2 + \frac{1}{2}d_xb + \frac{1}{8}d_x^2 - \frac{3}{8}\varphi_2 - \frac{1}{4}d_{xx} + \frac{i}{4}d_y\Big)v.
$$
The tangential terms of $\mathcal{D}_\rho$ vanish at $x = 0$, so
\begin{equation}\label{wkb:effective-operator}
H_{\mathrm{eff}}v = (\mathcal{D}_\rho\mathsf{E}v)\vert_{x = 0} = iv'+\frac{i}{2}d_yv + \mathcal{V}v,\qquad \mathcal{V} = \frac{5}{4}b^2 - \frac{3}{4}\varphi_2 - \frac{1}{4}d_x^2 - \frac{1}{2}d_{xx}.
\end{equation}
Conjugation by $\rho_0^{1/2}$ identifies this operator on $L^2(\T_\ell,\rho_0\dd{d}y)$, with domain $H^1(\T_\ell)$, with the self-adjoint operator $i\,\partial_y + \mathcal{V}$ on $L^2(\T_\ell,\dd{d}y)$. For $a_0 = \mathsf{E}v_0$, solvability at the next order requires
$(H_{\mathrm{eff}} + \mu_1)v_0 = 0$.
Its nonzero solutions are
\begin{equation}\label{wkb:leading-trace}
v_0(y) = C_*\rho_0(y)^{-1/2}\exp\Big(i\int_0^y(\mathcal{V}(t) + \mu_1)\dd{d}t\Big),\qquad C_*\neq0.
\end{equation}
Periodicity gives \eqref{wkb:main-compatibility}; therefore $\mu_1$ is real and $\rho_0\vert v_0\vert^2 = \vert C_*\vert^2$.

After a fixed shrink, $\widetilde a_0 = g_\rho v_0$ is nonvanishing. Set $\widetilde a_n = \widetilde a_0b_n$ and $\mathcal{Q} = (\widetilde a_0)^{-1}\widetilde{\mathcal{D}}_\rho\widetilde a_0 + \mu_1$. Then $(\widetilde a_0)^{-1}\widetilde{\mathcal{L}}\widetilde a_0 = x\mathcal{B}\partial_x$. The eikonal equation gives
$$
\widetilde{\mathcal{D}}_\rho = \partial_x^2 + 2\widetilde{\Phi_x}\partial_x\partial_z + \widetilde{\partial_x\log\rho}\partial_x + \widetilde{\mathcal{D}_\rho\Phi}\partial_z,
$$
with no $\partial_z^2$ term. Consequently,
$\mathcal{Q} = \partial_x^2 + xg_1\partial_x\partial_z + g_2\partial_x + g_3\partial_z + g_4$, with $g_3(0,z) = i$ and $g_4(0,z) = 0$,
where all coefficients are fixed and holomorphic. The trace identities follow from $(\mathcal{Q}c)\vert_{x = 0} = ic'$ for $c = c(z)$.

The hierarchy now reads
$$
x\mathcal{B}\partial_xb_n = \mathcal{Q}b_{n - 1} + \sum_{j = 1}^{n - 1}\mu_{j + 1}b_{n - 1 - j},\qquad b_0 = 1.
$$
Write $b_n = c_n(z) + r_n(x,z)$, with $r_n\vert_{x = 0} = 0$, $c_0 = 1$, $r_0 = 0$, and impose $\langle c_n\rangle = 0$ for $n \geq 1$. For $r = \sum_{k \geq 1}x^kr_k$, set
$\mathsf{T}_0r = (\mathcal{Q}r)\vert_{x = 0} = 2r_2 + g_2(0,z)r_1$.
The successive equations are
\begin{equation}\label{wkb:recursion}
r_n = \mathcal{S}\Big(\mathcal{Q}b_{n - 1} + \sum_{j = 1}^{n - 1}\mu_{j + 1}b_{n - 1 - j}\Big),\qquad c_n = -J_z\Big(\mathsf{T}_0r_n + \sum_{j = 1}^{n - 1}\mu_{j + 1}c_{n - j}\Big).
\end{equation}
The first normal source has zero trace because $g_4(0,z) = 0$. Subsequently the compatibility equation is $ic_n'+\mathsf{T}_0r_n + \mu_{n + 1} + \sum_{j = 1}^{n - 1}\mu_{j + 1}c_{n - j} = 0$. Its mean gives $\mu_{n + 1} = -\langle\mathsf{T}_0r_n\rangle$, and its zero-mean part fixes $c_n$. Both inverses preserve the domains, so every $b_n$ is holomorphic on $\mathcal{U}_{R_0,s_0}$.

\smallskip
\noindent
\emph{Factorial bounds.}
Use the stationary norms of Lemma~\ref{app:transport-calculus}. Set $A_0 = 1$ and, for $n \geq 1$,
$$
A_n = \mathcal{N}_n(b_n),\qquad R_n = \mathcal{N}_n(r_n),\qquad C_n = \mathcal{N}_n^{\mathrm{tr}}(c_n),\qquad V_n = \frac{\vert\mu_{n + 1}\vert}{W_{n,0,0}}.
$$
The zero trace of $r_n$ gives $A_n = R_n + C_n$. The common estimates applied to \eqref{wkb:recursion} give
$$
R_n \leq C\Big(A_{n - 1} + \sum_{j = 1}^{n - 1}V_jA_{n - 1 - j}\Big),\qquad V_n \leq CR_n,\qquad C_n \leq C\Big(R_n + \sum_{j = 1}^{n - 1}V_jC_{n - j}\Big).
$$
The first convolution is bounded at level $n - 1$, then at level $n$ by monotonicity; the last is bounded at level $n$. Thus $t_0 = 1$, $t_n = A_n + V_n$ satisfy
$$
t_n \leq K_0\Big(t_{n - 1} + \sum_{j = 1}^{n - 1}t_jt_{n - 1 - j} + \sum_{j = 1}^{n - 1}t_jt_{n - j}\Big).
$$
The majorant equation $\mathfrak{T} = K_0t(1 + \mathfrak{T})^2 + K_0\mathfrak{T}^2$, $\mathfrak{T}(0) = 0$, has a holomorphic solution near zero by the implicit function theorem. Its nonnegative Taylor coefficients dominate $t_n$, so $A_n + V_n \leq C_0C_1^n$. Hence
$\norm{b_n}_{\infty,\mathcal{U}_{R/2,s/2}} \leq 4^nn!A_n$ and $\vert\mu_{n + 1}\vert \leq 4^nn!V_n$.
Multiplication by $\widetilde a_0$, return to the original coordinates and one fixed Cauchy shrink prove \eqref{wkb:theorem-factorial} after increasing $\Lambda$.

\smallskip
\noindent
\emph{Reality of the energy coefficients.}
Fix a smooth cutoff $\chi$, supported in the construction tube and equal to one near the whole circle. For fixed $N$, set $a^{(N)} = \sum_{j = 0}^Nh^ja_j$, $E_N(h) = h\sum_{j = 0}^Nh^j\mu_j$, and
\begin{equation}\label{wkb:quasimode}
w_h^{(N)} = \chi e^{i\Phi/h}a^{(N)},\qquad u_h^{(N)} = w_h^{(N)}/\norm{w_h^{(N)}}.
\end{equation}
For $h = \ell/(2\pi n)$, these functions are periodic, extend smoothly by zero to $M$ and belong to $D(\triangle)$. The hierarchy gives the exact residual
\begin{equation}\label{wkb:exact-residual}
e^{-i\Phi/h}(P_h - E_N)e^{i\Phi/h}a^{(N)} = -h^{N + 2}\mathcal{D}_\rho a_N - \sum_{\substack{1 \leq j \leq N,\ 0 \leq k \leq N\\j + k \geq N + 1}}h^{j + k + 1}\mu_ja_k.
\end{equation}
Gaussian rescaling $x = \sqrt{h}t$ yields
\begin{equation}\label{wkb:normalization-limit}
h^{-1/2}\norm{\chi e^{i\Phi/h}a^{(N)}}^2\longrightarrow\sqrt{\pi}\int_0^\ell\rho_0(y)\vert v_0(y)\vert^2\dd{d}y>0.
\end{equation}
For fixed $N$, the residual gains the same factor $\mathrm{O}(h^{1/4})$ as the normalization. Cutoff derivatives are supported away from $x = 0$, hence contribute an exponentially small error. Therefore $\norm{(P_h - E_N)u_h^{(N)}} \leq C_Nh^{N + 2}$. Symmetry of $P_h$ gives $\vert\im E_N(h)\vert \leq C_Nh^{N + 2}$. If $\mu_0,\ldots,\mu_{N - 1}$ are real, then $\vert\im\mu_N\vert \leq C_Nh$. Letting the quantized sequence tend to zero proves reality by induction.
\end{proof}

\subsubsection{Truncation and exponential error}\label{app:exponential-quasimodes}

We truncate the hierarchy of Proposition~\ref{app:wkb-hierarchy} and include the cutoff error to obtain the normalized quasimodes of Theorem~\ref{thm:periodic-quasimodes}.

\begin{prop}\label{app:wkb-realization}
Let the hierarchy and $\Lambda$ be as in Proposition~\ref{app:wkb-hierarchy}. Let $\chi$ be a fixed smooth cutoff supported in the real slice of the tube where \eqref{wkb:theorem-factorial} holds, equal to one near the whole circle, and set $b_\chi = \inf_{\supp\dd{d}\chi}\im\Phi>0$. For every $0<\beta<\min(\Lambda^{-1},b_\chi)$, there is $0<\delta<\Lambda^{-1}$ such that $N(h) = \lfloor\delta/h\rfloor$ gives real energies $E(h) = E_{N(h)}(h) = h+\mathrm{O}(h^2)$ and
$\norm{(P_h - E(h))u_h} \leq C_\beta e^{-\beta/h}$, with $h = \ell/(2\pi n)\to 0$,
for the normalized states \eqref{wkb:quasimode}. The tube, hierarchy and $\Lambda$ are independent of $\beta$. Moreover, $\norm{xu_h}_{L^2(U,\nu)} = \mathrm{O}(\sqrt h)$ and the probability measures $\vert u_h\vert^2\nu$ converge weakly to $\dd{d}y/\ell$ on the singular circle.
\end{prop}
\begin{proof}
Fix $0<\delta<\Lambda^{-1}$ and $N = \lfloor\delta/h\rfloor$. For $1 \leq j \leq N$, the bound $j! \leq N^{j - 1}$ gives $(\Lambda h)^jj! \leq \Lambda h(\Lambda\delta)^{j - 1}$. Summation yields, uniformly with derivatives through order two,
\begin{equation}\label{wkb:uniform-truncation}
a^{(N(h))} = a_0+\mathrm{O}(h),\qquad E(h) = h+\mathrm{O}(h^2)\in\R.
\end{equation}
Write $\tau_m = (\Lambda h)^mm!$. The differentiated term of \eqref{wkb:exact-residual} is bounded by $Ch^2\tau_N$. In its double sum, write $j + k = N + m$, $1 \leq m \leq N$. Factorial convexity gives $j!k! \leq N!m!$, hence $\tau_j\tau_k \leq \tau_N\tau_m$. Also $\tau_m \leq \Lambda h$ for $1 \leq m \leq N$, since $\tau_{m + 1}/\tau_m \leq \Lambda\delta<1$. There are at most $N(N + 1)$ pairs. Thus the residual before multiplication by the phase has supremum norm at most $Ch^2(1 + \Lambda N(N + 1))\tau_N$. Stirling's formula gives
$\log\tau_N = -\frac{F_\Lambda(\delta)}{h}+\mathrm{O}(\log(1/h))$ with $F_\Lambda(\delta) = \delta(1 - \log(\Lambda\delta))$.
It follows that this residual is bounded by $C_{\beta_1}e^{-\beta_1/h}$ for every $\beta_1<F_\Lambda(\delta)$. Moreover, $F_\Lambda$ increases to $\Lambda^{-1}$ as $\delta\uparrow\Lambda^{-1}$.

The commutator is
$[P_h,\chi]f = -h^2\left(2\chi_x\partial_xf + 2Q\chi_y\partial_yf + (\mathcal{D}_\rho\chi)f\right)$.
By \eqref{wkb:uniform-truncation}, its norm on the truncated WKB state is at most $Ch e^{-b_\chi/h}$. The same estimate makes \eqref{wkb:normalization-limit} valid for $N = N(h)$. The interior residual gains a Gaussian factor $\mathrm{O}(h^{1/4})$, which cancels normalization, so
$$
\norm{(P_h - E(h))u_h} \leq Ce^{-\beta_1/h} + Ch^{3/4}e^{-b_\chi/h}.
$$
Choose $\beta<\beta_1<\Lambda^{-1}$ and then $\delta$ with $F_\Lambda(\delta)>\beta_1$. This proves the claim when $\beta<b_\chi$. Gaussian rescaling also gives concentration and the mass limit, since $\rho_0\vert v_0\vert^2$ is constant.
\end{proof}


\begin{thebibliography}{99}
\small
\setlength{\itemsep}{0pt}

\bibitem{Ag82}
S.~Agmon.
\newblock {\em Lectures on Exponential Decay of Solutions of Second-Order
Elliptic Equations: Bounds on Eigenfunctions of {$N$}-Body Schr{\"o}dinger Operators}.
\newblock Number~29 in Mathematical Notes. Princeton University Press,
Princeton, 1982.

\bibitem{ABB20}
A.~Agrachev, D.~Barilari, and U.~Boscain.
\newblock {\em A Comprehensive Introduction to Sub-Riemannian Geometry}.
\newblock Number~181 in Cambridge Studies in Advanced Mathematics. Cambridge
University Press, Cambridge, 2020.

\bibitem{ABS08}
A.~A. Agrachev, U.~Boscain, and M.~Sigalotti.
\newblock A {Gauss-Bonnet}-like formula on two-dimensional almost-{Riemannian}
manifolds.
\newblock {\em Discrete Contin. Dyn. Syst.}, 20(4):801-822, 2008.

\bibitem{AS23}
V.~Arnaiz and C.~Sun.
\newblock Sharp resolvent estimate for the damped-wave {Baouendi-Grushin} operator and applications.
\newblock {\em Comm. Math. Phys.}, 400(1):541-637, 2023.

\bibitem{BC17}
K.~Beauchard and P.~Cannarsa.
\newblock Heat equation on the {Heisenberg} group: observability and
applications.
\newblock {\em J. Differential Equations}, 262(8):4475-4521, 2017.

\bibitem{BCG14}
K.~Beauchard, P.~Cannarsa, and R.~Guglielmi.
\newblock Null controllability of {Grushin}-type operators in dimension two.
\newblock {\em J. Eur. Math. Soc.}, 16(1):67-101, 2014.

\bibitem{BDE20}
K.~Beauchard, J.~Dard{\'e}, and S.~Ervedoza.
\newblock Minimal time issues for the observability of {Grushin}-type
equations.
\newblock {\em Ann. Inst. Fourier}, 70(1):247-312, 2020.

\bibitem{BMM15}
K.~Beauchard, L.~Miller, and M.~Morancey.
\newblock {2D Grushin}-type equations: minimal time and null controllable data.
\newblock {\em J. Differential Equations}, 259(11):5813-5845, 2015.

\bibitem{BR96}
A.~Bella{\"i}che and J.-J.~Risler, editors.
\newblock {\em Sub-Riemannian Geometry}.
\newblock Number~144 in Progress in Mathematics. Birkh{\"a}user, Basel, 1996.

\bibitem{BR20}
Y.~Bonthonneau and N.~Raymond.
\newblock {WKB} constructions in bidimensional magnetic wells.
\newblock {\em Math. Res. Lett.}, 27(3):647-663, 2020.

\bibitem{BL13}
U.~Boscain and C.~Laurent.
\newblock The {Laplace-Beltrami} operator in almost-{Riemannian} geometry.
\newblock {\em Ann. Inst. Fourier}, 63(5):1739-1770, 2013.

\bibitem{BS22}
N.~Burq and C.~Sun.
\newblock Time optimal observability for {Grushin} {Schr{\"o}dinger} equation.
\newblock {\em Anal. PDE}, 15(6):1487-1530, 2022.

\bibitem{CHT22}
Y.~Colin~de~Verdi{\`e}re, L.~Hillairet, and E.~Tr{\'e}lat.
\newblock Spectral asymptotics for sub-{Riemannian} {Laplacians}.
\newblock \href{https://arxiv.org/abs/2212.02920}{Preprint, arXiv:2212.02920}, 2022.

\bibitem{DKR23}
J.~Dard{\'e}, A.~Koenig, and J.~Royer.
\newblock Null-controllability properties of the generalized two-dimensional
{Baouendi-Grushin} equation with non-rectangular control sets.
\newblock {\em Ann. H. Lebesgue}, 6:1479-1522, 2023.

\bibitem{Dav89}
E.~B. Davies.
\newblock {\em Heat Kernels and Spectral Theory}.
\newblock Number~92 in Cambridge Tracts in Mathematics. Cambridge University
Press, Cambridge, 1989.

\bibitem{Del23}
A.~Deleporte.
\newblock {WKB} eigenmode construction for analytic {Toeplitz} operators.
\newblock {\em Pure Appl. Anal.}, 5(2):213-260, 2023.

\bibitem{DV21}
A.~Deleporte and S.~V{\~u}~Ng{\d{o}}c.
\newblock Uniform spectral asymptotics for semiclassical wells on phase space
loops.
\newblock {\em Indag. Math.}, 32(1):3-32, 2021.

\bibitem{DR77}
S.~Dolecki and D.~L. Russell.
\newblock A general theory of observation and control.
\newblock {\em SIAM J. Control Optim.}, 15(2):185-220, 1977.

\bibitem{DK20}
M.~Duprez and A.~Koenig.
\newblock Control of the {Grushin} equation: non-rectangular control region and
minimal time.
\newblock {\em ESAIM Control Optim. Calc. Var.}, 26:article~3, 2020.

\bibitem{FI96}
A.~V. Fursikov and O.~{\relax Yu}. Imanuvilov.
\newblock {\em Controllability of Evolution Equations}.
\newblock Number~34 in Lecture Notes Series. Research Institute of Mathematics,
Seoul National University, Seoul, 1996.

\bibitem{HS18}
M.~Hitrik and J.~Sj{\"o}strand.
\newblock Two minicourses on analytic microlocal analysis.
\newblock In {\em Algebraic and Analytic Microlocal Analysis}, volume 269 of
{\em Springer Proceedings in Mathematics \& Statistics}, pages 483-540. Springer, Cham, 2018.

\bibitem{H67}
L.~H{\"o}rmander.
\newblock Hypoelliptic second order differential equations.
\newblock {\em Acta Math.}, 119:147-171, 1967.

\bibitem{Im95}
O.~{\relax Yu}. Imanuvilov.
\newblock Controllability of parabolic equations.
\newblock {\em Sb. Math.}, 186(6):879-900, 1995.

\bibitem{Ko17}
A.~Koenig.
\newblock Non-null-controllability of the {Grushin} operator in {2D}.
\newblock {\em C. R. Math. Acad. Sci. Paris}, 355(12):1215-1235, 2017.

\bibitem{KLS22}
K.~Krupchyk, T.~Liimatainen, and M.~Salo.
\newblock Linearized {Calder{\'o}n} problem and exponentially accurate
quasimodes for analytic manifolds.
\newblock {\em Adv. Math.}, 403:article~108362, 2022.

\bibitem{LL19}
C.~Laurent and M.~L{\'e}autaud.
\newblock Quantitative unique continuation for operators with partially
analytic coefficients. {Application} to approximate control for waves.
\newblock {\em J. Eur. Math. Soc.}, 21(4):957-1069, 2019.

\bibitem{LL22}
C.~Laurent and M.~L{\'e}autaud.
\newblock Tunneling estimates and approximate controllability for hypoelliptic
equations.
\newblock {\em Mem. Amer. Math. Soc.}, 276(1357), 2022.
\newblock vi+95~pp.

\bibitem{LR95}
G.~Lebeau and L.~Robbiano.
\newblock Contr{\^o}le exact de l'{\'e}quation de la chaleur.
\newblock {\em Comm. Partial Differential Equations}, 20(1-2):335-356, 1995.

\bibitem{Let21}
C.~Letrouit.
\newblock Exact observability properties of subelliptic wave and
{Schr{\"o}dinger} equations.
\newblock {\em S{\'e}m. Th{\'e}or. Spectr. G{\'e}om.}, 36:51-102, 2019-2021.

\bibitem{Lis25}
P.~Lissy.
\newblock A non-controllability result for the half-heat equation on the whole line based on the prolate spheroidal wave functions and its application to the {Grushin} equation.
\newblock {\em J. Differential Equations}, 433:article 113306, 2025.

\bibitem{Mon02}
R.~Montgomery.
\newblock {\em A Tour of Subriemannian Geometries, Their Geodesics and Applications}.
\newblock Number~91 in Mathematical Surveys and Monographs. American
Mathematical Society, Providence, RI, 2002.

\bibitem{Sj82}
J.~Sj{\"o}strand.
\newblock {\em Singularit{\'e}s analytiques microlocales}.
\newblock Number~95 in Ast{\'e}risque. Soci{\'e}t{\'e} Math{\'e}matique de
France, Paris, 1982.
\newblock iii+166~pp.

\bibitem{Tam22}
C.~Tamekue.
\newblock Null controllability of the parabolic spherical {Grushin} equation.
\newblock {\em ESAIM Control Optim. Calc. Var.}, 28:article~70, 2022.

\bibitem{Tam26}
C.~Tamekue.
\newblock Minimal time for null controllability of the parabolic spherical
{Baouendi-Grushin} equation.
\newblock \href{https://arxiv.org/abs/2604.19700v1}{Preprint, arXiv:2604.19700v1}, 2026.

\bibitem{Van26}
R.~Vanlaere.
\newblock Non null-controllability properties of the {Grushin}-like heat
equation on {2D}-manifolds.
\newblock \href{https://arxiv.org/abs/2503.00997v3}{Preprint, arXiv:2503.00997v3}, 2026.

\end{thebibliography}
\end{document}